\documentclass[11pt,reqno]{amsart}

\usepackage{enumerate}
\usepackage{amsmath, amssymb, amsthm}
\usepackage{mathrsfs}
\usepackage{esint}
\usepackage{xcolor}
\usepackage{mathtools}
\usepackage{hyperref}
\usepackage{bm}

\newtheorem{thm}{Theorem}[section]
\newtheorem{corollary}[thm]{Corollary}
\newtheorem{lem}[thm]{Lemma}
\newtheorem{prop}[thm]{Proposition}
\theoremstyle{definition}
\newtheorem{defn}[thm]{Definition}
\newtheorem{example}[thm]{Example}
\newtheorem{assumption}[thm]{Assumption}
\theoremstyle{remark}
\newtheorem{rem}[thm]{Remark}

\newcommand\bR{\mathbb{R}}

\newcommand\bZ{\mathbb{Z}}

\newcommand\bS{\mathbb{S}}

\newcommand\bE{\mathbb{E}}
\newcommand\bN{\mathbb{N}}
\newcommand\bP{\mathbb{P}}

\newcommand\fR{\mathbb{R}}

\newcommand\cA{\mathcal{A}}
\newcommand\cB{\mathcal{B}}

\newcommand\cD{\mathcal{D}}
\newcommand\cF{\mathcal{F}}

\newcommand\cI{\mathcal{I}}

\newcommand\cL{\mathcal{L}}

\newcommand\cS{\mathcal{S}}
\newcommand\cM{\mathcal{M}}

\newcommand\cO{\mathcal{O}}

\newcommand\frL{\mathfrak{L}}

\newcommand{\p}{\partial}

\newcommand{\mysection}[1]{\section{#1}
\setcounter{equation}{0}}

\begin{document}

\title[A maximal $L_p$-regularity theory with time measurable nonlocal operators]{A maximal $L_p$-regularity theory to initial value problems with time measurable nonlocal operators generated by additive processes}

\author[J.-H. Choi]{Jae-Hwan Choi}
\address[J.-H. Choi]{Department of Mathematical Sciences, KAIST, 291 Daehak-ro, Yuseong-gu, Daejeon, 34141, Republic of Korea}
\email{jaehwanchoi@kaist.ac.kr}

\author[I. Kim]{Ildoo Kim}
\address[I. Kim]{Department of Mathematics, Korea University, 145 Anam-ro, Seongbuk-gu, Seoul, 02841, Republic of Korea}
\email{waldoo@korea.ac.kr}
\thanks{The authors have been supported by the National Research Foundation of Korea(NRF) grant funded by the Korea government(MSIT) (No.2020R1A2C1A01003959)}

\subjclass[2020]{35S10, 60H30, 47G20, 42B25}

\keywords{Nonlocal operators, Stochastic processes, Integro-differential equations, Littlewood-Payley theory, Maximal $L_p$ regularity.}

\maketitle
\begin{abstract}
Let $Z=(Z_t)_{t\geq0}$ be an additive process with a bounded triplet $(0,0,\Lambda_t)_{t\geq0}$. 
Then the infinitesimal generators of $Z$ is given by time dependent nonlocal operators as follows:
\begin{align*}
\cA_Z(t)u(t,x)
&=\lim_{h \downarrow 0}\frac{\bE[u(t,x+Z_{t+h}-Z_t)-u(t,x)]}{h} \\
&=\int_{\mathbb{R}^d}(u(t,x+y)-u(t,x)-y\cdot \nabla_x u(t,x)1_{|y|\leq 1})\Lambda_t(dy).
\end{align*}
Suppose that for any Schwartz function $\varphi$ on $\bR^d$ whose Fourier transform is in $C_c^{\infty}(B_{c_s} \setminus B_{c_s^{-1}} )$, 
there exist positive constants $N_0$, $N_1$, and $N_2$ such that
\begin{equation*}
    \int_{\bR^d}|\bE[\varphi(x+r^{-1}Z_t)]|dx\leq N_0 e^{- \frac{ N_1 t}{s(r)}},\quad \forall (r,t)\in(0,1)\times[0,T],
\end{equation*}
and
$$
\|\psi^{\mu}(r^{-1}D)\varphi\|_{L_1(\bR^d)}\leq  \frac{N_2}{s(r)},\quad \forall r\in(0,1).
$$
where  $s$ is a scaling function (Definition \ref{20.05.28.13.20}), $c_s$ is a positive constant related to $s$, $\mu$ is a symmetric L\'evy measure on $\bR^d$,
$\psi^{\mu}(r^{-1}D)\varphi(x)= \cF^{-1} \left[ \psi^{\mu}(r^{-1}\xi) \cF[\varphi]\right](x)$ and
$\psi^{\mu}(\xi):=\int_{\bR^d}(e^{iy\cdot\xi}-1-iy\cdot\xi 1_{|y|\leq 1})\mu(dy)$.
In particular, above assumptions hold for L\'evy measures $\Lambda_t$ having a nice lower bound  and $\mu$ satisfying a weak-scaling property (Propositions \ref{22.07.07.16.05}, \ref{22.07.07.16.05-2}, and \ref{22.07.07.16.05-3}). 
We emphasize that there is no regularity condition on L\'evy measures $\Lambda_t$  and they do not have to be symmetric.
In this paper, we establish the $L_p$-solvability to the initial value problem
\begin{equation}
\label{20.07.15.17.02}
\frac{\p u}{\p t}(t,x)=\cA_Z(t)u(t,x),\quad u(0,\cdot)=u_0,\quad (t,x)\in(0,T)\times\mathbb{R}^d,
\end{equation}
where $u_0$ is contained in a scaled Besov space  $B_{p,q}^{s;\gamma-\frac{2}{q}}(\mathbb{R}^d)$ (see Definition \ref{20.08.20.17.26}) with a scaling  function $s$, exponent $p \in (1,\infty)$, $q\in[1,\infty)$, and order $\gamma \in [0,\infty)$.
We show that equation \eqref{20.07.15.17.02} is uniquely solvable and the solution $u$ obtains full-regularity gain from the diffusion generated by a stochastic process $Z$. In other words, there exists a unique solution $u$ to equation \eqref{20.07.15.17.02} in $L_q((0,T);H_p^{\mu;\gamma}(\mathbb{R}^d))$, where $H_p^{\mu;\gamma}(\mathbb{R}^d)$ is a generalized Bessel potential space (see Definition \ref{20.05.31.15.20}). Moreover, the solution $u$ satisfies
$$
\|u\|_{L_q((0,T);H_p^{\mu;\gamma}(\mathbb{R}^d))}\leq N\|u_0\|_{B_{p,q}^{s;\gamma-\frac{2}{q}}(\mathbb{R}^d)},
$$
where $N$ is independent of $u$ and $u_0$. 
We finally remark that our operators $\cA_{Z}(t)$ include logarithmic operators such as $-a(t)\log(1-\Delta)$ (Corollary \ref{cor 20220713 01}) and operators 
 whose symbols are non-smooth such as $-\sum_{j=1}^dc_j(t)(-\Delta)^{\alpha/2}_{x^j}$ (Corollary \ref{20.07.08.15.05}). 
\end{abstract}


\mysection{Introduction}
The diffusion processes and their infinitesimal generators  containing various nonlocal operators are used to describe natural phenomena. For example, $\alpha$-stable processes and the fractional Laplacian operators describe the motion of particles in inhomogeneous environments (see e.g. \cite{brockmann2003levy, garbaczewski2010levy}). A time inhomogeneous L\'evy process which is one of the kinds of additive process and its infinitesimal generators  are used in mathematical finance to predict behavior of time varying financial objects (see e.g. \cite{kluge2005time,koval2005time, vadori2019inhomogeneous}). Accordingly, there has been growing interest about time inhomogeneous diffusion processes and their infinitesimal generators such as time dependent nonlocal operators. However, there are not many results handling the solvability of equations with these operators in a mathematically rigorous level.

In this paper, we mainly discuss the $L_p$-solvability of Partial Differential Equations (PDEs) with nonlocal operators, and introduce related preceding results first. For elliptic and parabolic PDEs with nonlocal operators, see e.g. \cite{dong2012lp,kim2015holder,kim2016lp,kim2016lplq,kim2018lp,kim2019lp,kim2021lq,mikulevivcius1992,mikulevivcius2014,mikulevivcius2017p,mikulevivcius2019cauchy,zhang2013maximal,zhang2013lp}. Furthermore, for parabolic stochastic PDEs with nonlocal operators, see e.g. \cite{chang2012stochastic,kim2013parabolic,gyongy2021lp,kim2012lp,mikulevicius2020cauchy}.
We should also mention that there are many considerable regularity results with nonlocal operators (Harnack inequality, H\"older estimates, ABP type estimates and Littlewood-Paley type inequality etc.). See e.g.
\cite{barles2011holder,bass2002harnack,bass2005harnack,bass2005holder,caffarelli2009regularity,dong2013schauder,guillen2012aleksandrov,kassmann2009priori,kim2012generalization,kim2014global,silvestre2006holder}.

Our main results are inspired by  \cite{kim2019lp,mikulevivcius1992,mikulevivcius2014,mikulevivcius2017p,mikulevivcius2019cauchy,zhang2013maximal,zhang2013lp}.
Our main tasks are to consider time varying generalizations and find general assumptions which are weaker than those given in all previous results. 
In \cite{zhang2013maximal}, the author considered the inhomogeneous problem in the classical Bessel-potential space with a nonlocal operator 
\begin{align}
						\label{2020081501}
\cL^{\nu}f(x)=\int_{\mathbb{R}^d}(f(x+y)-f(x)-y^{(\alpha)}\cdot \nabla f(x))\nu(dy),
\end{align}
where $y^{(\alpha)}:=y(1_{\alpha=1}1_{|y|\leq1}+1_{\alpha\in(1,2)})$, the L\'evy measure $\nu(dy)$ satisfying for $\alpha\in(0,2)$,
$$
\nu_1^{\alpha}(dy)\leq \nu(dy)\leq \nu_2^{\alpha}(dy),\quad 1_{\alpha=1}\int_{r<|y|\leq R}y\nu(dy)=0,\quad \forall 0<r<R<\infty,
$$
and $\nu_i^{\alpha}$ ($i=1,2$) are L\'evy measures of two $\alpha$-stable processes taking the form
$$
\nu_i^{\alpha}(B):=\int_{\bS^{d-1}}\int_{0}^{\infty}\frac{1_B(r\theta)}{r^{1+\alpha}}dr\Sigma_i(d\theta).
$$
Here $\bS^{d-1}$ is the unit sphere centered at the origin in $\mathbb{R}^d$ and $\Sigma_i$ is a non-degenerate finite measure on $\bS^{d-1}$. 
The operator $\cL^{\nu}$ given in \eqref{2020081501} has a nice scaling property since the two L\'evy measures $\nu^\alpha_1$ and $\nu^\alpha_2$ satisfy
$$
\nu_i^\alpha(cB) = c^{-\alpha} \nu_i^\alpha(B) \qquad (i=1,2).
$$
This result is obtained on the basis of the scaling property. However, it is not expected that nonlocal operators have scaling properties in general. Thus it is natural to consider problems with non-scalable or weak-scalable operators.  Note that to discuss an optimal regularity of solutions to equation with these operators, a variant of classical Bessel-potential spaces are required since the classical Bessel-potential space only fit to strongly scalable operators. Similarly, the classical Besov space is not appropriate to handle the initial data from a point of view of optimal regularity. 
Recently, there appeared to be a few results handling weak-scaling nonlocal operators in variants of Bessel-potential spaces and Besov spaces. For instance, in \cite{mikulevivcius2017p,mikulevivcius2019cauchy}, authors considered the solvability of the problem with nonlocal operators in different frequency scaled Bessel-potential space (so called $\psi$-potential space) and Besov space. 
More precisely, under certain conditions on L\'evy measures so called assumptions $A_0(\sigma)$, $\textbf{D}(\kappa,l)$ and $\textbf{B}(\kappa,l)$
(see Remark \ref{20.06.29.15.35}), they obtained the maximal $L_p$-regularity result to the Cauchy problem
\begin{equation*}
\begin{cases}
\frac{\p u}{\p t}(t,x)=Lu(t,x)-\lambda u(t,x)+f(t,x),\quad &(t,x)\in(0,T)\times\mathbb{R}^d,\\
u(0,x)=u_0(x),\quad & x\in\mathbb{R}^d,
\end{cases}
\end{equation*}
where $L$ is a nonlocal operator of the form of \eqref{2020081501}, and $f$ and $u_0$ are in a $\psi$-Bessel potential space and a scaled Besov space, respectively. However, they did not consider time-dependent nonlocal operators.

There are also some results handling time dependent weak-scaling nonlocal operators. One of the simplest form of time dependent nonlocal operators is a stable-like operator, defined by
\begin{equation*}
    L(r,z)f(x):=\int_{\mathbb{R}^d}(f(x+y)-f(x)-y^{(\alpha)}\cdot\nabla f(x))a(r,z,y)|y|^{-d-\alpha}dy,
\end{equation*}
where $\alpha\in(0,2)$, $a$ is a nonnegative measurable function and
$y^{(\alpha)}:=y(1_{\alpha=1}1_{|y|\leq1}+1_{\alpha\in(1,2)})$.
One can consider the related Cauchy problem
\begin{equation*}
\begin{cases}
\frac{\p u}{\p t}(t,x)=L(t,x)u(t,x)+f(t,x),\quad &(t,x)\in(0,T)\times\mathbb{R}^d,\\
u(0,x)=u_0(x),\quad & x\in\mathbb{R}^d.
\end{cases}
\end{equation*}
If $a$ is H\"older continuous in $x$, smooth and $0$-homogeneous in $y$, then $L_p$-estimates for zero initial problem were obtained in \cite{mikulevivcius1992}. In \cite{mikulevivcius2014}, the authors improved the result in \cite{mikulevivcius1992} by assuming
$$
a_0(t,x,y)\leq a(t,x,y)\leq K,\quad \forall (t,x,y)\in[0,\infty)\times\mathbb{R}^d\times\mathbb{R}^d.
$$
Here, $a_0$ is smooth, homogeneous in $y$ and nondegenerate. In \cite{zhang2013lp}, the author provided $L_p$-solvability results for bounded and Dini continuous with respect to $x,y$ by perturbing the result in \cite{zhang2013maximal}.

In \cite{kim2019lp}, the authors deal with time-varying nonlocal operators beyond the stable-like operators and they also provided the $L_p$-solvability of the non-homogeneous problem
\begin{equation*}
\begin{cases}
\frac{\p u}{\p t}(t,x)=\cA_X(t)u(t,x)+f(t,x),\quad &(t,x)\in(0,T)\times\mathbb{R}^d,\\
u(0,x)=0,\quad & x\in\mathbb{R}^d,
\end{cases}
\end{equation*}
where 
$$
\cA_X(t)\varphi(x):=\lim_{h\downarrow0}\frac{\bE[\varphi(x+X_{t+h}-X_t)-\varphi(x)]}{h},
$$and $X$ is an additive process satisfying certain assumptions. Note that the infinitesimal generators of the stochastic process $X_t$, $\cA_X(t)$, contain the operators $L$ of the form of \eqref{2020081501} (see Lemma \ref{20.04.11.14.30}).
However, they only considered zero initial conditions and naturally initiated us to consider non-zero initial value problems.

The aim of this article is to investigate the maximal $L_p$-regularity of solutions to the homogeneous Initial Value Problems (IVPs)
\begin{equation}
\label{20.06.18.17.48}
\begin{cases}
\frac{\p u}{\p t}(t,x)=\cA_Z(t)u(t,x),\quad &(t,x)\in(0,T)\times\mathbb{R}^d,\\
u(0,x)=u_0(x),\quad & x\in\mathbb{R}^d,
\end{cases}
\end{equation}
where $Z$ is an additive process with (locally) bounded triplet $(a(t),0,\Lambda_t)_{t\geq0}$ (see Definition \ref{20.05.01.15.17}).
We emphasize that operators depend on time without any regularity condition with respect to the time variable and L\'evy measures are not symmetric in general. Moreover, our assumptions  are weaker than all previous results mentioned above. To the best of our knowledge, we assert that our assumptions are the weakest conditions in $L_p$-maximal regularity theories for nonlocal operators even though time independent operators are considered. 
Here is a list of our assumptions given on  a locally bounded triplet $(a(t),0,\Lambda_t)_{t\geq0}$.
For an additive process  $Z$ with a bounded triplet $(a(t),0,\Lambda_t)_{t\geq0}$ and a symmetric L\'evy measure $\mu$ on $\bR^d$, 
we suppose that for any $\varphi \in  \cS(\bR^d)$ so that $\cF[\varphi]\in C_c^{\infty}(B_{c_s} \setminus B_{c_s^{-1}} )$, 
there exist positive constants $N_0$, $N_1$, and $N_2$ such that
\begin{equation}
							\label{20.07.14.14.24}
    \int_{\bR^d}|\bE[\varphi(x+r^{-1}Z_t)]|dx\leq N_0e^{-N_1s(r)^{-1}t},\quad \forall (r,t)\in(0,1)\times[0,T],
\end{equation}
and
\begin{align}
					\label{20.07.14.16.18}
\|\psi^{\mu}(r^{-1}D)\varphi\|_{L_1(\bR^d)}\leq N_2s(r)^{-1},\quad \forall r\in(0,1).
\end{align}
where  $s$ is a scaling function (Definition \ref{20.05.28.13.20}),
$$
\psi^{\mu}(r^{-1}D)\varphi(x)= \cF^{-1} \left[ \psi^{\mu}(r^{-1}\xi) \cF[\varphi]\right](x)
$$
and
$$
\psi^{\mu}(\xi)
:=\int_{\bR^d}(e^{iy\cdot\xi}-1-iy\cdot\xi 1_{|y|\leq 1})\mu(dy).
$$
It is an interesting observation that our assumptions hold for L\'evy measures $\Lambda_t$ having a lower bound and $\mu$ satisfying a weak-scaling property (Propositions \ref{22.07.07.16.05}, \ref{22.07.07.16.05-2}, and \ref{22.07.07.16.05-3}). 
More precisely, it suffices to show that there exist a  L\'evy measure $\nu$ and positive constants $N_\nu$ and $c$  such that  for nonnegative Borel measurable function $f$,
\begin{equation}
						\label{20.07.14.14.24-2}
   s(r)\int_{\mathbb{R}^d}f(y)\Lambda_t(r\,dy):=s(r)\int_{\bR^d}f\left(\frac{y}{r}\right)\Lambda_t(dy)\geq\int_{\bR^d}f(y)\nu(dy),\quad \forall r>0, 
\end{equation}
\begin{equation}
						\label{20.07.14.14.24-2-2}
    \inf_{|\xi|=1}\int_{|y|\leq N_\nu }|y\cdot\xi|^2\nu(dy) >0,
\end{equation}
or
\begin{align}
											\notag
&\int_{\bR^d}f\left(y\right)\Lambda_t(dy)\geq\int_{\bR^d}f(y)\nu(dy) \\
											\label{eqn 20220716 01}
&    \inf_{r \in (0, \infty),|\xi|=c} s(r)\int_{|y|\leq N_1}(1-\cos(y \cdot \xi))\nu(r dy) >0, \\
											\label{eqn 20220716 02}
&\sup_{r \in (0, \infty)} \left[ s(r)\int_{\bR^d}(1\wedge|y|^2)\nu(r\,dy) \right] < \infty,
\end{align}
Moreover, it suffices to find a symmetric L\'evy measure $\mu$ satisfying
\begin{equation}
												\label{20.07.14.16.18-2}
  \sup_{r>0}s(r)\int_{\bR^d}\min\{1,|r^{-1}y|^2\}\mu(dy)<\infty.
\end{equation}
More specifically, \eqref{20.07.14.14.24-2}-\eqref{20.07.14.14.24-2-2} and \eqref{eqn 20220716 01}-\eqref{eqn 20220716 02} are  sufficient conditions of \eqref{20.07.14.14.24}. On the other hand, \eqref{20.07.14.16.18-2} is a sufficient condition of \eqref{20.07.14.16.18} and 
\eqref{eqn 20220716 02} also can be a sufficient condition of \eqref{20.07.14.16.18} with 
$$
\mu(dy)=\nu_{sym}(dy)= \frac{\left( \nu(dy) + \nu(-dy) \right)}{2}.
$$
Our main result Theorem \ref{main1} says that, under these assumptions \eqref{20.07.14.14.24} and \eqref{20.07.14.16.18}, for any given $\gamma\in[0,\infty)$, $T\in(0,\infty)$, and
$$
u_0\in B_{p,q}^{s;\gamma-\frac{2}{q}}(\bR^d),\quad p\in(1,\infty),\quad q\in[1,\infty),
$$
there exists a unique solution $u\in L_q((0,T);H_p^{\mu;\gamma}(\bR^d))$ to IVP \eqref{20.06.18.17.48} with the estimate
\begin{equation}
\label{20.06.20.20.54}
\begin{aligned}
\|u\|_{L_q((0,T);H_p^{\mu;\gamma}(\bR^d))}&\approx\left(\int_{0}^T\|u(t,\cdot)\|_{L_p(\bR^d)}^qdt+\int_{0}^T\|(-\psi^{\mu}(D))^{\gamma/2}u(t,\cdot)\|_{L_p(\bR^d)}^qdt\right)^{1/q}\\
&\leq N \|u_0\|_{B_{p,q}^{s;\gamma-\frac{2}{q}}(\bR^d)},
\end{aligned}
\end{equation}
where
$N$ is independent of $u$ and $u_0$.
Here $\mu$ is a symmetric L\'evy measure on $\bR^d$,
$$
\psi^{\mu}(\xi):=\int_{\bR^d}(e^{iy\cdot\xi}-1-iy\cdot\xi 1_{|y|\leq 1})\mu(dy),
$$
and $(-\psi^{\mu}(D))^{\gamma/2}$ is a pseudo-differential operator with the symbol $(-\psi^{\mu}(\xi))^{\gamma/2}$.

The main tools used in the proof of our main theorems are  probabilistic and analytic representations of a solution  and a modified version of the Littlewood-Paley theory. First, in Theorem \ref{20.05.04.10.58}, we prove that for a smooth initial data $u_0$, IVP \eqref{20.06.18.17.48} has a unique strong solution $u$, and $u$ can be represented by
\begin{equation}
\label{20.06.20.17.40}
    u(t,x)=\bE[u_0(x+Z_t)],
\end{equation}
Note that the symbol of $Z$ and the operator $\psi^{\mu}(D)$ are not scalable in general.
Thus one cannot adapt the classical Littlewood-Paley theory to \eqref{20.06.20.17.40}.
 However, \eqref{20.07.14.14.24} allows us to apply a modified version of a Littlewood-Paley theory.
Additionally, \eqref{20.07.14.16.18} enables us to control the lack of scalability of the operator $\psi^{\mu}(D)$, that is, this assumption gives a weak-scaling property on our main operators. 
With the help of these two facts, by applying a modified version of Littlewood-Paley theory to \eqref{20.06.20.17.40}, we finally obtain \eqref{20.06.20.20.54}.

The reason why we only deal with the homogeneous IVPs is because we figure out that assumptions can be considerably weaken if there is no inhomogeneous term in the equation. In particular, we found out that Assumptions $A_0(\sigma)$, $\textbf{D}(\kappa,l)$ and $\textbf{B}(\kappa,l)$ in \cite{mikulevivcius2017p,mikulevivcius2019cauchy} are strongly given to ensure that the kernel satisfies H\"ormander's condition. 
In other words, if we only consider the homogeneous case, then the kernel does not have to satisfy H\"ormander's condition and thus assumptions in previous results can be strongly weaken (see Remark \ref{20.06.29.15.35} for more details). Due to these weakened assumptions, our operators contain a certain pseudo-differential operator $\cA_Z(t)$ whose symbol has a logarithmic growth.
Consider the operator
$$
\cA_Z(t) = -a(t)\log (1-\Delta).
$$
Then the family of L\'evy measures is given by 
$$
\Lambda_t(dy):=a(t)\mu(dy),\quad 0<C^{-1}\leq a(t)\leq C,\quad \forall t \in[0,\infty),
$$
where     
$$
\mu(dy):=J(|y|)dy,\quad J(|y|):=\int_0^{\infty}(4\pi t)^{-d/2}e^{-\frac{|y|^2}{4t}}\frac{e^{-t}}{t}dt.
$$
Then our main assumptions holds for the scaling function 
$$
s(x) := \frac{1}{\log(1+x^{-2})}.
$$
For more details, see Example \ref{exam loga}.

Another interesting example is coordinate-wise fractional Laplacian operators with time measurable coefficients $-\sum_{j=1}^dc_j(t)(-\Delta)^{\alpha/2}_{x^j}$ whose
symbols are not smooth. Consider the one parameter family of L\'evy measures $(\Lambda_t)_{t\geq0}$ given by
\begin{equation*}
\begin{gathered}
\Lambda_{t}(dy):=\sum_{j=1}^d\Lambda_{j,t}(dy),\quad 
\Lambda_{j,t}(dy):=\frac{c_j(t,y)}{|y^j|^{1+\alpha}}dy^j\epsilon_0(dy^1,\cdots,dy^{j-1},dy^{j+1},\cdots,dy^d),
\end{gathered}
\end{equation*}
where $\alpha\in(0,2)$, $c_j(t,y)$ are measurable functions on $[0,\infty)\times\bR^d$ satisfying
$$
0<C^{-1}\leq c_j(t,y) \leq C,\quad \forall (j,t,y)\in\{1,2,\cdots,d\}\times[0,\infty)\times\bR^d,
$$
and $\epsilon_0$ is the centered Dirac measure on $\bR^{d-1}$. If $c_j$ are independent of $y$, i.e. $c_j(t,y)=c_j(t)$ for all $j$, $t$, and $y$, then for any Schwartz function $\varphi$ on $\bR^d$,
\begin{align*}
\cA_Z(t)\varphi(x)
&=-\sum_{j=1}^d c_j(t)\frac{1}{(2\pi)^{d/2}}\int_{\bR^d}\left(|\xi^j|^{\alpha}\cF[\varphi](\xi)e^{i\xi\cdot x} \right)d\xi \\
&=-\sum_{j=1}^dc_j(t)(-\Delta)^{\alpha/2}_{x^j} \varphi(x).
\end{align*}
If we put
$$
\mu(dy):=\frac{1}{|y|^{d+\alpha}}dy,\quad s(r):=r^{\alpha},
$$
then it is easy to check that \eqref{20.07.14.14.24-2} and \eqref{20.07.14.16.18-2} hold. 
For more details, see Example \ref{20.07.05.16.22}.
Thanks to the associate editor and referees, we figured out that the second example,  the coordinate-wise fractional Laplacian operators with time measurable coefficients,  could be covered by the previous results in \cite{ mikulevivcius2017p,mikulevivcius2019cauchy} if all coefficients $c_j$ are independent of time $t$.
However, the first example, logarithmic operators, cannot be covered by the previous results even if coefficients do not depend on $t$.

This article is organized as follows. In Section \ref{20.06.19.15.59}, we provide our main results with definitions of function spaces and solutions.
Specific examples for our results are given in Section \ref{example section}.
In Section \ref{20.06.20.22.06}, the solvability results to the Cauchy problem for smooth data are presented. 
In Section \ref{20.08.20.17.29}, the properties of function spaces are introduced.
In Section \ref{20.08.20.17.31}, we prove a priori estimate and main Theorem \ref{main1}.
Finally, we give all details of proofs of propositions in the appendix.

We finish this section with the notations used in the article.
\begin{itemize}
\item 
Let $\bN$ and $\bZ$ denote the natural number system and the integer number system, respectively.
As usual $\fR^{d}$, $d \in \bN$, stands for the Euclidean space of points $x=(x^{1},...,x^{d})$.
 For $i=1,...,d$, a multi-index $\alpha=(\alpha_{1},...,\alpha_{d})$ with
$\alpha_{i}\in\{0,1,2,...\}$, and function $f$, we set
$$
f_{x^{i}}=\frac{\partial f}{\partial x^{i}}=D_{i}f,\quad
D^{\alpha}f=D_{1}^{\alpha_{1}}\cdot...\cdot D^{\alpha_{d}}_{d}f,\quad |\alpha|:=\sum_{i=1}^d\alpha_i
$$
For $\alpha_i =0$, we define $D^{\alpha_i}_i f = f$. 

\item 
The gradient of a function $f$ is denoted  by 
\[
\nabla f = (D_1f, D_2f, \cdots, D_df).
\]
where $D_{i}f = \frac{\partial f}{\partial x^{i}}$ for $i=1,...,d$.
\item 
For $\cO\subset\bR^d$, we denote by $C^\infty(\cO)$ the space of infinitely differentiable functions on $\cO$, denote by $C_c^\infty(\cO)$ the subspace of $C^\infty(\cO)$ with the compact support and denote by $\cD(\cO)$ the space of distributions.
Let $\cS(\bR^d)$ be the Schwartz space consisting of infinitely differentiable and rapidly decreasing functions on $\bR^d$ and
$\cS'(\bR^d)$ be the space of tempered distributions on $\bR^d$. We say that $f_n \to f$ in $S(\bR^d)$  as $n \to \infty$ if for all multi-indices $\alpha$ and $\beta$ 
\begin{align}
							\label{sch conver}
\sup_{x \in \bR^d} |x^\alpha (D^\beta (f_n-f))(x)| \to 0 \quad \text{as}~ n \to \infty.
\end{align}
\item 
For $p \in [1,\infty)$, a normed space $F$, and a  measure space $(X,\mathcal{M},\mu)$, we denote by $L_{p}(X,\cM,\mu;F)$ the space of all $\mathcal{M}^{\mu}$-measurable functions $u : X \to F$ with the norm 
\[
\left\Vert u\right\Vert _{L_{p}(X,\cM,\mu;F)}:=\left(\int_{X}\left\Vert u(x)\right\Vert _{F}^{p}\mu(dx)\right)^{1/p}<\infty
\]
where $\mathcal{M}^{\mu}$ denotes the completion of $\cM$ with respect to the measure $\mu$. We also denote by $L_{\infty}(X,\cM,\mu;F)$ the space of all $\mathcal{M}^{\mu}$-measurable functions $u : X \to F$ with the norm
$$
\|u\|_{L_{\infty}(X,\cM,\mu;F)}:=\inf\left\{r\geq0 : \mu(\{x\in X:\|u(x)\|_F\geq r\})=0\right\}<\infty.
$$
If there is no confusion for the given measure and $\sigma$-algebra, we usually omit them.
\item 
For $\cO\subseteq \bR^d$, we denote by $\cB(\cO)$ the set of all Borel sets contained in $\cO$.
\item  
For $\cO\subset \fR^d$ and a normed space $F$, we denote by $C(\cO;F)$ the space of all $F$-valued continuous functions $u : \cO \to F$ with the norm 
\[
|u|_{C}:=\sup_{x\in O}|u(x)|_F<\infty.
\]
\item 
We denote the $d$-dimensional Fourier transform of $f$ by 
\[
\cF[f](\xi) := \frac{1}{(2\pi)^{d/2}}\int_{\bR^{d}} e^{- i\xi \cdot x} f(x) dx
\]
and the $d$-dimensional inverse Fourier transform of $f$ by 
\[
\cF^{-1}[f](x) := \frac{1}{(2\pi)^{d/2}}\int_{\bR^{d}} e^{  ix \cdot \xi} f(\xi) d\xi.
\]
\item 
If we write $N=N(a,b,\cdots)$, this means that the constant $N$ depends only on $a,b,\cdots$. 
\item 
For $a,b\in \bR$,
$$
a \wedge b := \min\{a,b\},\quad a \vee b := \max\{a,b\},\quad \lfloor a \rfloor:=\max\{n\in\bZ: n\leq a\}.
$$
\item
For a nonnegative Borel measure $\nu$ and a constant $c>0$, we denote
\begin{equation*}
\begin{gathered}
\tilde{\nu}(dy)=\frac{1}{2}(\nu(dy)+\nu(-dy)),\\
\nu^{c}(dy):=\nu(c\,dy),~ \text{i.e.} ~ \nu^c(A) = \nu(cA) \quad \text{for all Borel sets}~ A.
\end{gathered}
\end{equation*}
\end{itemize}
\mysection{Main results}
\label{20.06.19.15.59}
Fix a complete probability space $(\Omega, \cF, \bP)$ with a filtration $(\cF_t)_{t\geq0}$ satisfying the usual conditions. 
Denote the expectation with respect to $\bP$ by $\bE$. Every stochastic process considered in this article is $\cF_t$-adapted, $\bR^d$-valued, and zero at the initial time.

Recall that a nonnegative Borel measure $\mu$ on $\bR^d$ is called a L\'evy measure if
$$
\mu(\{0\})=0\quad\text{and}\quad L(\mu):=\int_{\bR^d}(1\wedge|y|^2)\mu(dy)<\infty.
$$
We denote by $\frL_d$ the set of all L\'evy measures on $\bR^d$. For a L\'evy process $X$, by the L\'evy-Khintchine representation (e.g. \cite[Theorem 8.1]{sato1999levy}), there exists a L\'evy triplet $(a,A,\mu)$ such that
\begin{equation*}
\begin{gathered}
\bE[e^{i\xi\cdot X_t}]=e^{t\Psi_X(\xi)},\quad\Psi_X(\xi):=ia\cdot\xi-\frac{1}{2}(A\xi\cdot\xi)+\int_{\bR^d}(e^{iy\cdot\xi}-1-iy\cdot\xi1_{|y|\leq1})\mu(dy),
\end{gathered}
\end{equation*}
where $t\in[0,\infty)$, $\xi\in\bR^d$, $a\in\bR^d$, $\mu\in\frL_d$, and $A$ is a nonnegative definite $d\times d$ matrix. 
It is well-known that the infinitesimal generators of a L\'evy process,
$$
\cA_{X_t} f(x)
=\lim_{h\downarrow0}\frac{\bE[f(x+X_{t+h}-X_t)-f(x)]}{h},
$$
 is independent of the time variable since the distributions of the increments of a L\'evy process are stationary. On the other hand, the distributions of the increments of additive processes are not stationary. Therefore, the infinitesimal generators of the additive processes are time dependent in general.


\begin{defn}
\label{20.05.01.15.17}
We say that $Z=(Z_t)_{t \geq 0}$ is an additive process with a (locally in time) bounded triplet $(a(t),A(t),\Lambda_t)_{t\geq0}$ if and only if
\begin{enumerate}[(i)]
    \item $Z$ is a stochastic process with independent increments with respect to the filtration $(\cF_t)_{t\geq0}$, i.e. for all $t>s \geq0$, $Z_t - Z_s$ and $\cF_s$ are independent. 
    \item $a(t):[0,\infty)\to\bR^d$ is $\cB([0,\infty))$-measurable.
    \item $A(t)$ is a nonnegative definite $d\times d$ matrix-valued $\cB([0,\infty))$-measurable function.
    \item $(\Lambda_t)_{t\geq0}$ is
    one parameter measurable family of L\'evy measures on $\bR^d$, i.e. $\Lambda_t \in \frL_d$ for all $t \geq 0$, and the mapping
$$
t \mapsto \int_{\bR^d} f(y) \Lambda_t(dy) ~\text{is measurable for any  integrable function $f$}.
$$
In particular, 
    $$
    t\mapsto L(t;\Lambda):=\int_{\bR^d}(1\wedge|y|^2)\Lambda_t(dy)\quad \text{and} \quad(t,\xi)\mapsto\int_{\bR^d}(e^{iy\cdot\xi}-1-iy\cdot\xi1_{|y|\leq1})\Lambda_{t}(dy)
    $$
    are $\cB([0,\infty))$ and $\cB([0,\infty)\times\bR^d)$-measurable, respectively.
    \item For $0\leq s \leq t$ and $\xi\in\bR^d$,
\begin{equation}
\label{22.07.12.13.26}
\begin{gathered}
\bE[e^{i\xi\cdot(Z_t-Z_s)}]=e^{\int_{s}^t\Psi_Z(r,\xi)dr},\\
\Psi_Z(r,\xi):= ia(r)\cdot\xi-\frac{1}{2}(A(r)\xi\cdot\xi)+\int_{\bR^d}(e^{iy\cdot\xi}-1-iy\cdot\xi1_{|y|\leq1})\Lambda_{r}(dy).
\end{gathered}
\end{equation}
The function $\Psi_Z(r,\xi)$ is called the characteristic exponent of an additive process $Z$. 
\item For each $T\in(0,\infty)$, 
$$
\|a\|_{L_{\infty}([0,T])}+\|A\|_{L_{\infty}([0,T])}+\|L(\cdot;\Lambda)\|_{L_{\infty}([0,T])}<\infty.
$$
\end{enumerate}
\end{defn}

\begin{rem}
\label{20.05.01.17.21}
It is easy to check that Definition \ref{20.05.01.15.17} includes L\'evy processes.
Indeed, if the triplet $(a(t),A(t),\Lambda_t)_{t\geq0}$ is independent of $t$, i.e. $(a(0),A(0),\Lambda_0) =(a(t),A(t),\Lambda_t)$ for all $t>0$, then $Z$ becomes a L\'evy process with the triplet $(a(0),A(0),\Lambda_0)$. 
There are many results characterizing additive processes as a generalization of L\'evy processes.  For details on properties of additive processes, we refer to \cite{sato1999levy}.
\end{rem}

Let $Z$ be an additive process with a bounded triplet $(a(t),0,\Lambda_t)_{t\geq0}$ and define
$$
\cA_Z(t)u(t,x):=\lim_{h\downarrow 0}\frac{\bE[u(t,x+Z_{t+h}-Z_t)-u(t,x)]}{h}.
$$
We consider the following Cauchy problem
\begin{equation}
\label{20.05.06.22.49}
\begin{cases}
\frac{\p u}{\p t}(t,x)=\cA_Z(t)u(t,x),\quad &(t,x)\in(0,T)\times\bR^d,\\
u(0,x)=u_0(x),\quad & x\in\bR^d.
\end{cases}
\end{equation}
We introduce definitions of function spaces handling initial data $u_0$ and solutions $u$ to equation \eqref{20.05.06.22.49}, first.
\begin{defn}[$\psi$-Bessel potential space]
\label{20.05.31.15.20}
For $p\in[1,\infty)$, and $\gamma\in\bR$, $H^{\psi;\gamma}_p(\bR^d)$ denote the class of all tempered distribution $f$ on $\bR^d$ such that
\begin{equation*}
\begin{gathered}
\|f\|_{H_p^{\psi;\gamma}(\bR^d)}:=\|(1-\psi(D))^{\gamma/2}f\|_{L_p(\bR^d)}:=\|\cF^{-1}[(1-\psi)^{\gamma/2}\cF[f]]\|_{L_p(\bR^d)}<\infty.
\end{gathered}
\end{equation*}
where $\psi$ is a real-valued locally integrable function defined on $\bR^d$.
In particular, if $\psi(\xi)=-|\xi|^2$, then $H^{\psi;\gamma}_p(\bR^d)$ becomes the classical Bessel potential space with the exponent $p$ and order $\gamma$. For this classical one, we use the simpler notation $H^{\gamma}_p(\bR^d)$.
Moreover, we can consider a $\psi$-potential spaces related to a symmetric L\'evy measure $\mu$ on $\bR^d$.
Denote 
$$
\psi^{\mu}(\xi)
:=\int_{\bR^d}(e^{iy\cdot\xi}-1-iy\cdot\xi 1_{|y|\leq 1})\mu(dy).
$$
In this case we prefer to use $\psi^{\mu}(\xi)$ instead of $\psi$ to emphasize the related L\'evy measure. 
Moreover, we simply write 
$$
H^{\psi^{\mu};\gamma}_p(\bR^d)
=H^{\mu;\gamma}_p(\bR^d).
$$
\end{defn}

It is a common way to classify L\'evy measures based on ratios of their characteristic exponents (cf. \cite{kim2012generalization,kim2014global,mikulevivcius2019cauchy}). If there is a control on ratios of characteristic exponents, we simply say that they satisfies a weak-scaling property. This type weak-scaling property can be classified by two factor functions called a upper bound and a lower bound functions. Here is a definition of scaling functions which control the L\'evy measures of additive processes. 
\begin{defn}
					\label{20.05.28.13.20}
For measurable functions $s,s_L,s_U,:(0,\infty)\to(0,\infty)$, we say that $\boldsymbol{s}:=(s,s_L,s_U)$ is a scaling triple  with $R_0$ if the triple satisfies followings :
\begin{itemize}
    \item $s_L$ and $s_U$ are non-decreasing functions with
$$
\lim_{r\to0}s_L(r)=\lim_{r\to0}s_U(r)=0,\quad \lim_{R\to\infty}s_L(R)=\lim_{R\to\infty}s_U(R)=\infty.
$$
    \item There exists $R_0\in[1,\infty]$ such that for all $0<r\leq R < R_0$,
\begin{equation}
\label{20.04.28.13.46}
    s_L\left(\frac{R}{r}\right)\leq \frac{s(R)}{s(r)} \leq s_U\left(\frac{R}{r}\right).
\end{equation}
 
\end{itemize}
Separately, $s$ is called a scaling function and $s_L$ (resp. $s_U$) is called a lower (resp. upper) scaling factor of $s$. For a scaling triple $\boldsymbol{s}=(s,s_L,s_U)$, denote 
\begin{align}
							\label{2020082401}
m_s:=\min\{m\in\bN:s_L(2^m)> 1\}~\text{and}~\quad c_s:=2^{m_s} \geq 2.
\end{align}
\end{defn}

\begin{rem}
\label{20.08.13.11.21}
Recall \cite[Definition 1]{mikulevivcius2017p} : A continuous function $\kappa:(0,\infty)\to(0,\infty)$ is a scaling function with scaling factor $l$ if 
 \begin{enumerate}[(i)]
     \item $\lim_{r\to0}\kappa(r)=0$ and $\lim_{R\to\infty}\kappa(R)=\infty$.
     \item $l$ is a nondecreasing continuous function such that $\lim_{\varepsilon\to0}l(\varepsilon)=0$ and
     $$
     \kappa(\varepsilon r)\leq l(\varepsilon)\kappa(r),\quad \forall \varepsilon,r>0.
     $$
 \end{enumerate}
We claim that \cite[Definition 1]{mikulevivcius2017p} implies Definition \ref{20.05.28.13.20} but the converse does not hold in general.
 Suppose that a scaling function $\kappa$ with scaling factor $l$ is given. One can observe that for $0< r \leq R$,
 $$
 \left(l\left(\frac{r}{R}\right)\right)^{-1}\leq \frac{\kappa(R)}{\kappa(r)}\leq l\left(\frac{R}{r}\right).
 $$
 Therefore, $\boldsymbol{\kappa}=(\kappa,\kappa_L,\kappa_U)$ is a scaling triple in the sense of Definition \ref{20.05.28.13.20} with $R_0=\infty$, where
 \begin{equation}
 \label{22.07.07.16.16}
     \kappa_L(r):=\left(l\left(r^{-1}\right)\right)^{-1},\quad \kappa_U(r):=l(r).
 \end{equation}
On the other hand, recall that $s$ can be discontinuous in a scaling triple $\boldsymbol{s}=(s,s_L,s_U)$ and $R_0$ can be a finite positive number. Therefore, our scaling function $s$ is not contained in a class of scaling functions in \cite[Definition 1]{mikulevivcius2017p}. 
 However for the special case $R_0 = \infty$, if there is the extra assumption that all $s$, $s_L$ and $s_U$ are  continuous, then $\kappa$ and $l$ can  be chosen by
$$
\kappa:=s,\quad l(\varepsilon):=\begin{cases}
\frac{1\vee (s_L(1)s_U(1))}{s_L(\varepsilon^{-1})},\quad &\varepsilon\geq1,\\
\frac{s_U(\varepsilon)}{1\wedge (s_L(1)s_U(1))},\quad & 0<\varepsilon<1.
\end{cases}
$$
Then it is obvious that 
$$
\kappa(\varepsilon r)\leq l(\varepsilon)\kappa(r),\quad \forall \varepsilon, r>0.
$$
\end{rem}
In the whole space, the Besov spaces  can be characterized completely by Littlewood-Paley functions (cf. \cite{bergh2012interpolation}).
Since we need a variant of the Besov space whose orders are given in different scales, we present different scaled Littlewood-Paley functions.
\begin{defn}[Littlewood-Paley function] 
									\label{LP function}
For an integer $n\geq2$, let $\Phi_n(\bR^d)$ denote the subset of $\cS(\bR^d)$ with the following properties ; $\zeta\in\Phi_n(\bR^d)$ if and only if
\begin{enumerate}[(i)]
\item $\cF[\zeta]>0$ on $\{\xi\in\bR^d:n^{-1}<|\xi|< n\}$.
\item $supp(\cF[\zeta])=\{\xi\in\bR^d:n^{-1}\leq|\xi|\leq n\}$.
\item For $\xi\in\bR^d\setminus\{0\}$,
$$
\sum_{j=-\infty}^{\infty}\cF[\zeta](n^{-j}\xi)=1.
$$
\end{enumerate}
For $\zeta\in\Phi_n(\bR^d)$, we put $\boldsymbol{\zeta}=(\zeta_0,\zeta_1,\cdots)$, where 
$$
\zeta_0:=\cF^{-1}\left[1-\sum_{j=1}^{\infty}\cF[\zeta](n^{-j}\cdot)\right]
$$
and $\zeta_j:=n^{jd}\zeta(n^j\cdot)$ for all $j\geq1$.
\end{defn}
The existence of Littlewood-Paley functions is guaranteed by \cite[Lemma 6.1.7]{bergh2012interpolation}.
\begin{rem}
\label{20.04.29.20.00}
Let $n\geq2$ be an integer and $\zeta,\iota\in\Phi_n(\bR^d)$. Then $\boldsymbol{\zeta}$ and $\boldsymbol{\iota}$ enjoy the almost orthogonality in the following sense:
\begin{equation}
\label{almost orthogonality}
\begin{gathered}
\cF[\zeta_0](\xi)=\cF[\zeta_0](\xi)\{\cF[\iota_0](\xi)+\cF[\iota](n^{-1}\xi)\}\\
\cF[\zeta](n^{-j}\xi)=\cF[\zeta](n^{-j}\xi)\{\cF[\iota](n^{-j+1}\xi)+\cF[\iota](n^{-j}\xi)+\cF[\iota](n^{-j-1}\xi)\},\quad j\geq1.
\end{gathered}
\end{equation}
\end{rem}

\begin{defn}[A scaled Besov space]
				\label{20.08.20.17.26}
For $p\in[1,\infty)$, $q\in(0,\infty)$, $\gamma\in\bR$, a scaling triple $\boldsymbol{s}=(s,s_L,s_U)$, and $\varphi\in\Phi_{c_s}(\bR^d)$, $B_{p,q}^{s,\varphi;\gamma}(\bR^d)$ denote the class of all tempered distribution $f$ on $\bR^d$ such that
$$
\|f\|_{B_{p,q}^{s,\varphi;\gamma}(\bR^d)}:=\|f\ast\varphi_0\|_{L_p(\bR^d)}+\left(\sum_{j=1}^{\infty}s(c_s^{-j})^{-\frac{\gamma q}{2}}\|f\ast\varphi_j\|_{L_p(\bR^d)}^q\right)^{1/q}<\infty.
$$
Moreover $B_{p,\infty}^{s.\varphi;\gamma}(\bR^d)$ denotes the class of all tempered distribution $f$ on $\bR^d$ such that
$$
\|f\|_{B_{p,\infty}^{s,\varphi;\gamma}(\bR^d)}:=\|f\ast\varphi_0\|_{L_p(\bR^d)}+\sup_{j\in\bN}s(c_s^{-j})^{-\frac{\gamma}{2}}\|f\ast\varphi_j\|_{L_p(\bR^d)}<\infty.
$$
In particular, if  $s(x)=x^2$ and $\varphi\in\Phi_{2}(\bR^d)$, then $B_{p,q}^{s,\varphi;\gamma}(\bR^d)$ becomes the classical Besov space and we use the simpler notation $B_{p,q}^{\gamma}(\bR^d)$.
Moreover,  the choice of $\varphi\in\Phi_{c_s}(\bR^d)$ is not crucial, which will be proved in Proposition \ref{20.04.29.20.24} $(iv)$. 
Hence, we sometimes write $B_{p,q}^{s;\gamma}(\bR^d)$ instead of $B_{p,q}^{s,\varphi;\gamma}(\bR^d)$ for the simplicity of the notation.
\end{defn}
Properties of the $\psi$-Bessel potential spaces and the scaled Besov spaces are given in Section \ref{20.08.20.17.29}.
In particular, we prove that  $C_c^{\infty}(\bR^d)$ is dense in both $H^{\mu;\gamma}_p(\bR^d)$ and $B_{p,q}^{s.\varphi;\gamma}(\bR^d)$.
We finally introduce the definition of a (strong) solution $u$ in these $\psi$-Bessel potential and scaled Besov spaces. 
\begin{defn}[Solution]
\label{20.06.06.15.50}
Let $T\in(0,\infty)$, $p\in[1,\infty)$, $q\in(0,\infty)$, $\gamma\in\bR$, $\mu$ be a symmetric L\'evy measure on $\bR^d$, $Z$ be an additive process with a bounded triplet $(a(t),0,\Lambda_t)_{t\geq0}$, and $\boldsymbol{s}=(s,s_L,s_U)$ be a scaling triple. For given 
$$f\in L_q((0,T);H_p^{\mu;\gamma}(\bR^d)),\quad u_0\in B_{p,q}^{s,\varphi;\gamma+2-\frac{2}{q}}(\bR^d),$$
we say that $u\in L_q((0,T);H_p^{\mu;\gamma+2}(\bR^d))$ is a solution to the Cauchy problem
\begin{equation*}
\begin{cases}
\frac{\p u}{\p t}(t,x)=\cA_Z(t)u(t,x)+f(t,x),\quad &(t,x)\in(0,T)\times\bR^d,\\
u(0,x)=u_0(x),\quad & x\in\bR^d,
\end{cases}
\end{equation*}
if there exists a sequence of smooth functions $u_n\in\tilde{C}^{1,\infty}([0,T]\times\bR^d)$ (see Definition \ref{20.05.17.14.08}) such that $u_n(0,\cdot)\in C_c^{\infty}(\bR^d)$,
\begin{equation*}
\begin{gathered}
\frac{\p u_n}{\p t}-\cA_Z(t)u_n\to f\quad\text{in}\quad L_q((0,T);H_p^{\mu;\gamma}(\bR^d)),\\
u_n(0,\cdot)\to u_0\quad\text{in}\quad B_{p,q}^{s,\varphi;\gamma+2-\frac{2}{q}}(\bR^d),
\end{gathered}
\end{equation*}
and
$$
u_n\to u \quad\text{in}\quad L_q((0,T);H_p^{\mu;\gamma+2}(\bR^d))
$$
as $n\to\infty$.
\end{defn}
Next we introduce our main assumptions given on an additive process $Z$, a L\'evy measure $\mu$, and a scaling function $s$. Denote
$$
B_r:=\{x\in\bR^d:|x|<r\}.
$$
\begin{assumption}
						\label{22.06.26.18.46}
Let  $Z$ be an additive process and $s$ be a scaling function.
We assume that for any $\varphi \in  \cS(\bR^d)$ so that $\cF[\varphi]\in C_c^{\infty}(B_{c_s} \setminus B_{c_s^{-1}} )$, there exist positive constants $N_0$ and $N_1$ such that 
\begin{equation}
\label{22.07.12.13.50}
    \int_{\bR^d}|\bE[\varphi(x+r^{-1}Z_t)]|dx\leq N_0e^{-\frac{N_1 t}{s(r)}},\quad \forall (r,t)\in(0,1)\times[0,T].
\end{equation}
\end{assumption}
\begin{assumption}
					\label{22.06.26.18.47}
Let $\mu$ be a a symmetric L\'evy measure on $\bR^d$  and $s$ be a scaling function.
We assume that for any $\varphi \in  \cS(\bR^d)$ so that $\cF[\varphi]\in C_c^\infty(B_{c_s} \setminus B_{c_s^{-1}} )$, there exists a positive constant $N_2$ such that 
$$
\|\psi^{\mu}(r^{-1}D)\varphi\|_{L_1(\bR^d)}\leq \frac{N_2}{s(r)},\quad \forall r\in(0,1).
$$
\end{assumption}
\begin{rem}
								\label{rem 2022071301}
\begin{enumerate}[(i)]
\item If an additive process $Z$ has a bounded triplet $(a(t),0,\Lambda_t)_{t\geq0}$, then Assumption \ref{22.06.26.18.46} can be written with respect to the characteristic exponent of $Z$.
Indeed for any $\varphi \in  \cS(\bR^d)$ so that $\cF[\varphi]\in C_c^\infty(B_{c_s} \setminus B_{c_s^{-1}}) $, 
by applying elementary properties of the Fourier transform, \eqref{22.07.12.13.26}, and Fubini's theorem,
we have
$$
\cF[\bE[\varphi(\cdot+r^{-1}Z_t)]](\xi)=\cF[\varphi](\xi)\bE[e^{i\xi\cdot r^{-1}Z_t}]=\cF[\varphi](\xi)\exp\left(\int_0^t\Psi_Z(s,r^{-1}\xi)ds\right).
$$
Therefore, Assumption \ref{22.06.26.18.46} becomes
\begin{align}
\label{22.07.12.13.49}
   \int_{\bR^d}\left|\cF^{-1}\left[\exp\left(\int_0^t\Psi_Z(s,r^{-1}\cdot)ds\right)\cF[\varphi]\right](x)\right|dx
   = \int_{\bR^d}|\bE[\varphi(x+r^{-1}Z_t)]|dx
   \leq N_0e^{-N_1s(r)^{-1}t}.
\end{align}
In our proof of main estimates, we use \eqref{22.07.12.13.49} instead of \eqref{22.07.12.13.50} since we handle additive processes with a bounded triplet.
Moreover, we show that if an additive process $Z$ has a nice lower bound $\nu$ with respect to $s$, then Assumption \ref{22.06.26.18.46} holds (see Propositions \ref{22.07.07.16.05} and \ref{22.07.07.16.05-3}).

\item An interesting sufficient condition of Assumption \ref{22.06.26.18.47} is a weak scaling property. 
If a symmetric L\'evy measure $\mu$ on $\bR^d$ satisfies a weak scaling property with respect to $s$, then Assumption \ref{22.06.26.18.47} holds (see Proposition \ref{22.07.07.16.05-2}).
\end{enumerate}
\end{rem}


Now, we state the main result of this article and the proof is given in Section \ref{20.08.20.17.31}.
\begin{thm}
\label{main1}
Let $s$ be a scaling function, $Z$ be an additive process with a bounded triplet $(a(t),0,\Lambda_t)_{t\geq0}$ satisfying Assumption \ref{22.06.26.18.46}, and $\mu$ be a symmetric L\'evy measure on $\bR^d$ satisfying Assumption \ref{22.06.26.18.47}. Then for all $p\in(1,\infty)$, $q\in[1,\infty)$, $T\in(0,\infty)$, $\gamma\in[0,\infty)$, and $u_0\in B_{p,q}^{s,\varphi;\gamma-\frac{2}{q}}(\bR^d)$, the initial value problem
\begin{equation*}
\begin{cases}
\frac{\p u}{\p t}(t,x)=\cA_Z(t)u(t,x),\quad &(t,x)\in(0,T)\times\bR^d,\\
u(0,x)=u_0(x),\quad & x\in\bR^d,
\end{cases}
\end{equation*}
has a unique solution $u\in L_q((0,T);H_p^{\mu;\gamma}(\bR^d))$ with
\begin{equation}
									\label{2020082301}
    \|u\|_{L_q((0,T);H_p^{\mu;\gamma}(\bR^d))}\leq N\|u_0\|_{B_{p,q}^{s,\varphi;\gamma-\frac{2}{q}}(\bR^d)},
\end{equation}
where $N$ is independent of $u$ and $u_0$.
\end{thm}

\begin{rem}
Note that ${B_{p,q}^{s,\varphi;\gamma-\frac{2}{q}}(\bR^d)}$ is a quasi Banach space and a priori estimate \eqref{2020082301} holds  (see Lemma  \ref{20.08.20.17.31}) even for $q \in (0,1)$ and $p \in [1,\infty)$.
However,  Calder\'on's complex interpolation does not work on ${B_{p,q}^{s,\varphi;\gamma-2/q}(\bR^d)}$ if $q \in (0,1)$ since this interpolation heavily depends on the duality property. Nevertheless, if there is the additional assumption that $\gamma$ is an even nonnegative integer, Theorem \ref{main1} still holds even though $q \in (0,1)$ and $p \in [1,\infty)$.
\end{rem}

\mysection{Examples}
								\label{example section}
In this section, we specify interesting examples which cannot be covered by previous results. 
First, we introduce a lower weak scaling property of the logarithmic function near zero.
For $ a \in (0,\infty)$, we use $\log a$ instead of $\log_2a$ for a simpler notation. 
\begin{lem}
\label{22.07.12.23.26}
For $0<r\leq R\leq1$,
\begin{equation}
\label{22.07.12.16.13}
 \log\left(1+\frac{R}{r}\right)\leq\frac{\log(1+R)}{\log(1+r)}.
\end{equation}
\end{lem}
\begin{proof}
Fix $ r \in (0,1]$ and let
$$
f_r(x):=\frac{\log(1+x)}{\log(1+r)}-\log\left(1+\frac{x}{r}\right).
$$
Then the derivative of $f_r$ is
\begin{align*}
\frac{df_r}{dx}(x)=\frac{1}{\log(1+r)(1+x)}-\frac{1}{r+x}\geq\frac{1}{r(1+x)}-\frac{1}{r+x}=\frac{(1-r)x}{r(1+x)(r+x)}\geq 0
\end{align*}
for all $x \geq 0$. 
Since $f_r(0)=0$, the lemma is proved.
\end{proof}

\begin{example}[The logarithmic function]
							\label{exam loga}
The first example is the logarithmic function
$$
g(\lambda):=\log(1+\lambda).
$$
These are well-known facts (cf. \cite{schbern}) on the logarithmic function $g$ :
\begin{enumerate}[(1)]
    \item The function $g$ is a Bernstein function, that is, $g:(0,\infty)\to(0,\infty)$ satisfies
    \begin{equation}
    \label{22.07.12.14.15}
        (-1)^n\frac{d^ng}{d\lambda^n}(\lambda)\leq0,\quad \forall \lambda>0,n\in\bN.
    \end{equation}
    \item By \eqref{22.07.12.14.15}, $g$ is concave on $[0,\infty)$, thus,
    \begin{equation}
    \label{22.07.12.23.47}
        \frac{g(R)}{g(r)}\leq\frac{R}{r},\quad 0<r\leq R<\infty.
    \end{equation}
    \item The function $g$ has a following integral representation ;
    \begin{equation}
    \label{22.07.06.13.28}
        g(\lambda)=\int_0^{\infty}(1-e^{-t\lambda})\frac{e^{-t}}{t}dt.
    \end{equation}
    \item There exists a nonnegative real-valued L\'evy process $S=(S_t)_{t\geq0}$ such that
    $$
    \bE[e^{-\lambda S_t}]=e^{-tg(\lambda)}.
    $$
    \item The $d$-dimensional subordinate Brownian motion $X_t:=W_{S_t}$ is a L\'evy process on $\bR^d$ with symbol
    $$
    \psi^{\mu}(\xi):=-g(|\xi|^2)=\int_{\bR^d}(e^{iy\cdot\xi}-1-iy\cdot\xi1_{|y|\leq1})\mu(dy)
    $$
    and its infinitesimal generator is
    \begin{align*}
        \cA_Xf(x)&=\lim_{t\downarrow0}\frac{\bE[f(x+X_t)]-f(x)}{t}=\cF^{-1}[-g(|\cdot|^2)\cF[f]](x)\\
        &=\int_{\bR^d}(f(x+y)-f(x)-y\cdot\nabla f(x)1_{|y|\leq1})\mu(dy),
    \end{align*}
    where $W=(W_t)_{t\geq0}$ is a $d$-dimensional Brownian motion independent of $S$ and
    $$
\mu(dy):=J(|y|)dy,\quad J(|y|):=\int_0^{\infty}(4\pi t)^{-d/2}e^{-\frac{|y|^2}{4t}}\frac{e^{-t}}{t}dt.
$$
    \item Using \eqref{22.07.06.13.28}, we have
    \begin{equation}
    \label{22.07.06.14.11}
        \lambda^n\left|\frac{d^ng}{d\lambda^n}(\lambda)\right|\leq N(n)g(\lambda),\quad \forall \lambda>0,n\in\bN.
    \end{equation}
\end{enumerate}
First we show that the operators $-g(-\Delta)$ cannot be treated in the previous results. 
Suppose that there exists a symmetric L\'evy measure $\mu_0$ such that $\mu(dy)\geq \mu_0(dy)$. Then
$$
\int_{\bR^d}|\xi|^4\exp(\psi^{\mu_0}(\xi))d\xi\geq \int_{\bR^d}|\xi|^4\exp(-g(|\xi|^2))d\xi=\int_{\bR^d}\frac{|\xi|^4}{1+|\xi|^2}d\xi=+\infty.
$$
This implies that Assumption $A_0(\sigma)$ in \cite{mikulevivcius2017p,mikulevivcius2019cauchy} cannot hold for the logarithmic function $g$. 
In \cite[Assumption 2.2]{kim2019lp}, the authors assume that there exist $\delta_3\geq \delta_2>0$ and $N_2,N_3>0$ such that
\begin{equation}
		\label{22.06.18.15.37}
N_2\left(\frac{\lambda_2}{\lambda_1}\right)^{\delta_2}\leq\frac{\phi(\lambda_2)}{\phi(\lambda_1)}\leq N_3\left(\frac{\lambda_2}{\lambda_1}\right)^{\delta_3},\quad \forall \lambda_2\geq\lambda_1>0,
\end{equation}
where $\phi$ denotes a characteristic exponent of a stochastic process. 
Moreover, an inequality similar to \eqref{22.06.18.15.37} is proved in \cite[Lemma 1]{mikulevicius2020cauchy}.
However, the left hand side of \eqref{22.06.18.15.37} cannot hold for the logarithmic function $g$. Therefore, the results in \cite{kim2019lp,mikulevivcius2017p,mikulevivcius2019cauchy,mikulevicius2020cauchy} cannot cover the logarithmic operators. For more detail information on the operator $-g(-\Delta)$, we refer the reader to \cite{feulefack2021logarithmic}.

Now, we verify that our main theorem (Theorem \ref{main1}) even works for this case. 
Put the scaling triple as follows:
$$
s(x):=\frac{1}{\log(1+x^{-2})},\quad s_L(x):=\log(1+x^2),\quad s_U(x):=x^2.
$$
Then by Lemma \ref{22.07.12.23.26} and \eqref{22.07.12.23.47},  
$\boldsymbol{s}=(s,s_L,s_U)$ becomes a scaling triple with $R_0=1$ and $c_s=2$. 
If we put
\begin{align}
								\label{dependent y}
\Lambda_t(dy):=a(t,y)\mu(dy),\quad 0<C^{-1}\leq a(t,y)\leq C,\quad \forall (t,y)\in[0,\infty)\times\bR^d,
\end{align}
then there exists an additive $Z$ process with a bounded triplet $(0,0,\Lambda_t)_{t\geq0}$. 
Next we prove that Assumption \ref{22.06.26.18.46} holds based on Remark \ref{rem 2022071301}(i).
Let $\varphi \in  \cS(\bR^d)$ so that $\cF[\varphi]\in C_c^\infty(B_{c_s} \setminus B_{c_s^{-1}} )$.
Using H\"older's inequality and Plancherel's theorem,
we have
\begin{align*}
    &\int_{\bR^d}\left|\cF^{-1}\left(\exp\left(\int_0^t\Psi_Z(h,r^{-1}\cdot)dh\right)\cF[\varphi]\right)(x)\right|dx\\
    &\leq\left(\int_{\bR^d}\frac{1}{(1+|x|^2)^{2d_0}}dx\right)^{1/2}\left(\int_{\bR^d}(1+|x|^2)^{2d_0}\left|\cF^{-1}\left(\exp\left(\int_0^t\Psi_Z(h,r^{-1}\cdot)dh\right)\cF[\varphi]\right)(x)\right|^2dx\right)^{1/2}\\
    &=N(d)\left(\int_{\bR^d}\left|(1-\Delta)^{d_0}\left(\exp\left(\int_0^t\Psi_Z(h,r^{-1}\cdot)dh\right)\cF[\varphi]\right)(\xi)\right|^2d\xi\right)^{1/2}\\
    &\leq (1+t^{2d_0})\left(\int_{c_s^{-1}\leq|\xi|\leq c_s}\frac{N}{(1+r^{-2}|\xi|^2)^t}d\xi\right)^{1/2}\\
    &\leq \frac{N(d,c_s,T,C)}{(1+r^{-2})^{t/2}}=N_0e^{-N_1s(r)^{-1}t},
\end{align*}
where $d_0:=\lfloor d/4\rfloor+1$. 
Similarly, by using H\"older's inequality, Plancherel's theorem, \eqref{22.07.12.23.47} and \eqref{22.07.06.14.11},
\begin{align*}
    &\|\psi^{\mu}(r^{-1}D)\varphi\|_{L_1(\bR^d)}\\
    &\leq N(d)\left(\int_{\bR^d}\left|(1-\Delta)^{d_0}\left(\log(1+r^{-2}|\cdot|^2)\cF[\varphi]\right)(\xi)\right|^2d\xi\right)^{1/2}\\
    &\leq N\left(\int_{c_s^{-1}\leq|\xi|\leq c_s}|\log(1+r^{-2}|\xi|^2)|^2d\xi\right)^{1/2}\leq N\log(1+r^{-2}c_s^4)\\
    &\leq N(d,c_s)\log(1+r^{-2})=N_2s(r)^{-1}.
\end{align*}
Therefore Assumption \ref{22.06.26.18.47} holds with $\mu$ and $s$ given above. 
In particular, if we consider coefficients $a$ depending only on $t$, then we obtain the following corollary from our main theorem.
\end{example}
\begin{corollary}
							\label{cor 20220713 01}
For given $p\in(1,\infty)$, $q\in[1,\infty)$, $T\in(0,\infty)$, $\gamma\in[0,\infty)$, $\varphi\in\Phi_{2}(\bR^d)$ and $u_0\in\cS'(\bR^d)$ satisfying
$$
\|u_0\|_{B_{p,q}^{\log,\varphi,\gamma-2/q}}:=\|u_0\ast\varphi_0\|_{L_p(\bR^d)}+\left(\sum_{j=1}^{\infty}\log(1+2^{2j})^{q\gamma/2}\|u_0\ast\varphi_j\|_{L_p(\bR^d)}^q\right)^{1/q}<\infty,
$$
the initial value problem
\begin{equation*}
\begin{cases}
\frac{\p u}{\p t}(t,x)=-a(t)\log(1-\Delta)u(t,x),\quad &(t,x)\in(0,T)\times\bR^d,\\
u(0,x)=u_0(x),\quad & x\in\bR^d,
\end{cases}
\end{equation*}
has a unique solution $u$ in $L_q((0,T);H_p^{\log,\gamma}(\bR^d))$, where $a$ is measurable, bounded above and below function and
$$
\|u\|_{H_p^{\log,\gamma}(\bR^d)}:=\|(1+\log(1-\Delta))^{\gamma/2}u\|_{L_p(\bR^d)}=\left(\int_{\bR^d}\left|\cF^{-1}[(1+\log(1+|\cdot|^2))^{\gamma/2}\cF[u]](x)\right|^pdx\right)^{1/p}.
$$
Moreover, the solution $u$ satisfies
$$
\|u\|_{L_q((0,T);H_p^{\log,\gamma}(\bR^d))}\leq N\|u_0\|_{B_{p,q}^{\log,\varphi,\gamma-2/q}(\bR^d)},
$$
where $N$ is independent of $u$ and $u_0$.
\end{corollary}

For easy verification of Assumptions \ref{22.06.26.18.46} and \ref{22.06.26.18.46}, we present sufficient conditions which satisfy Assumptions \ref{22.06.26.18.46} and \ref{22.06.26.18.47} the following propositions. 
Based on these propositions, we show that our operators  includes those in \cite{mikulevivcius2017p,mikulevivcius2019cauchy,mikulevicius2020cauchy}.
All proofs of these propositions  are given later.
In other words, the proofs of Propositions \ref{22.07.07.16.05}, \ref{22.07.07.16.05-2}, and \ref{22.07.07.16.05-3} are given in Section \ref{22.07.07.14.36}

\begin{prop}
							\label{22.07.07.16.05}
Let $\boldsymbol{s}=(s,s_L,s_U)$ be a scaling triple with $R_0=\infty$ and $Z$ be an additive process with a bounded triplet $(a(t),0,\Lambda_t)_{t\geq0}$. 
We  assume that there exist L\'evy measure $\nu$ and a positive constant $N_\nu $ such that for any $t\geq0$ and nonnegative $\cB(\bR^d)$-measurable function $f$,
\begin{equation}
			\label{20.08.11.14.28}
s(r)\int_{\bR^d}f(y)\Lambda_t(r\,dy):=s(r)\int_{\bR^d}f\left(\frac{y}{r}\right)\Lambda_t(dy)\geq\int_{\bR^d}f(y)\nu(dy),\quad\forall r\in(0,\infty),
\end{equation}
(we simply put $s(r)\Lambda_t(r\,dy)\geq\nu(dy)$ to denote \eqref{20.08.11.14.28}) where $\nu(dy)$  satisfies
\begin{equation}
\label{22.07.07.16.35}
    \inf_{|\xi|=1}\int_{|y|\leq N_\nu }|y\cdot\xi|^2\nu(dy) >0.
\end{equation}
Then Assumption \ref{22.06.26.18.46} holds.
\end{prop}
\begin{rem}
\label{22.07.25.11.42}
If \eqref{20.08.11.14.28} holds, then for all $c,r\in(0,\infty)$,
$$
s(r/c)\Lambda_t(r\,dy)\geq\nu(c\,dy)
$$
and \eqref{22.07.07.16.35} is replaced by
$$
\inf_{|\xi|=1}\int_{|y|\leq cN_{\nu}}|y\cdot\xi|^2\nu(dy)>0.
$$
Moreover, for any fixed $c>0$, 
\begin{align*}
s(r/c) \approx s(r).
\end{align*}
Therefore $N_{\nu}$ is not crucial in Proposition \ref{22.07.07.16.05}.
\end{rem}
\begin{prop}					\label{22.07.07.16.05-2}
Let $s$ be a scaling function and $\mu$ be a symmetric L\'evy measure on $\bR^d$.
Assume that
$$
L(r;s,\mu):=s(r)\int_{\bR^d}(1\wedge|r^{-1}y|^2)\mu(dy)=s(r)\int_{\bR^d}(1\wedge|y|^2)\mu(r\,dy),
$$
is uniformly bounded on $(0,\infty)$, that is,
\begin{equation}
\label{22.07.09.21.28}
    \sup_{0<r<\infty}L(r;s,\mu)<\infty.
\end{equation}
Then Assumption \ref{22.06.26.18.47} holds.
\end{prop}

\begin{prop}
							\label{22.07.07.16.05-3}
Let $\boldsymbol{s}=(s,s_L,s_U)$ be a scaling triple with $R_0=\infty$ and $Z$ be an additive process with a bounded triplet $(a(t),0,\Lambda_t)_{t\geq0}$. 
We assume that there exist L\'evy measure $\nu$ and positive constants $N_\nu$ and $c$ such that 
for any $t\geq0$ and nonnegative $\cB(\bR^d)$-measurable function $f$,
\begin{equation}
							\label{eqn 20220715 01}
\int_{\bR^d}f\left(y\right)\Lambda_t(dy)\geq\int_{\bR^d}f(y)\nu(dy),
\end{equation}
\begin{equation}
							\label{eqn 20220715 02}
    \inf_{r \in (0, \infty ),|\xi|=c} s(r)\int_{|y|\leq N_\nu}(1-\cos(y \cdot \xi))\nu(r\, dy) >0,
\end{equation}
and,
\begin{align}
	\label{eqn 20220718 10}
\sup_{r \in (0,\infty)}\left[ s(r)\int_{\bR^d}(1\wedge|r^{-1}y|^2)\nu(dy) \right]
= \sup_{r \in (0,\infty)} \left[ s(r)\int_{\bR^d}(1\wedge|y|^2)\nu(r\,dy) \right] < \infty,
\end{align}
where the constant $c_s$ appears in Definition \ref{20.05.28.13.20}.
Then Assumption \ref{22.06.26.18.46} holds.
Additionally, Proposition \ref{22.07.07.16.05-2} holds with 
$$
\mu= \nu_{sym}=\frac{1}{2} \left(\nu(dy) + \nu(-dy)  \right).
$$
\end{prop}

We simply say that  an additive process $Z$ with a bounded triplet $(0,0,\Lambda_t)_{t\geq0}$ has a nice lower bound $\nu$ with respect to $s$ if assumptions of Proposition \ref{22.07.07.16.05} or Proposition \ref{22.07.07.16.05-3} hold.
Moreover, we simply say that $\mu$ satisfies a weak scaling property with respect to $s$ if Proposition \ref{22.07.07.16.05-2} holds.

\begin{rem}
\label{20.06.29.15.35}
We compare our assumptions in Propositions \ref{22.07.07.16.05} and \ref{22.07.07.16.05-2} to those introduced in \cite{mikulevivcius2017p,mikulevivcius2019cauchy}.
Recall that Assumptions $A_0(\sigma)$, \textbf{D}$(\kappa,l)$, and \textbf{B}$(\kappa,l)$ in \cite{mikulevivcius2017p,mikulevivcius2019cauchy}:
\begin{itemize}
\item (Assumption $A_0(\sigma)$) Let $\mu_0(dy)=1_{|y|\leq 1}\mu_0(dy)$ and
\begin{align}
							\label{2020081601}
    \int_{|y|\leq1}|y|^2\mu_0(dy)+\int_{\bR^d}|\xi|^4[1+\lambda(\xi)]^{d+3}\exp(\psi^{\tilde{\mu}_0}(\xi))d\xi\leq n_0<\infty,
\end{align}
where
 \begin{equation*}
    \begin{gathered}
    \lambda(\xi)=\int_{|y|\leq1}\chi_{\sigma}(y)|y|(|\xi||y|\wedge1)\nu(dy),\quad \chi_{\sigma}(y)=1_{\sigma=1}1_{|y|\leq1}(y)+1_{\sigma\in(1,2)},\\
    \psi^{\tilde{\mu}_0}(\xi)=\int_{\bR^d}(\cos(y\cdot\xi)-1)\tilde{\mu}_0(dy),\quad \tilde{\mu}_0(dy)=\frac{1}{2}(\mu_0(dy)+\mu_0(-dy)).
    \end{gathered}
\end{equation*}
In addition, for any $\xi\in\{\xi\in\bR^d:|\xi|=1\}$,
$$
\int_{|y|\leq1}|y\cdot\xi|^2\mu_0(dy)\geq c_1>0.
$$
\item (Assumption \textbf{D}$(\kappa,l)$) Let $\kappa$ be a scaling function with a scaling factor $l$ (see Remark \ref{20.08.13.11.21}). For every $R>0$,
$$
\tilde{\pi}_R(dy):=\kappa(R)\pi(R\,dy)\geq\mu_0(dy)
$$
with $\mu_0$ satisfying Assumption $A_0(\sigma)$. In addition,
$$
\int_{R<|y|\leq R'}y\mu_0(dy)=0,\quad \forall0<R\leq R'\leq1,
$$
if $\sigma=1$.
\item (Assumption \textbf{B}$(\kappa,l)$) Let $\kappa$ be a scaling function with a scaling factor $l$. There exist $\alpha_1$, $\alpha_2$, and a constant $N_0>0$ such that
$$
\int_{|y|\leq1}|y|^{\alpha_1}\tilde{\pi}_R(dy)+\int_{|y|>1}|y|^{\alpha_2}\tilde{\pi}_R(dy)\leq N_0,\quad \forall R>0,
$$
where $\alpha_1,\alpha_2\in(0,1]$ if $\sigma\in(0,1)$; $\alpha_1,\alpha_2\in(1,2]$ if $\sigma\in(1,2)$; $\alpha_1\in(1,2]$ and $\alpha_2\in[0,1)$ if $\sigma=1$.
\end{itemize}
It is obvious that Assumptions $A_0(\sigma)$ and $\textbf{D}(\kappa,l)$ imply the L\'evy process $X$ with triplet $(0,0,\pi)$ has a nice lower bound $\mu_0$ with respect to $\boldsymbol{\kappa}=(\kappa,\kappa_L,\kappa_U)$, where $\kappa_L$ and $\kappa_U$ are functions defined by \eqref{22.07.07.16.16}. However, the integrability condition on the symbol \eqref{2020081601} is not necessary in our assumptions. 
Moreover, elementary computations show that assumption $\textbf{B}(\kappa,l)$ (in \cite{mikulevivcius2017p,mikulevivcius2019cauchy}) implies the weak scaling property of $\pi$ with respect to $\boldsymbol{\kappa}$ since
$$
1\wedge |y|^2 \leq |y|^{\alpha_1}1_{|y|\leq1}+|y|^{\alpha_2}1_{|y|>1},\quad \forall y\in\bR^d
$$
for all $\alpha_1,\alpha_2\in[0,2]$.
\end{rem}

\begin{rem}
Next, we compare assumptions in Proposition \ref{22.07.07.16.05-3}  to those in \cite{mikulevicius2020cauchy}.
In other words, assumptions in Proposition \ref{22.07.07.16.05-3}  are weaker than those in \cite{mikulevicius2020cauchy}.
Recall the setting and assumptions in \cite{mikulevicius2020cauchy}.
Let $\nu$ be a L\'evy measure on $\bR^d$ and define
\begin{align*}
w(r) = \frac{1}{ \nu( \{ y \in \bR^d : |y| >r \})}.
\end{align*}
Assume that $w(r)$ is an O-RV function at both and infinity.
We skip stating all details about the indices of the O-RV function (for details, see \cite[Section 2.3]{mikulevicius2020cauchy}).
They also assumed 
\begin{align}
							\label{eqn 20220714 01}
\inf_{R \in (0,\infty), |\hat \xi|=1} \int_{|y| \leq 1} \left| \hat \xi \cdot y \right|^2 \tilde \nu_R(dy) > 0,
\end{align}
where $\tilde \nu_R(dy) = w(R) \nu(R\, dy)$.
Now we claim that assumptions in Proposition \ref{22.07.07.16.05-3}  hold by putting
$$
\Lambda_t = \nu \quad \text{and} \quad s(r) = w(r).
$$
First, we show \eqref{eqn 20220714 01}  implies \eqref{eqn 20220715 02}.
Obviously,
$$
1-\cos(y\cdot\xi) \geq 0 \quad \forall y~\text{and}~ \forall \xi
$$
and
\begin{equation*}
1-\cos(\lambda)\geq\frac{1}{4}\lambda^2 \quad \forall \lambda \in [-1,1].
\end{equation*}
Fix $c \in (0,\infty)$.
Then for all $r \in (0, \infty)$ and $\xi\in\bR^d$ such that $|\xi|=c$, we have 
\begin{align*}
 & s(r) \int_{\bR^d}(1-\cos(y\cdot\xi))\nu (r\, dy) \\
 & s(r) \frac{s(r c^{-1})}{s(r c^{-1})} \int_{\bR^d}(1-\cos(y\cdot\xi))\nu (r\, dy) \\
 &= \frac{s(r) }{s(r c^{-1})} s(rc^{-1}) \int_{\bR^d}\left(1-\cos\left(y\cdot \frac{\xi}{ c } \right) \right)\nu (r c^{-1}\, dy) \\
 &\geq  \frac{s(r) }{s(r c^{-1})} s(rc^{-1}) \int_{|y| \leq 1}\left(1-\cos\left(y\cdot \frac{\xi}{c} \right) \right)\nu (r c^{-1}\, dy) \\
 &\geq  \frac{s(r) }{s(r c^{-1})} s(rc^{-1}) \int_{|y| \leq 1} |y \cdot c^{-1} \xi|^2 \nu (r c^{-1}\, dy) \\
 &\geq \frac{s(r) }{s(r c^{-1})} N \geq N,
\end{align*}
where a property of the scaling function used in the the last inequality.
Next we prove 
\begin{align*}
\sup_{r \in (0,\infty)} s(r) \int_{\bR^d}(1\wedge|r^{-1}y|^2)\nu(dy) <\infty.
\end{align*}
Observe that for all $ r \in (0,\infty)$,
\begin{align*}
s(r)\int_{|y| > r }(1\wedge|r^{-1}y|^2)\nu(dy)
&=w(r)\int_{|y| > r } \nu(dy)  \\
&= w(r)\nu( \{ y \in \bR^d : |y| > r\} ) 
= 1
\end{align*}
and by \cite[Lemma 4]{mikulevicius2020cauchy},
\begin{align*}
s(r)\int_{|y|  \leq r }(1\wedge|r^{-1}y|^2)\nu(dy)
&=w(r)\int_{|y| \leq r }  |r^{-1}y|^2 \nu(dy) \\
&=\int_{|y| \leq 1 }  |y|^2 w(r) \nu(r\, dy) \leq N,
\end{align*}
where $N$ is independent of $r$. 
\end{rem}
\begin{example}
\label{20.07.05.16.22}
With the help of the characteristic exponent $\Psi_Z(r,\xi)$ of an additive process $Z$, one can understand the infinitesimal generators $\cA_Z(t)$ as the following time dependent pseudo-differential operators (see Lemma \ref{20.04.11.14.30}): 
$$
\cA_Z(r)u(t,x)= \cF^{-1} \left[\Psi_Z(r,\cdot)\cF\left[u(t,\cdot) \right] \right](x).
$$
We remark that the symbol $\Psi_Z(r,\xi)$ does not have to be smooth on the whole space.
Here is a more concrete example. 
Let $\alpha\in(0,2)$ and $c_j(t,y)$ $(j=1,2,\cdots,d)$ be a measurable function on $[0,\infty)\times\bR^d$ such that
$$
0<C^{-1}\leq c_j(t,y)\leq C, \quad \forall (t,y)\in[0,\infty)\times\bR^d,\quad j=1,2,\cdots,d.
$$
For each $t > 0$ and $j=1,2,\cdots,d$, define L\'evy measures
\begin{equation*}
\begin{gathered}
\Lambda_{j,t}(dy):=\frac{c_j(t,y)}{|y^j|^{1+\alpha}}dy^j\epsilon_0(dy^1,\cdots,dy^{j-1},dy^{j+1},\cdots,dy^d),\quad\Lambda_{t}(dy):=\sum_{j=1}^d\Lambda_{j,t}(dy),
\end{gathered}
\end{equation*}
where $\epsilon_0$ is the centered Dirac measure on $\bR^{d-1}$. One can observe that for $r>0$,
\begin{equation}
\label{20.07.05.15.25}
    \Lambda_{j,t}(r\,dy)=\frac{c_j(t,ry)}{r^{\alpha}|y^j|^{1+\alpha}}dy^j\epsilon_0(dy^1,\cdots,dy^{j-1},dy^{j+1},\cdots,dy^d).
\end{equation}
One can easily check that $\boldsymbol{s}=(s,s_L,s_U)$ defined by
$$
s(r)=s_L(r)=s_U(r):=r^{\alpha}
$$
is a scaling triple with $c_s=2$. Put
\begin{equation*}
    \begin{gathered}
    \nu_j(dy):=1_{R_d}(y)\frac{1}{C|y^j|^{1+\alpha}}dy^j\epsilon_0(dy^1,\cdots,dy^{j-1},dy^{j+1},\cdots,dy^d),\quad \nu(dy):=\sum_{j=1}^d\nu_j(dy)\\
    \mu(dy):=\frac{1}{|y|^{d+\alpha}}dy,
    \end{gathered}
\end{equation*}
where
$$
R_d:=(-r_0,r_0)\times\cdots(-r_0,r_0)\subseteq \{y\in\bR^d:|y|<4^{-1}\}.
$$
By \eqref{20.07.05.15.25},
$$
s(r)\Lambda_{j,t}(r\,dy)\geq\nu_j(dy).
$$
Furthermore, for $\xi\in\{\xi\in\bR^d:|\xi|=1\}$,
$$
\int_{\bR^d}|y\cdot\xi|^2\nu(dy)=\sum_{j=1}^d\frac{2|\xi^j|^2}{C}\int_{0}^{r_0}x^2x^{-1-\alpha}dx=\frac{2|\xi|^2}{C(2-\alpha)}r_0^{2-\alpha}=\frac{2r_0^{2-\alpha}}{C(2-\alpha)}>0.
$$
Therefore, $s(r)\Lambda_t(r\,dy)\geq\nu(dy)$ and this implies that the additive process $Z$ with a bounded triplet $(0,0,\Lambda_t)_{t\geq0}$ has a lower bound $\nu$ with respect to $\boldsymbol{s}$.

Next we check the weak scaling property of $\mu$. By the change of variable formula, for $r>0$,
\begin{equation*}
    \begin{aligned}
    N(d)\int_{\bR^d}(1\wedge|r^{-1}y|^2)\mu(dy)&=r^{-2}\int_0^r x^{1-\alpha}dx+\int_{r}^{\infty}\frac{1}{x^{1+\alpha}}dx\\
    &=\frac{r^{-\alpha}}{2-\alpha}+\frac{r^{-\alpha}}{\alpha}.
    \end{aligned}
\end{equation*}
Therefore, $\mu$ satisfies a weak scaling property with respect to $\boldsymbol{s}$. This example leads to the following corollary.
\end{example}
Before stating the corollary, recall the fractional Laplacian operator with respect to the $j$-th component.
For $j=1,\ldots,d$, define
$$
(-\Delta)^{\alpha/2}_{x^j}\varphi(x):=\frac{1}{(2\pi)^{d/2}}\int_{\bR^d}|\xi^j|^{\alpha}\cF[\varphi](\xi)e^{ix\cdot\xi}d\xi.
$$
Note that $|\xi^j|^{\alpha}$ is not smooth on the line $\{\xi \in \bR^d : \xi^j=0\}$.
\begin{corollary}
			\label{20.07.08.15.05}
Let $\alpha\in(0,2)$ and $c_j(t)$ $(j=1,\ldots, d)$ be a measurable function on $[0,\infty)$ and assume that there exists a positive constant $C\geq1$ such that 
$$
0<C^{-1}\leq c_j(t)\leq C,\quad \forall t\in[0,\infty).
$$
Then for all $p\in(1,\infty)$,  $q\in[1,\infty)$, $T\in(0,\infty)$, $\gamma\in\bR$, and $u_0\in B_{p,q}^{\gamma-\frac{\alpha}{q}}(\bR^d)$, the initial value problem
\begin{equation*}
\begin{cases}
\frac{\p u}{\p t}(t,x)=-\sum_{j=1}^dc_j(t)(-\Delta)^{\alpha/2}_{x^j}u(t,x),\quad &(t,x)\in(0,T)\times\bR^d,\\
u(0,x)=u_0(x),\quad & x\in\bR^d,
\end{cases}
\end{equation*}
has a unique solution $u\in L_q((0,T);H_p^{\gamma}(\bR^d))$. Moreover if $q\in(1,\infty)$, then the solution $u$ satisfies
\begin{equation*}
    \|u\|_{L_q((0,T);H_p^{\gamma}(\bR^d))}\leq N\|u_0\|_{B_{p,q}^{\gamma-\frac{\alpha}{q}}(\bR^d)},
\end{equation*}
where $N$ is independent of $u$ and $u_0$. 
\end{corollary}
\begin{proof}
It is an easy consequence of Proposition \ref{22.07.07.16.05}, Proposition \ref{22.07.07.16.05-2}, and Example \ref{20.07.05.16.22} if we consider the case $\gamma\in[0,\infty)$. For the case $\gamma\in(-\infty,0)$ using the facts that $(1-\Delta)^{-\gamma_0/2}$ is a bijective quasi-isometry from $B_{p,q}^{\gamma-\frac{\alpha}{q}}(\bR^d)$ to $B_{p,q}^{\gamma+\gamma_0-\frac{\alpha}{q}}(\bR^d)$ (cf. (\cite[Theorem 2.2]{sawano2018theory})), we also obtain the results. The corollary is proved.
\end{proof}
\begin{rem}
Thanks to the associate editor and referees, we figured out that Corollary \ref{20.07.08.15.05} could be obtained from the results in \cite{ mikulevivcius2017p,mikulevivcius2019cauchy} if all coefficients $c_j$ are independent of time $t$.
\end{rem}

\mysection{The Cauchy problems with smooth functions}
\label{20.06.20.22.06}
For a L\'evy process $X$, then the generator $\cA_X$ does not map $\cS(\bR^d)$ into  $\cS(\bR^d)$ itself (see \cite[Remark 2.1.10]{farkas2001function} ). It naturally raises necessities of smooth function spaces larger than $\cS(\bR^d)$ which control the range of $\cA_X$ on $\cS(\bR^d)$.
It is enough to consider a set of functions whose derivatives  are contained in $L_p$-spaces and it is provided in \cite{mikulevivcius2017p, mikulevivcius2019cauchy} previously. We consider a subspace of this space to avoid a trace issue with respect to the time variable.
Here are the definitions of the function space which play important roles in approximations of smooth functions which fit to infinitesimal generators of additive processes.
\begin{defn}
\label{20.05.17.14.08}
Let $T\in(0,\infty)$ and $p\in[1,\infty)$. 
 \begin{enumerate}[(i)]
 \item $C_p^{\infty}(\bR^d)$ denotes the set of all infinitely differentiable functions $f$ on $\bR^d$ such that for any multi-index $\alpha$, their derivatives
     $$
     D^{\alpha}_xf\in L_p(\bR^d).
     $$
     Define
     $$
     \tilde{C}^{\infty}(\bR^d):=\bigcap_{p\in[1.\infty)}C_p^{\infty}(\bR^d).
     $$
    \item  $C_p^{\infty}([0,T]\times\bR^d)$ denotes the set of all $\cB([0,T]\times\bR^d)$-measurable functions $f$ on $[0,T]\times\bR^d$ such that for any multi-index $\alpha$ with respect to the space variable, their derivatives
\begin{equation*}
D^{\alpha}_{x}f\in L_{\infty}([0,T];L_p(\bR^d)).
\end{equation*}
Define
$$
\tilde{C}^{\infty}([0,T]\times\bR^d):=\bigcap_{p\in[1,\infty)}C_p^{\infty}([0,T]\times\bR^d).
$$
\item $C_p^{1,\infty}([0,T]\times\bR^d)$ denotes the set of all $f\in C_p^{\infty}([0,T]\times\bR^d)$ such that for any multi-index $\alpha$, their derivatives
\begin{equation*}
D^{\alpha}_{x}f\in C([0,T]\times \bR^d),
\end{equation*}
and
$$
\frac{\p f}{\p t}\in C_p^{\infty}([0,T]\times\bR^d).
$$
Define
$$
\tilde{C}^{1,\infty}([0,T]\times\bR^d):=\bigcap_{p\in[1,\infty)}C_p^{1,\infty}([0,T]\times\bR^d).
$$
\end{enumerate}
\end{defn}

Due to Sobolev's embedding theorem, for all $f \in C_p^\infty(\fR^d)$  and multi-index $\alpha$,
$$
D_x^\alpha f \in C(\bR^d).
$$
Here are specific examples contained in $\tilde{C}^{\infty}([0,T]\times\bR^d)$ and $\tilde{C}^{1,\infty}([0,T]\times\bR^d)$.
\begin{example}
\label{20.04.22.17.12}
\begin{enumerate} [(i)]
    \item Obviously, $C_c^{\infty}((0,T)\times\bR^d)\subsetneq \tilde{C}^{\infty}([0,T]\times\bR^d)$.
    \item Let $f$ be a infinitely differentiable function on $[0,T]\times\bR^d$ such that for any multi-indices $\alpha$ and $\beta$,
\begin{equation}
\label{20.04.22.15.18}
    \sup_{t\in[0,T]}\sup_{x\in\bR^d}|x^{\beta}D^{\alpha}_{t,x}f(t,x)|<\infty.
\end{equation}
Then $f\in\tilde{C}^{\infty}([0,T]\times\bR^d)$.
    \item Let 
    $$f(t,x):=\frac{t}{(1+|x|^2)^d}.$$
    Then $f$ is not in $C_c^{\infty}((0,T)\times\bR^d)$ and does not satisfy \eqref{20.04.22.15.18}. However, it is easy to check that $f\in\tilde{C}^{\infty}([0,T]\times\bR^d)$.
    \item For $u_0\in \tilde{C}^{\infty}(\bR^d)$ and $f\in\tilde{C}^{\infty}([0,T]\times\bR^d)$, put
    $$
    g(t,x):=u_0(x)+\int_{0}^tf(s,x)ds.
    $$
Then it is obvious  that $D^{\alpha}_xg\in C([0,T]\times\bR^d)$ for all multi-index $\alpha$ with respect to the space variable.  Additionally, by the fundamental theorem of calculus,
    $$
    \frac{\p g}{\p t}=f\in\tilde{C}^{\infty}([0,T]\times\bR^d).
    $$
    Therefore, $g\in\tilde{C}^{1,\infty}([0,T]\times\bR^d)$.
\end{enumerate}
\end{example}

Example \ref{20.04.22.17.12}(iii) shows that $\tilde{C}^{\infty}([0,T]\times\bR^d)$ contains some smooth functions with polynomial decays.
Thus one cannot expect that the Fourier transform operator  maps $\tilde{C}^{\infty}([0,T]\times\bR^d)$ into some smooth function spaces. 
However, fortunately, the  Fourier inversion theorem still works on the set of all Fourier transforms of $\tilde{C}^{\infty}([0,T]\times\bR^d)$ with respect to the space variable.
\begin{lem}
\label{20.03.30.13.32}
Let $T\in(0,\infty)$ and $f\in\tilde{C}^{\infty}([0,T]\times\bR^d)$. 
Then for almost every $0\leq t \leq T$,
$$
\cF[f(t,\cdot)]\in L_1(\bR^d)\cap C(\bR^d).
$$
In particular, the Fourier inversion theorem holds, i.e. for almost every $0\leq t \leq T$,
\begin{equation*}
    \cF^{-1}[\cF[f(t,\cdot)]](x)=f(t,x),\quad \forall x \in \fR^d.
\end{equation*}
\end{lem}
\begin{proof}
Since $f(t,\cdot)\in L_1(\bR^d)$,
$$\cF[f(t,\cdot)]\in C(\bR^d).$$
Let $d_0=\lfloor d/4\rfloor+1$. By the definition of $\tilde{C}^{\infty}([0,T]\times\bR^d)$,
$$(1-\Delta)^{d_0}f(t,\cdot)\in L_2(\bR^d).$$
By Plancherel's theorem,
$$(1+|\cdot|^2)^{d_0}\cF[f(t,\cdot)]\in L_2(\bR^d).$$
Therefore, by virtue of H\"older's inequality,
$$
\int_{\bR^d}|\cF[f](t,\xi)|d\xi\leq\left(\int_{\bR^d}\frac{1}{(1+|\xi|^2)^{2d_0}}d\xi\right)^{1/2}\left(\int_{\bR^d}(1+|\xi|^2)^{2d_0}|\cF[f](t,\xi)|^2d\xi\right)^{1/2}<\infty.
$$
The lemma is proved.
\end{proof}
Next we show that an infinitesimal generator of an additive process $Z$ can be represented as a pseudo-differential operator or a nonlocal operator.
\begin{lem}
\label{20.04.11.14.30}
Let $Z$ be an additive process with a bounded triplet $(a(t),0,\Lambda_t)_{t\geq0}$. Then $\cA_Z(t)$ is a linear operator on $\tilde{C}^{\infty}([0,T]\times\bR^d)$, i.e. $\cA_Z(t)f \in \tilde{C}^{\infty}([0,T]\times\bR^d)$ for all  $f \in \tilde{C}^{\infty}([0,T]\times\bR^d)$. 
Moreover, for almost all $t$, 
\begin{equation*}
\begin{aligned}
\cA_Z(t)f(t,x)&:=\lim_{h\downarrow0}\frac{\bE[f(t,x+Z_{t+h}-Z_t)-f(t,x)]}{h}=\cF^{-1}[\Psi_Z(t,\cdot)\cF[f(t,\cdot)]](x)\\
    &=a(t)\cdot\nabla_x f(t,x)+\int_{\bR^d}(f(t,x+y)-f(t,x)-y\cdot \nabla_x f(t,x)1_{|y|\leq1})\Lambda_t(dy) 
\end{aligned}
\end{equation*}
for all $x \in \bR^d$, where
$$
\Psi_Z(t,\xi):=ia(t)\cdot\xi+\int_{\bR^d}(e^{iy\cdot\xi}-1-iy\cdot\xi1_{|y|\leq1})\Lambda_t(dy).
$$
\end{lem}
\begin{proof}
Let $f\in\tilde{C}^{\infty}([0,T]\times\bR^d)$. 
Then, applying Lemma \ref{20.03.30.13.32}, Fubini's theorem, the Lebesgue dominated convergence theorem, and elementary properties of the Fourier transform, we have
\begin{equation*}
\begin{aligned}
\cA_Z(t)f(t,x)&:=\lim_{h\downarrow0}\frac{\bE[f(t,x+Z_{t+h}-Z_t)-f(t,x)]}{h}\\
    &=\lim_{h\downarrow0}\frac{\cF^{-1}[(e^{\int_{t}^{t+h}\Psi_Z(r,\cdot)dr}-1)\cF[f](t,\cdot)](x)}{h}\\
	&=\cF^{-1}[\Psi_Z(t,\cdot)\cF[f](t,\cdot)](x)\\
	&=a(t)\cdot\nabla_x f(t,x)+\int_{\bR^d}(f(t,x+y)-f(t,x)-y\cdot \nabla_x f(t,x)1_{|y|\leq1})\Lambda_t(dy).
\end{aligned}
\end{equation*}
Since $a\cdot\nabla_x f\in \tilde{C}^{\infty}([0,T]\times\bR^d)$, it suffices to show that
$$
I^{\Lambda}f(t,x):=\int_{\bR^d}(f(t,x+y)-f(t,x)-y\cdot \nabla_x f(t,x)1_{|y|\leq1})\Lambda_t(dy)
$$
is in $\tilde{C}^{\infty}([0,T]\times\bR^d)$. By Taylor's theorem and the Lebesgue dominated convergence theorem, 
\begin{equation}
\label{20.04.10.16.44}
D^{\alpha}_{x}I^{\Lambda}f(t,x)=I^{\Lambda}D^{\alpha}_{x}f(t,x),
\end{equation}
for any multi-index $\alpha$ with respect to the space variable. Due to Minkowski's inequality and Taylor's theorem,
\begin{equation}
\label{20.04.10.16.45}
\begin{gathered}
\|I^{\Lambda}f(t,\cdot)\|_{L_p(\bR^d)}\leq 4 L(t;\Lambda)\sum_{|\alpha|\leq2}\|D^{\alpha}_xf(t,\cdot)\|_{L_p(\bR^d)},\quad \forall p\in[1,\infty]
\end{gathered}
\end{equation}
where
$$
L(t;\Lambda)=\int_{\bR^d}(1\wedge|y|^2)\Lambda_t(dy).
$$
Therefore, combining two relations \eqref{20.04.10.16.44} and \eqref{20.04.10.16.45}, we have $\cA_Z(t)f \in \tilde{C}^{\infty}([0,T]\times\bR^d)$.
The lemma is proved.
\end{proof}
Next, we show that if data $f$ and $u_0$ are smooth, then the following Cauchy problem
\begin{equation}
				\label{20.06.06.16.02}
\begin{cases}
\frac{\p u}{\p t}(t,x)=\cA_Z(t)u(t,x)+f(t,x),\quad &(t,x)\in(0,T)\times\bR^d,\\
u(0,x)=u_0(x),\quad & x\in\bR^d
\end{cases}
\end{equation}
has a unique strong solution. 
Note that the existence and uniqueness of a solution to \eqref{20.06.06.16.02} is guaranteed without our main assumptions unless we consider a regularity improvement on a solution. 
\begin{thm}
					\label{20.05.04.10.58}
Let $Z$ be an additive process with bounded triplet $(a(t),0,\Lambda_t)_{t\geq0}$, $T\in(0,\infty)$, $u_0\in\tilde C^\infty(\bR^d)$, and $f\in\tilde{C}^{\infty}([0,T]\times\bR^d)$. 
Then there exists a unique solution $u\in\tilde{C}^{1,\infty}([0,T]\times\bR^d)$  to the Cauchy problem \eqref{20.06.06.16.02}.
Moreover, the solution $u$ has the following probabilistic representation :
$$
u(t,x)=\bE[u_0(x+Z_t)] + \int_{0}^t\bE[f(s,x+Z_t-Z_s)]ds,\quad  \forall (t,x)\in [0,T]\times\bR^d.
$$
\end{thm}
\begin{proof}
\textbf{(Uniqueness)} We prove the uniqueness first. Suppose that $u\in\tilde{C}^{1,\infty}([0,T]\times\bR^d)$ is a solution of the Cauchy problem
\begin{equation*}
\begin{cases}
\frac{\p u}{\p t}(t,x)=\cA_Z(t)u(t,x),\quad &(t,x)\in(0,T)\times\bR^d,\\
u(0,x)=0,\quad & x\in\bR^d.
\end{cases}
\end{equation*}
Taking the Fourier transform to the above equation with respect to the space variable and integration with respect to the time variable, we have  
\begin{equation*}
\cF[u(t,\cdot)](\xi)=\int_0^t\Psi_Z(s,\xi)\cF[u(s,\cdot)](\xi)ds,\quad \forall 0\leq t<T, \quad \forall \xi \in \bR^d.
\end{equation*}
Note that due to Taylor's theorem and (locally) bounded assumptions on the triplet,
$$
|\Psi_Z(t,\xi)|\leq4L(t;\Lambda)(1+|\xi|^2)+ |a(t)||\xi|.
$$
Since $a(t)$ and $L(t;\Lambda)$ are in $L_{\infty}([0,T])$, for each $\xi\in\bR^d$,
$$
\Psi_Z(\cdot,\xi)\in L_{\infty}([0,T]).
$$
Hence, Gronwall's inequality yields that $u=0$. Therefore by the linearity of equation \eqref{20.06.06.16.02}, the uniqueness is proved 
\vspace{3mm}

\textbf{(Existence)} Next we prove the existence of a solution. Define
\begin{equation*}
u(t,x):=I_Z(u_0)(t,x)+F_Z(f)(t,x),
\end{equation*}
where
$$
I_Z(u_0)(t,x):=\bE[u_0(x+Z_t)],\quad \forall (t,x)\in[0,T]\times\bR^d,
$$
and
$$
F_Z(f)(t,x):=\int_{0}^t\bE[f(s,x+Z_t-Z_s)]ds,\quad  \forall(t,x)\in[0,T]\times\bR^d.
$$
It is an easy application of the Lebesgue dominated convergence theorem that for multi-index $\alpha$ with respect to the space variable and $(t,x)\in[0,T]\times\bR^d$,
\begin{equation}
\label{20.07.11.20.26}
    D^{\alpha}_xI_Z(u_0)(t,x)=I_Z(D^{\alpha}_xu_0)(t,x),\quad D^{\alpha}_xF_Z(f)(t,x)=F_Z(D^{\alpha}_xf)(t,x).
\end{equation}
Therefore $u\in\tilde{C}^{\infty}([0,T]\times\bR^d)$ since $u_0\in \tilde C^\infty(\bR^d)$, and $f\in\tilde{C}^{\infty}([0,T]\times\bR^d)$.
By Fubini's theorem, for $(t,\xi)\in(0,T]\times\bR^d$,
\begin{equation*}
\begin{aligned}
    \cF[I_Z(u_0)(t,\cdot)+F_Z(f)(t,\cdot)](\xi)&=\bE[\cF[u_0](\xi)e^{i\xi\cdot Z_t}]+\int_{0}^t\bE[\cF[f(s,\cdot)](\xi)e^{i\xi\cdot (Z_{t}-Z_s)}]ds\\
    &=\cF[u_0](\xi)e^{\int_0^t\Psi_Z(r,\xi)dr}+\int_{0}^t\cF[f(s,\cdot)](\xi)e^{\int_{s}^t\Psi_Z(r,\xi)dr}ds.
\end{aligned}
\end{equation*}
Since $|\exp(\int_{s}^t\Psi_Z(r,\cdot)dr)|\leq 1$, we have
$$
\left|\cF[u_0](\xi)e^{\int_0^t\Psi_Z(r,\xi)dr}+\int_{0}^t\cF[f(s,\xi)]e^{\int_{s}^t\Psi_Z(r,\xi)dr}ds\right|\leq |\cF[u_0](\xi)|+\int_{0}^t|\cF[f(s,\xi)]|ds.
$$
Moreover, as a result of Lemma \ref{20.03.30.13.32}, for  each $t\in[0,T]$,
$$
|\cF[u_0]|+\int_{0}^t|\cF[f(s,\cdot)]|ds\in L_1(\bR^d).
$$
Using Fubini's theorem again, for $(t,\xi)\in(0,T]\times\bR^d$, we have
\begin{equation}
\label{20.07.11.20.31}
\begin{aligned}
&I_Z(u_0)(t,x)+F_Z(f)(t,x)\\
&=\cF^{-1}[\cF[u_0]e^{\int_{0}^t\Psi_Z(r,\cdot)dr}](x)+\int_0^t\cF^{-1}[\cF[f(s,\cdot)]e^{\int_{s}^t\Psi_Z(r,\cdot)dr}](x)ds.
\end{aligned}
\end{equation}
The Lebesgue dominated convergence theorem, \eqref{20.07.11.20.31} and \eqref{20.07.11.20.26} yield that for any multi-index $\alpha$ with respect to the space variable,
$$
D^{\alpha}_xu\in C([0,T]\times\bR^d).
$$
Next, define
$$
v(t,x):=u_0(x)+\int_{0}^t(\cA_Z(s)u(s,x)+f(s,x))ds.
$$
By Lemma \ref{20.04.11.14.30} and Example \ref{20.04.22.17.12} (iv), $v\in\tilde{C}^{1,\infty}([0,T]\times\bR^d)$.
We claim that 
$$
u(t,x)=v(t,x),\quad\forall(t,x)\in[0,T]\times\bR^d,
$$
which obviously completes the proof. For $(t,x)\in[0,T]\times\bR^d$, applying Lemma \ref{20.03.30.13.32}, \ref{20.04.11.14.30}, \eqref{20.07.11.20.31} and Fubini's theorem, we have
\begin{equation*}
    \begin{aligned}
    &v(t,x)-u_0(x)-\int_0^tf(s,x)ds\\
    &=\int_{0}^t\cF^{-1}\left[\Psi_Z(s,\cdot)\left(\cF[u_0]e^{\int_{0}^s\Psi_Z(r,\cdot)dr}+\int_0^s\cF[f(l,\cdot)]e^{\int_{l}^s\Psi_Z(r,\cdot)dr}dl\right)\right](x)ds\\
    &=\cF^{-1}\left[\cF[u_0](e^{\int_{0}^t\Psi_Z(r,\cdot)dr}-1)\right](x)+\int_0^t\cF^{-1}\left[\cF[f(l,\cdot)](e^{\int_{l}^t\Psi_Z(r,\cdot)dr}-1)\right](x)dl\\
    &=I_Z(u_0)(t,x)+F_Z(f)(t,x)-u_0(x)-\int_0^tf(s,x)ds=u(t,x)-u_0(x)-\int_0^tf(s,x)ds.
    \end{aligned}
\end{equation*}
The theorem is proved.
\end{proof}
In the next corollary, we show that a solution to \eqref{20.06.06.16.02} is unique in the sense of Definition \ref{20.06.06.15.50} even for general data. The proof is based on the probabilistic solution representation and uniqueness of a solution for smooth data in Theorem \ref{20.05.04.10.58}.
\begin{corollary}
\label{20.06.06.16.06}
Let $Z$ be an additive process with a bounded triplet $(a(t),0,\Lambda_t)_{t\geq0}$. Then the Cauchy problem \eqref{20.06.06.16.02} has at most one solution in the sense of Definition \ref{20.06.06.15.50}.
\end{corollary}
\begin{proof}
Let $u,v\in L_q((0,T);H_p^{\mu;\gamma+2}(\bR^d))$ be solutions to Cauchy problem \eqref{20.06.06.16.02}. Then by the definition of solution, one can find $u_n,v_n\in\tilde{C}^{1,\infty}([0,T]\times\bR^d)$ such that $u_n(0,\cdot),v_n(0,\cdot)\in C_c^{\infty}(\bR^d)$ and 
\begin{equation*}
\begin{gathered}
\frac{\p u_n}{\p t}-\cA_Z(t)u_n,\quad \frac{\p v_n}{\p t}-\cA_Z(t)v_n\to f\quad\text{in}\quad L_q((0,T);H_p^{\mu;\gamma}(\bR^d)),\\
u_n(0,\cdot),v_n(0,\cdot)\to u_0\quad \text{in }\quad B_{p,q}^{s.\varphi;\gamma+2-\frac{2}{q}}(\bR^d)\\
u_n\to u \quad\text{in}\quad L_q((0,T);H_p^{\mu;\gamma+2}(\bR^d)),\\
v_n\to v \quad\text{in}\quad L_q((0,T);H_p^{\mu;\gamma+2}(\bR^d)).
\end{gathered}
\end{equation*}
Let
\begin{equation*}
    \begin{gathered}
    f_n:=\frac{\p u_n}{\p t}-\cA_Z(t)u_n\in\tilde{C}^{\infty}([0,T]\times\bR^d),\\
    g_n:=\frac{\p v_n}{\p t}-\cA_Z(t)v_n\in\tilde{C}^{\infty}([0,T]\times\bR^d).
    \end{gathered}
\end{equation*}
Then by Proposition \ref{20.05.04.10.58},
\begin{equation*}
    \begin{gathered}
    u_n(t,x)=\bE[u_n(0,x+Z_t)]+\int_{0}^t\bE[f_n(s,x+Z_t-Z_s)]ds,\\
    v_n(t,x)=\bE[v_n(0,x+Z_t)]+\int_{0}^t\bE[g_n(s,x+Z_t-Z_s)]ds.
    \end{gathered}
\end{equation*}
Since both $f_n$ and $g_n$ converge to $f$ in $L_q((0,T);H_p^{\mu;\gamma}(\bR^d))$ and both $u_n(0,\cdot)$ and $v_n(0,\cdot)$ converge to $u(0,\cdot)$ in $B_{p,q}^{s,\varphi;\gamma+2-\frac{2}{q}}(\bR^d)$, we conclude that $u=v$ as a tempered distribution valued functions.
Finally, we have $u=v$ in $L_q((0,T);H_p^{\mu;\gamma+2}(\bR^d))$ since both $u$ and $v$ are in the class $L_q((0,T);H_p^{\mu;\gamma+2}(\bR^d))$. 
The corollary is proved.
\end{proof}

\mysection{Properties of function spaces}
\label{20.08.20.17.29}
There are many well-known interesting properties for the classical Bessel potential and Besov spaces such as the denseness of smooth functions, the duality property, and interpolation properties (cf. \cite{triebel1978interpolation,triebel1983theory}). In this section, we provide that  lots of useful properties still hold even for  $\psi$-Bessel potential spaces and scaled Besov spaces. 
For the readers' convenience, the detailed proofs of the following results are attached in the appendix (Section \ref{20.08.20.15.44}).
\begin{prop}
\label{properties}
Let $\mu$ be a symmetric L\'evy measure on $\bR^d$ and denote
$$
\psi^{\mu}(\xi):=\int_{\bR^d}(e^{iy\cdot\xi}-1-iy\cdot\xi1_{|y|\leq1})\mu(dy).
$$
For $\alpha \in \bR$ and $\lambda\geq0$, we define the pseudo-differential operator
$$
\left(\lambda - \psi^{\mu}(D)\right)^{\alpha/2}f(x):=\cF^{-1}[\left(\lambda - \psi^{\mu}\right)^{\alpha/2}\cF[f]](x) \quad x \in \bR^d
$$ 
and
$$
\psi^{\mu}(D)f(x):=\cF^{-1}[\psi^{\mu}\cF[f]](x).
$$
Then the following statements hold:
\begin{enumerate}[(i)]
\item If $\phi_n\to\phi$ in $\cS(\bR^d)$, then $\phi_n\to\phi$ in $H_p^{\mu;\gamma}(\bR^d)$.
    \item For $p\in[1,\infty)$, $\gamma\in\bR$, and $f\in L_p(\bR^d)$, $(1-\psi^{\mu}(D))^{\gamma/2}f\in H_p^{\mu;-\gamma}(\bR^d)$.
    \item For $p\in[1,\infty)$ and $\gamma\in\bR$, $H_p^{\mu;\gamma}(\bR^d)$ is a Banach space equipped with the norm $\|\cdot\|_{H_p^{\mu;\gamma}(\bR^d)}$.
    \item For $p \in [1,\infty)$ and $\gamma \in \bR$, 
        $$
    \tilde{C}^{\infty}(\bR^d) \subseteq H_p^{\mu;\gamma}(\bR^d).
    $$
Moreover, $C_c^\infty(\bR^d)$ is dense in $H_p^{\mu;\gamma}(\bR^d)$. 

\item For $p\in[1,\infty)$ and $\gamma\geq0$, the norm $\|\cdot\|_{H_p^{\mu;\gamma}(\bR^d)}$ is equivalent to $\|\cdot\|_{L_p(\bR^d)}+\|\cdot\|_{\dot{H}_p^{\mu;\gamma}(\bR^d)}$, where
$$
\|f\|_{\dot{H}_p^{\mu;\gamma}(\bR^d)}:=\|(-\psi^{\mu}(D))^{\gamma/2}f\|_{L_p(\bR^d)}.
$$

\item For $p\in[1,\infty)$ and $\gamma_1,\gamma_2\in\bR$, the operator 
$$
(1-\psi^{\mu}(D))^{-\gamma_2/2}:H^{\mu;\gamma_1}_p(\bR^d)\to H^{\mu;\gamma_1+\gamma_2}_p(\bR^d)
$$
is a bijective isometry.

\item  For $p\in(1,\infty)$ and $\gamma\in\bR$,
$$
\left(H_p^{\mu;\gamma}(\bR^d)\right)^*=H_{p'}^{\mu;-\gamma}(\bR^d),
$$
where $\left(H_p^{\mu;\gamma}(\bR^d)\right)^*$ is the topological dual space of $H_p^{\mu;\gamma}(\bR^d)$ and $\frac{1}{p}+\frac{1}{p'}=1.$

\item For $p\in(1,\infty)$ and $\gamma_1\leq \gamma_2$,
$H_p^{\mu;\gamma_2}(\bR^d)$ is continuously embedded into $H_p^{\mu;\gamma_1}(\bR^d)$, i.e. there exists a positive constant $N$ such that
$$
\|f\|_{H_p^{\mu;\gamma_1}(\bR^d)}\leq N\|f\|_{H_p^{\mu;\gamma_2}(\bR^d)},
$$
for all $f\in\cS(\bR^d)$. Moreover, for $0\leq \gamma_1\leq \gamma_2$, $H_1^{\mu;\gamma_2}(\bR^d)$ is continuously embedded into $H_1^{\mu;\gamma_1}(\bR^d)$.

\item For $p_0,p_1\in(1,\infty)$, $\gamma_0,\gamma_1\in\bR$, and $0<\theta<1$,
$$
[H_{p_0}^{\mu;\gamma_0}(\bR^d),H_{p_1}^{\mu;\gamma_1}(\bR^d)]_\theta=H_p^{\mu;\gamma}(\bR^d),
$$
where 
$$
\gamma=(1-\theta)\gamma_0+\theta\gamma_1,\quad\frac{1}{p}=\frac{1-\theta}{p_0}+\frac{\theta}{p_1},
$$
and $[B_0,B_1]_\theta$ is a complex interpolation space between Banach spaces $B_0$ and $B_1$.
\end{enumerate}
\end{prop}
To state properties of a scaled Besov space, we need a definition of a Banach space valued sequence space.
\begin{defn}
For $q\in(0,\infty)$, $\gamma\in\bR$, scaling triple $\boldsymbol{s}=(s,s_L.s_U)$, and Banach space $B$, $l_q^{s;\gamma}(B)$ denote the set of all sequences $x=(x_0,x_1\cdots)$ in $B$ such that
$$
\|x\|_{l_p^{s;\gamma}(B)}:=\|x_0\|_B+\left(\sum_{j=1}^{\infty}s(c_s^{-j})^{-\frac{q\gamma}{2}}\|x_j\|_{B}^q\right)^{1/q}<\infty.
$$
We also denote $l_{\infty}^{s;\gamma}(B)$ denote the set of all sequences $x=(x_0,x_1\cdots)$ in $B$ such that
$$
\|x\|_{l_{\infty}^{s;\gamma}(B)}:=\|x_0\|_B+\sup_{j\in\bN}s(c_s^{-j})^{-\frac{\gamma}{2}}\|x_j\|_{B}<\infty.
$$
\end{defn}

We connect the above Banach space-valued sequence to the scaled Besov space $B_{p,q}^{s,\varphi;\gamma}(\bR^d)$ considering the mapping
\begin{equation*}
f\in  \cS'(\bR^d)  \mapsto I(f):=(f\ast\varphi_0,f\ast\varphi_1,\cdots).
\end{equation*}
In particular, it is easy to check that
$$
\|I(f)\|_{l_q^{s;\gamma}(L_p(\bR^d))} 
= \|f\|_{B_{p,q}^{s,\varphi;\gamma}(\bR^d)}.
$$
Recall that 
$$
m_s:=\min\{m\in\bN:s_L(2^m)> 1\},\quad c_s:=2^{m_s}>1, 
$$
and $\Phi_n(\bR^d)$ denotes a set of Littlewood-Paley functions whose frequency support is within $\{\xi \in \bR^d : n^{-1} \leq |\xi| \leq n\}$ (Definition \ref{LP function}).
\begin{prop}
\label{20.04.29.20.24}
Let $\gamma\in\bR$, $\boldsymbol{s}=(s,s_L,s_U)$ be a scaling triple, and $\varphi\in\Phi_{c_s}(\bR^d)$.
\begin{enumerate}[(i)]
\item For $p\in[1,\infty]$, $q\in(0,\infty]$, there exist a positive constant $N$ such that
\begin{equation*}
N^{-1}\|f\|_{B_{p,q}^{m_s\theta_0\gamma}(\bR^d)}\leq\|I(f)\|_{l_q^{s;\gamma}(L_p(\bR^d))} 
= \|f\|_{B_{p,q}^{s,\varphi;\gamma}(\bR^d)}\leq N\|f\|_{B_{p,q}^{m_s\theta_1\gamma}(\bR^d)},
\end{equation*}
for all $f\in\cS'(\bR^d)$, where $\theta_0=\log_{c_s}(s_L(c_s))$, $\theta_1=\log_{c_s}(s_U(c_s))$ and $N$ is independent of $f$.
\item For $p\in[1,\infty]$, $q\in(0,\infty]$, $B_{p,q}^{s,\varphi;\gamma}(\bR^d)$ is a quasi-Banach space equipped with the quasi-norm $\|\cdot\|_{B_{p,q}^{s,\varphi;\gamma}(\bR^d)}$. In particular, $B_{p,q}^{s,\varphi;\gamma}(\bR^d)$ is a Banach space if $p,q\geq1$.
\item For $p\in[1,\infty]$ and $q\in(0,\infty]$,
    $$
    \tilde{C}^{\infty}(\bR^d) \subseteq B_{p,q}^{s,\varphi;\gamma}(\bR^d).
    $$
    Moreover, if $p\in[1,\infty)$ and $q\in(0,\infty)$, then $C_c^{\infty}(\bR^d)$ is dense in $B_{p,q}^{s,\varphi;\gamma}(\bR^d)$.
    
\item For $p\in[1,\infty]$, $q\in(0,\infty]$, $\vartheta\in\Phi_{c_s}(\bR^d)$, the two quasi-norms $\|\cdot\|_{B_{p,q}^{s,\varphi;\gamma}(\bR^d)}$ and $\|\cdot\|_{B_{p,q}^{s,\vartheta;\gamma}(\bR^d)}$ are equivalent.
\item For $p\in[1,\infty)$ and $f\in\cS'(\bR^d)$,
\begin{equation}
\label{20.08.20.10.49}
    \begin{gathered}
    \|f\|_{B_{p,q}^{s,\varphi;\gamma_1}(\bR^d)}\leq \|f\|_{B_{p,q}^{s,\varphi;\gamma_2}(\bR^d)},\quad q\in(0,\infty],\,-\infty<\gamma_1\leq\gamma_2<\infty,\\
    \|f\|_{B_{p,q_2}^{s,\varphi;\gamma}(\bR^d)}\leq \|f\|_{B_{p,q_1}^{s,\varphi;\gamma}(\bR^d)},\quad \gamma\in\bR,\,0<q_1\leq q_2\leq\infty.
    \end{gathered}
\end{equation}
\item For $p\in[1,\infty)$ and $q\in[1,\infty)$, the topological dual space of $B_{p,q}^{s,\varphi;\gamma}(\bR^d)$ is a subspace of tempered distribution space $\cS'(\bR^d)$.
\item Let $q_0,q_1\in[1,\infty)$, $\gamma_0,\gamma_1\in\bR$, and $B$ be  a Banach space. Then for $\theta\in[0,1]$,
$$
[l_{q_0}^{s;\gamma_0}(B),l_{q_1}^{s;\gamma_1}(B)]_{\theta}=l_q^{s;\gamma}(B),
$$
where
$$
\frac{1}{q}=\frac{1-\theta}{q_0}+\frac{\theta}{q_1},\quad \gamma=(1-\theta)\gamma_0+\theta\gamma_1.
$$
\item Let $p\in[1,\infty]$, $q_0,q_1\in[1,\infty)$, and $\gamma_0,\gamma_1\in\bR$. Then for $\theta\in[0,1]$,
$$
[B_{p,q_0}^{s,\varphi;\gamma_0}(\bR^d),B_{p,q_1}^{s,\varphi;\gamma_1}(\bR^d)]_{\theta}=B_{p,q}^{s;\gamma}(\bR^d),
$$
where $[B_0,B_1]_{\theta}$ denotes a complex interpolation space between $B_0$ and $B_1$ and
$$
\quad\frac{1}{q}=\frac{1-\theta}{q_0}+\frac{\theta}{q_1},\quad \gamma=(1-\theta)\gamma_0+\theta\gamma_1.
$$
\end{enumerate}
\end{prop}


\mysection{Proof of Theorem \ref{main1}}
\label{20.08.20.17.31}
We first prove a priori estimate for a smooth function in nonnegative even integer order spaces, which is the most crucial part of the proof of our main theorem. For convenience, we write $c$ instead of $c_s$.

\begin{lem}
\label{22.07.13.00.07}
    Let $Z$ be an additive process and $\mu$ be a L\'evy measure in $\bR^d$. Then
    \begin{equation*}
\int_{\bR^d}|\bE[\varphi(x+Z_t)]|dx+\|\psi^{\mu}(D)\varphi\|_{L_1(\bR^d)}\leq N\sum_{|\alpha|\leq2}\|D^{\alpha}\varphi\|_{L_1(\bR^d)},\quad \forall \varphi\in \cS(\bR^d),
    \end{equation*}
where $N$ depend only on $\mu$.
\end{lem}
\begin{proof}
    By Minkowski's inequality
    $$
\int_{\bR^d}|\bE[\varphi(x+Z_t)]|dx\leq \|\varphi\|_{L_1(\bR^d)},
    $$
    and by \eqref{20.04.10.16.45}
    $$
    \|\psi^{\mu}(D)\varphi\|_{L_1(\bR^d)}\leq 4L(\mu)\sum_{|\alpha|\leq 2}\|D^{\alpha}\varphi\|_{L_1(\bR^d)},
    $$
    where
    $$
    L(\mu):=\int_{\bR^d}(1\wedge|y|^2)\mu(dy).
    $$
    The lemma is proved.
\end{proof}

\begin{lem}[A priori estimate]
								\label{a priori lem}
Let $p\in[1,\infty)$, $q\in(0,\infty)$, $T\in(0,\infty)$, $n\in \bN \cup \{0\}$, $\mu\in\frL_d$, $\boldsymbol{s}=(s,s_L,s_U)$ be a scaling triple, and $Z$ be an additive process with a bounded triplet $(a(t),0,\Lambda_t)_{t\geq0}$. If Assumptions \ref{22.06.26.18.46}  and \ref{22.06.26.18.47} hold, then  there exists a positive constant $N$ such that for all $u_0\in C_c^{\infty}(\bR^d)$,
\begin{equation}
    \|(-\psi^{\mu}(D))^{n}I_Z(u_0)\|_{L_q((0,T);L_p(\bR^d))}\leq N\|u_0\|_{B_{p,q}^{s;2n-\frac{2}{q}}(\bR^d)},
\end{equation}
where 
$$
I_Z(u_0)(t,x):=\bE[u_0(x+Z_t)].
$$
\end{lem}
\begin{proof}
By Theorem \ref{20.05.04.10.58},
$$
u(t,x):=I_Z(u_0)(t,x)\in \tilde{C}^{1,\infty}([0,T]\times\bR^d)
$$
is a unique solution to the Cauchy problem
\begin{equation*}
\begin{cases}
\frac{\p u}{\p t}(t,x)=\cA_Z(t)u(t,x),\quad &(t,x)\in(0,T)\times\bR^d,\\
u(0,x)=u_0(x),\quad & x\in\bR^d.
\end{cases}
\end{equation*}
One can easily check that for each $\varphi\in\Phi_c(\bR^d)$,
\begin{equation*}
\begin{gathered}
\tilde{\varphi}:=\cF^{-1}[\cF[\varphi](c^{-1}\cdot)+\cF[\varphi]+\cF[\varphi](c\cdot)]\in\Phi_{c^2}(\bR^d),
\end{gathered}
\end{equation*}
and
$$
\cF[\varphi](c^{-j}\cdot)=\cF[\varphi](c^{-j}\cdot)\cF[\tilde{\varphi}](c^{-j}\cdot),\quad \forall j\in\bN.
$$
Recall
$$
\psi^{\mu}(\xi)=\int_{\bR^d}(e^{iy\cdot\xi}-1-iy\cdot\xi 1_{|y|\leq 1}) \mu(dy).
$$
For $n\in\bN\cup\{0\}$ and $j\in\bN$, 
\begin{equation*}
\begin{aligned}
\cF[(-\psi^{\mu}(D))^nu(t,\cdot)\ast\varphi_j]&=(-\psi^{\mu})^n\exp\left(\int_{0}^t\Psi_{Z}(r,\cdot)dr\right)\cF[u_0]\cF[\varphi](c^{-j}\cdot)\\
	&=(\cF[\cI_1^j])^n\cF[\cI_2^j](t,\cdot)\cF[u_0]\cF[\varphi](c^{-j}\cdot),
\end{aligned}
\end{equation*}
where 
\begin{equation*}
\cI_1^j(x):=-\cF^{-1}[\psi^{\mu}\cF[\tilde{\varphi}](c^{-j}\cdot)](x)=-\sum_{i=-1}^{1}c^{(j+i)d}\cF^{-1}[\psi^{\mu}(c^{j+i}\cdot)\cF[\varphi]](c^{j+i}x)
\end{equation*}
and
\begin{equation*}
\begin{aligned}
\cI_2^j(t,x)&:=\cF^{-1}\left[\exp\left(\int_{0}^{t}\Psi_{Z}(r,\cdot)dr\right)\cF[\tilde{\varphi}](c^{-j}\cdot)\right](x)\\
    &=\sum_{i=-1}^1c^{(j+i)d}\cF^{-1}\left[\exp\left(\int_{0}^{t}\Psi_{Z}(r,c^{j+i}\cdot)dr\right)\cF[\varphi]\right](c^{j+i}x).
\end{aligned}
\end{equation*}
Recalling $c \geq 2$ and applying Assumptions \ref{22.06.26.18.46} and \ref{22.06.26.18.47}, we have
\begin{equation*}
\begin{aligned}
\int_{\bR^d}|\cI_1^j(x)|dx&=N\sum_{i=-1}^{1}\int_{\bR^d}\left|\cF^{-1}[\psi^{\mu}(c^{j+i}\cdot)\cF[\varphi]](x)\right|dx\\
&\leq N\sum_{i=-1}^{1}s(c^{-j-i})^{-1}\leq Ns(c^{-j})^{-1},\\
\int_{\bR^d}|\cI_2^j(t,x)|dx&=\sum_{i=-1}^{1}\int_{\bR^d}\left|\cF^{-1}\left[\exp\left(\int_{0}^{t}\Psi_{Z}(r,c^{j+i}\cdot)dr\right)\cF[\varphi]\right](x)\right|dx\\
&\leq Ne^{-N's(c^{-j})^{-1}t}.
\end{aligned}
\end{equation*}
Similarly,
\begin{equation*}
\begin{aligned}
&\cF[(-\psi^{\mu}(D))^nu(t,\cdot)\ast\varphi_0]=(-\psi^{\mu})^n\exp\left(\int_{0}^t\Psi_{Z}(r,\cdot)dr\right)\cF[u_0]\cF[\varphi_0]\\
&=(-\psi^{\mu}\cF[\tilde{\varphi}_0])^n\exp\left(\int_{0}^t\Psi_{Z}(r,\cdot)dr\right)\cF[\tilde{\varphi}_0]\cF[u_0]\cF[\varphi_0]\\
&=(\cF[\cI_1^0])^n\cF[\cI_2^0](t,\cdot)\cF[u_0]\cF[\varphi_0],
\end{aligned}
\end{equation*}
where
\begin{equation*}
\begin{gathered}
\cI_1^0(x):=-\cF^{-1}[\psi^{\mu}\cF[\tilde{\varphi}_0]](x)~\text{and}~ \cI_2^0(t,x):=\cF^{-1}\left[\exp\left(\int_{0}^t\Psi_{Z}(r,\cdot)dr\right)\cF[\tilde{\varphi}_0]\right](x)=\bE[\tilde{\varphi}_0(x+Z_t)].
\end{gathered}
\end{equation*}
By Lemma \ref{22.07.13.00.07},
$$
\int_{\bR^d}|\cI_1^0(x)|dx+
\int_{\bR^d}|\cI_2^0(t,x)|dx\leq N.
$$
Hence, by Young's convolution inequality, there exists a positive constant $N$ such that for all $n\in\bN\cup\{0\}$,
\begin{equation*}
\begin{aligned}
&\|(-\psi^{\mu}(D))^nu(t,\cdot)\ast\varphi_j\|_{L_p(\bR^d)}\\
&\leq N\left(s(c^{-j})^{-n}e^{-N's(c^{-j})^{-1}t}1_{j\in\bN}+1_{j=0}\right)\|u_0\ast\varphi_j\|_{L_p(\bR^d)},
\end{aligned}
\end{equation*}
where $N$ depends only on $d,n,s,\tilde{\varphi},\tilde{\varphi}_0$ and constants in Assumptions \ref{22.06.26.18.46} and \ref{22.06.26.18.47}.
In addition, by Minkowski's inequality,
\begin{equation*}
\begin{aligned}
\int_{0}^{T}\|(-\psi^{\mu}(D))^nu(t,\cdot)\|_{L_p(\bR^d)}^qdt&=\int_{0}^{T}\left\|\sum_{j=0}^{\infty}(-\psi^{\mu}(D))^n(u(t,\cdot)\ast\varphi_j)\right\|_{L_p(\bR^d)}^{q} dt\\
&\leq N\int_{0}^{T}\left(\sum_{j=0}^{\infty}\|(-\psi^{\mu}(D))^nu(t,\cdot)\ast\varphi_j\|_{L_p(\bR^d)}\right)^{q}dt\\
&\leq N \|u_0\ast\varphi_0\|_{L_p(\bR^d)}^q+N\cI(q)\\
&\leq N \|u_0\|_{B_{p,q}^{s;2n-\frac{2}{q}}}^q+N\cI(q),
\end{aligned}
\end{equation*}
where 
$$
\cI(q)=\int_{0}^{T}\left(\sum_{j=1}^{\infty}s(c^{-j})^{-n}e^{-N's(c^{-j})^{-1}t}\|u_0\ast\varphi_j\|_{L_p(\bR^d)}\right)^{q}dt
$$
and $N$ is independent of $u$ and $u_0$ but depends on $T$.
By \eqref{2020082402} in Proposition \ref{poly}, we have
\begin{align}
								\label{2020082410}
 N c^{ j\theta_0}  \leq s(c^{-j})^{-1}  \leq N^{-1} c^{ j\theta_1}   \qquad \forall  j \geq 1.
\end{align}
Thus if $0<q \leq 1$, then by \eqref{2020082410}  and Fubini's theorem,
\begin{align*}
    \cI(q)
    &\leq\sum_{j=1}^{\infty}s(c^{-j})^{-nq}\|u_0\ast\varphi_j\|_{L_p(\bR^d)}^{q}\int_{0}^Te^{-qN's(c^{-j})^{-1}t}dt \\
&\leq N\|u_0\|_{B_{p,q}^{2n-\frac{2}{q}}(\bR^d)},
\end{align*}
where $N$ is independent of $u_0$. 
On the other hand, if $q>1$ then we divide the computations into three parts.
Recalling $\theta_0=\log_{c}(s_L(c))$, $\theta_1=\log_{c}(s_U(c))$, we set
\begin{equation*}
    \begin{gathered}
    \cI_1(q):=\int_{0}^{1}\left(\sum_{j\in J_1(t)}s(c^{-j})^{-n}e^{-N's(c^{-j})^{-1}t}\|u_0\ast\varphi_j\|_{L_p(\bR^d)}\right)^{q}dt,\\
    \cI_2(q):=\int_{0}^{1}\left(\sum_{j\in J_2(t)}s(c^{-j})^{-n}e^{-N's(c^{-j})^{-1}t}\|u_0\ast\varphi_j\|_{L_p(\bR^d)}\right)^{q}dt,\\
    \cI_3(q):=1_{T>1}\int_{1}^{T}\left(\sum_{j=1}^{\infty}s(c^{-j})^{-n}e^{-N's(c^{-j})^{-1}t}\|u_0\ast\varphi_j\|_{L_p(\bR^d)}\right)^{q}dt,
    \end{gathered}
\end{equation*}
where
\begin{equation*}
    J_1(t):=\{j\in\bN:c^{j\theta_1}t\leq1\}~ \text{and}~ J_2(t):=\{j\in\bN:c^{j\theta_1}t\geq 1\}.
\end{equation*}
Then it is obvious that
\begin{equation*}
    J_1(t)\cup J_2(t)=\bN \qquad \forall t >0
\end{equation*}
and
$$
\cI(q)\leq N(\cI_1(q)+\cI_2(q)+\cI_3(q))
$$
with a positive constant $N$ depending only on $q$. 
First, we estimate $\cI_1(q)$. 
Applying H\"older's inequality to $\cI_1(q)$ with a $a\in(0,1)$,  we have
\begin{align*}
\cI_1(q)
&=\int_{0}^{1}\left(\sum_{j\in J_1(t)}s(c^{-j})^{-n}e^{-N's(c^{-j})^{-1}t}\|u_0\ast\varphi_j\|_{L_p(\bR^d)}\right)^{q}dt  \\
&=\int_{0}^{1}\left(\sum_{j\in J_1(t)} c^{j\theta_1 a/q  } c^{-j\theta_1 a/q} s(c^{-j})^{-n}e^{-N's(c^{-j})^{-1}t}\|u_0\ast\varphi_j\|_{L_p(\bR^d)}\right)^{q}dt  \\
&\leq\int_{0}^{1} U_1(t)^{q-1}U_2(t)dt,
\end{align*}
where
\begin{equation*}
    U_1(t):=\sum_{j\in J_1(t)}c^{\frac{j\theta_1a}{q-1}}\quad \text{and}\quad U_2(t):=\sum_{j\in J _1(t)}s(c^{-j})^{-nq}\|u_0\ast\varphi_j\|_{L_p(\bR^d)}^qc^{-j\theta_1a}.
\end{equation*}
For each $t>0$, put $j_1(t)=\sup J_1(t)$ and note that $j_1(t) \in J_1(t) \subset\bN$ since $J_1(t)$ is a finite set.
Then since all terms are nonnegative in the summation and $c^{j_1(t)\theta_1 }\leq t^{-1}$, we have
$$
U_1(t)= \sum_{j\in J_1(t)}c^{\frac{j\theta_1a}{q-1}}\leq \sum_{j=-\infty}^{j_1(t)}c^{\frac{j\theta_1a}{q-1}} = \frac{c^{\frac{j_1(t)\theta_1a}{q-1}}}{1-c^{\frac{-\theta_1a}{q-1}}}\leq N(q,c,a,\theta_1)t^{-\frac{a}{q-1}}.
$$
Moreover, by Fubini's theorem and \eqref{2020082410},
\begin{equation*}
\begin{aligned}
\cI_1(q)\leq\int_{0}^1U_1(t)^{q-1}U_2(t)dt &\leq N\int_{0}^1\sum_{j\in J_1(t)}s(c^{-j})^{-nq}\|u_0\ast\varphi_j\|_{L_p(\bR^d)}^qc^{-j\theta_1a}t^{-a}dt\\
&\leq N\sum_{j=1}^{\infty}s(c^{-j})^{-nq}\|u_0\ast\varphi_j\|_{L_p(\bR^d)}^qc^{-aj\theta_1}\int_{0}^{c^{-j\theta_1}}t^{-a}dt\\
&=N\sum_{j=1}^{\infty}s(c^{-j})^{-nq}\|u_0\ast\varphi_j\|_{L_p(\bR^d)}^qc^{-j\theta_1}\\
&\leq N\sum_{j=1}^{\infty}s(c^{-j})^{-nq+1}\|u_0\ast\varphi_j\|_{L_p(\bR^d)}^p \leq N \|u_0\|_{B_{p,q}^{s;2n-\frac{2}{q}}}^q,
\end{aligned}
\end{equation*}
where $N$ is independent of $u_0$ and $T$. In particular, we can take $a=1/2$ in the above estimates. 
Second, we estimate $\cI_2(q)$. Using H\"older's inequality to $\cI_2(q)$ with a $b \in (0,1)$, we obtain
\begin{equation*}
    \cI_2(q)=\int_{0}^{1}\left(\sum_{j\in J_2(t)}s(c^{-j})^{-n}e^{-N's(c^{-j})^{-1}t}\|u_0\ast\varphi_j\|_{L_p(\bR^d)}\right)^{q}dt\leq \int_{0}^{1}U_3(t)^{q-1}U_4(t)dt,
\end{equation*}
where
\begin{equation*}
    \begin{gathered}
    U_3(t):=\sum_{j\in J_2(t)}c^{\frac{j\theta_0b}{q-1}}\exp(-N'(q-1)^{-1}s(c^{-j})^{-1}t),\\
    U_4(t):=\sum_{j\in J_2(t)}s(c^{-j})^{-nq}e^{-qN's(c^{-j})^{-1}t}\|u_0\ast\varphi_j\|_{L_p(\bR^d)}^qc^{-j\theta_0b}.
    \end{gathered}
\end{equation*}
For all $t \in (0,1)$, applying \eqref{2020082410} and changing variables, we have
\begin{equation*}
    \begin{aligned}
    U_3(t)&\leq\sum_{j\in J_2(t)}c^{\frac{j\theta_0b}{q-1}}\exp(-Nc^{j\theta_0}t)\leq N'\int_{0}^{\infty}1_{tc^{\lambda \theta_1}> 1}c^{\frac{\lambda\theta_0b}{q-1}}\exp(-Nc^{\lambda\theta_0}t)d\lambda\\
    &= N'\int_{t^{-\theta_0/\theta_1}}^{\infty}\lambda^{\frac{b}{q-1}}e^{-N\lambda t}\lambda^{-1}d\lambda=N't^{-\frac{b}{q-1}}\int_{t^{1-\frac{\theta_0}{\theta_1}}}^{\infty}\lambda^{\frac{b}{q-1}-1}e^{-N\lambda}d\lambda\\
    &\leq N't^{-\frac{b}{q-1}},
    \end{aligned}
\end{equation*}
where $N$ and $N'$ are independent of $t$.
By Fubini's theorem,
\begin{equation*}
    \begin{aligned}
     \cI_2(q)\leq\int_{0}^{1}U_3(t)^{q-1}U_4(t)dt&\leq N\int_{0}^1\sum_{j=1}^{\infty}s(c^{-j})^{-nq}t^{-b}e^{-qN's(c^{-j})^{-1}t}\|u_0\ast\varphi_j\|_{L_p(\bR^d)}^qdt\\
     &\leq N\sum_{j=1}^{\infty}s(c^{-j})^{-nq}s(c^{-j})^{1+b}\|u_0\ast\varphi_j\|_{L_p(\bR^d)}^q\int_{0}^{\infty}t^{-b}e^{-Nt}dt\\
     &\leq N\sum_{j=1}^{\infty}s(c^{-j})^{-\frac{(2n)q}{2}}s(c^{-j})^{-\frac{(-2/q) \cdot q}{2} }\|u_0\ast\varphi_j\|_{L_p(\bR^d)}^q\int_{0}^{\infty}t^{-b}e^{-Nt}dt\\
     &\leq N \|u_0\|_{B_{p,q}^{s;2n-\frac{2}{q}}(\bR^d)}^q,
    \end{aligned}
\end{equation*}
where $N$ is independent of $u_0$ and $T$. Particularly, we can take $b=1/2$ in the above estimates.
Finally, it only remains to compute the last term $\cI_3(q)$. If $T>1$, then $\cI_3(q) \neq 0$.
Moreover, by H\"older's inequality,
\begin{align}
							\notag
    \cI_3(q)
    &=\int_{1}^{T}\left(\sum_{j=1}^{\infty}s(c^{-j})^{-n}e^{-N's(c^{-j})^{-1}t}\|u_0\ast\varphi_j\|_{L_p(\bR^d)}\right)^{q}dt\\
								\label{2020082430}
    &\leq  \sum_{j=1}^{\infty} s(c^{-j})^{-n q} \|u_0\ast\varphi_j\|_{L_p(\bR^d)}^q \int_{1}^{T}  U_5(t)^{q-1}\cdot  e^{-N's(c^{-j})^{-1}t}  dt,
\end{align}
where 
 $$
U_5(t)=\sum_{j=1}^{\infty}\exp\left(- N' s(c^{-j}_s)^{-1}t\right).
 $$
Note that by \eqref{2020082410}, 
$$
U_5(t) \leq \sum_{j=1}^{\infty}\exp\left(- N' c^{j \theta} \right)  <\infty \qquad \forall t \geq 1.
$$
and thus
\begin{align}
								\label{2020082431}
\int_{1}^{T}  U_5(t)^{q-1}\cdot  e^{-N's(c^{-j})^{-1}t}  dt
\leq Ns(c^{-j}).
\end{align}
Putting \eqref{2020082431} in \eqref{2020082430}, we have
\begin{align*}
    \cI_3(q)
    &\leq N\sum_{j=1}^{\infty} s(c^{-j})^{-n q+1} \|u_0\ast\varphi_j\|_{L_p(\bR^d)}^q= N \|u_0\|_{B_{p,q}^{s;2n-\frac{2}{q}}(\bR^d)}^q.
\end{align*}
Therefore, combining all estimates above, we obtain
\begin{equation*}
\|(-\psi^{\mu}(D))^nu\|_{L_q((0,T);L_p(\bR^d))}\leq N\|u_0\|_{B_{p,q}^{s;2n-\frac{2}{q}}(\bR^d)}, \quad \forall n\in\bN\cup\{0\},
\end{equation*}
where $N$ is independent of $u_0$.
In particular,
\begin{align*}
&\|u\|_{L_q((0,T);L_p(\bR^d))} +\|(-\psi^{\mu}(D))^nu\|_{L_q((0,T);L_p(\bR^d))}\\
&\leq N \left(\|u_0\|_{B_{p,q}^{s;-\frac{2}{q}}(\bR^d)} + \|u_0\|_{B_{p,q}^{s;2n-\frac{2}{q}}(\bR^d)}\right), \quad \forall n\in\bN\cup\{0\},
\end{align*}
Finally, applying Propositions \ref{properties}($v$) and \ref{20.04.29.20.24}($v$), we conclude that
$$
u=I_Z(u_0)\in L_q((0,T);H_p^{\mu;2n}(\bR^d))
$$
and
\begin{equation*}
\|u\|_{L_q((0,T);H_p^{\mu;2n}(\bR^d))}\leq N \|u_0\|_{B_{p,q}^{s;2n-\frac{2}{q}}(\bR^d)},
\end{equation*}
where $N$ is independent of $u$ and $u_0$. The lemma is proved.
\end{proof}
Next, we consider a general order by applying a complex interpolation theorem. 
Since Calder\'on's complex interpolation theorem is based on a duality property of $L_q$-spaces, the range of $q$ is restricted to $q  \in [1,\infty)$.
\begin{corollary}[A priori estimate]
\label{20.08.20.14.30}
Let $p\in(1,\infty)$, $q\in[1,\infty)$, $T\in(0,\infty)$, $\gamma\in (1,\infty)$, $\mu\in\frL_d$, $\boldsymbol{s}=(s,s_L,s_U)$ be a scaling triple, and $Z$ be an additive process with a bounded triplet $(a(t),0,\Lambda_t)_{t\geq0}$. If Assumptions \ref{22.06.26.18.46}  and \ref{22.06.26.18.47} hold, then for any $u_0\in C_c^{\infty}(\bR^d)$,
\begin{equation}
    \|(-\psi^{\mu}(D))^{\gamma/2}I_Z(u_0)\|_{L_q((0,T);L_p(\bR^d))}\leq N\|u_0\|_{B_{p,q}^{s;\gamma-\frac{2}{q}}(\bR^d)},
\end{equation}
where $N$ is independent of $u_0,T$ and
$$
I_Z(u_0)(t,x):=\bE[u_0(x+Z_t)].
$$
\end{corollary}
\begin{proof}
We apply the complex interpolation theorem to Lemma \ref{a priori lem}.
For $\gamma\in(0,\infty)\setminus\bN$, there exist $m\in\bN\cup\{0\}$ and $\theta\in(0,1)$ such that
$$\gamma=2m(1-\theta)+(2m+2)\theta.$$
Therefore, by Proposition \ref{properties} $(ix)$, \cite[Theorem 5.1.2]{bergh2012interpolation}, \cite[Theorem C.2.6]{hytonen2016analysis}, and Proposition \ref{20.04.29.20.24} $(viii)$,
\begin{equation*}
\begin{aligned}
\|I_Z(u_0)\|_{L_q((0,T);H_p^{\mu;\gamma}(\bR^d))}&:=\|u\|_{L_q((0,T);H_p^{\mu;\gamma}(\bR^d))}\\
&\approx\|u\|_{L_q((0,T);[H_p^{\mu;2m}(\bR^d),H_p^{\mu;2m+2}(\bR^d)]_{\theta})}\\
&\leq N \|u_0\|_{[B_{p,q}^{s;2m-\frac{2}{q}}(\bR^d),B_{p,q}^{s;2m+2-\frac{2}{q}}(\bR^d)]_{\theta}}\\
&\approx N \|u_0\|_{B_{p,q}^{s;\gamma-\frac{2}{q}}(\bR^d)},
\end{aligned}
\end{equation*}
where $N$ is independent of $u$ and $u_0$. The corollary is proved.
\end{proof}

\textbf{Proof of Theorem \ref{main1}}. 
Due to Corollary \ref{20.06.06.16.06}, it is sufficient to prove the existence of a solution and the a priori estimate \eqref{2020082301}.
By Proposition \ref{20.04.29.20.24}(iii), there exist sequence $h_l\in C_c^{\infty}(\bR^d)$ such that
\begin{equation*}
    \begin{gathered}
    h_l\to u_0,\quad\text{in } B_{p,q}^{s;\gamma-\frac{2}{q}}(\bR^d),
    \end{gathered}
\end{equation*}
as $l\to\infty$. By Theorem \ref{20.05.04.10.58}, for each $l$, the IVP
\begin{equation}
\label{20.07.16.10.05}
\begin{cases}
\frac{\p u_l}{\p t}(t,x)=\cA_Z(t)u_l(t,x),\quad &(t,x)\in(0,T)\times\bR^d,\\
u_l(0,x)=h_l(x),\quad & x\in\bR^d,
\end{cases}
\end{equation}
has a unique solution $u_l = I_Z(h_l)\in \tilde{C}^{1,\infty}([0,T]\times\bR^d)$, where
$$
I_Z(u_0)(t,x)=\bE[h_l(x+Z_t)].
$$
By Corollary \ref{20.08.20.14.30}, we have
\begin{equation}
\label{20.07.16.10.06}
\|u_l\|_{L_q((0,T);H_p^{\mu;\gamma}(\bR^d))}\leq N\|h_l\|_{B_{p,q}^{s;\gamma-\frac{2}{q}}(\bR^d)},
\end{equation}
where $N$ is independent of $l$ and $T$. By the linearity of \eqref{20.07.16.10.05} and \eqref{20.07.16.10.06}, $\{u_l\}_{l=1}^{\infty}$ is a Cauchy sequence in $L_q((0,T);H_p^{\mu;\gamma}(\bR^d))$. 
Thus as a result of Proposition \ref{properties}(iii), there exists a $u\in L_q((0,T);H_p^{\mu;\gamma}(\bR^d))$ such that
$$
u_l\to u,\quad \text{in } L_q((0,T);H_p^{\mu;\gamma}(\bR^d))
$$
as $l\to\infty$. 
Taking $l \to \infty$ in \eqref{20.07.16.10.06}, we have
\begin{equation*}
\|u\|_{L_q((0,T);H_p^{\mu;\gamma}(\bR^d))}\leq N\|u_{0}\|_{B_{p,q}^{s;\gamma-\frac{2}{q}}(\bR^d)}.
\end{equation*}
Moreover, by Definition \ref{20.06.06.15.50}, one can easily check that $u$ becomes a solution to the Cauchy problem
\begin{equation*}
\begin{cases}
\frac{\p u}{\p t}(t,x)=\cA_Z(t)u(t,x),\quad &(t,x)\in(0,T)\times\bR^d,\\
u(0,x)=u_{0}(x),\quad & x\in\bR^d.
\end{cases}
\end{equation*}
The theorem is proved.

\mysection{Appendix}
			\label{20.08.20.15.44}
In this appendix, we prove Propositions \ref{22.07.07.16.05}, \ref{22.07.07.16.05-2}, \ref{22.07.07.16.05-3}, \ref{properties}, and  \ref{20.04.29.20.24}. We split the proofs into three subsections.

\subsection{Proof of Propositions \ref{22.07.07.16.05}, \ref{22.07.07.16.05-2}, and \ref{22.07.07.16.05-3}}
\label{22.07.07.14.36}

First we prove that L\'evy measures are closed under a scaling. 
Recall that $\frL_d$ is a set of L\'evy measures on $\bR^d$ and for $\mu\in\frL_d$ and $c>0$,
$$
\mu^c(dy):=\mu(c\,dy) \quad \text{and} \quad L(\mu):=\int_{\bR^d}(1\wedge|y|^2)\mu(dy).
$$
\begin{prop}
\label{20.04.07.14.35}
If $\mu\in\frL_d$ and $c>0$, then $\mu^c\in\frL_d$.
\end{prop}
\begin{proof}
If suffices to show that $L(\mu^c)$ is finite.
\begin{equation*}
\begin{aligned}
L(\mu^c)&=\int_{\bR^d}\left(1\wedge\frac{|x|^2}{c^2}\right)\mu(dx)\\
	&\leq\int_{\bR^d}(1\wedge|x|^2)1_{c\geq1}+\left[\frac{|x|^2}{c^2}1_{|x|<1}+1_{|x|\geq1}\right]1_{0<c<1}\mu(dx)\\
	&\leq L(\mu)\left(1_{c\geq1}+\frac{1+c^2}{c^2}1_{0<c<1}\right)<\infty.
\end{aligned}
\end{equation*}
The proposition is proved.
\end{proof}
Recall that for an additive process $Z$ with a bounded triplet  $(a(t),A(t),\Lambda_t)_{t\geq0}$,
$$
\Psi_Z(t,\xi):=ia(t)\cdot\xi-\frac{1}{2}(A(t)\xi\cdot\xi)+\int_{\bR^d}(e^{iy\cdot\xi}-1-iy\cdot\xi1_{|y|\leq1})\Lambda_t(dy).
$$
On the other hand, if a bounded triplet $(a(t),A(t),\Lambda_t)_{t\geq0}$ is given, then there exists a an additive process $Z$ 
(\cite[Theorem 9.8]{sato1999levy}) such that
\begin{equation*}
\begin{gathered}
\bE[e^{i\xi\cdot (Z_t-Z_s)}]=e^{\int_s^t\Psi_Z(r,\xi)dr},\\
\end{gathered}
\end{equation*}
for all $0\leq s<t<\infty$ and $\xi\in\bR^d$.
\begin{prop}
					\label{20.05.03.19.17}
Let $Z$ be an additive process with a bounded triplet  $(a(t),A(t),\Lambda_t)_{t\geq0}$.
For each $c\in(0,\infty)$, define
\begin{equation}
\label{20.04.28.14.40}
a^{c}_{\Lambda}(t):=\frac{1}{c}\left(a(t)-1_{c\in(0,1)}\int_{c<|y|\leq1}y\Lambda_t(dy)+1_{c\in(1,\infty)}\int_{1<|y|\leq c}y\Lambda_t(dy)\right),~\text{and}~ A^c(t):=\frac{1}{c^2}A(t).
\end{equation}
Then there exists an additive process $Z^c$ with the bounded triplet $(a^{c}_{\Lambda}(t),A^c(t),\Lambda_t^c)_{t\geq0}$.
Moreover, for all $t\geq0$ and $\xi\in\bR^d$,
\begin{equation}
\label{22.07.09.17.25}
\Psi_Z(t,\xi/c)=\Psi_{Z^c}(t,\xi). 
\end{equation}
\end{prop}
\begin{proof}
By proposition \ref{20.04.07.14.35}, $\Lambda_t^c$ is a L\'evy measure on $\bR^d$ for all $t \geq 0$ and $c>0$.
Thus \cite[Theorem 9.8]{sato1999levy} guarantees the existence of an additive process $Z^c$ with bounded the triplet $(a^{c}_{\Lambda}(t),A^c(t),\Lambda^c_t)_{t\geq0}$.
Moreover, by \eqref{20.04.28.14.40}, it is easy to check that 
\begin{equation*}
\begin{aligned}
&\Psi_{Z^c}(t,\xi)=ia^{c}_{\Lambda}(t)\cdot\xi-\frac{1}{2}(A^c(t)\xi\cdot\xi)+\int_{\bR^d}(e^{iy\cdot\xi}-1-iy\cdot\xi1_{|y|\leq1})\Lambda^c_t(dy)\\
	&\quad=\frac{i}{c}\left(a(t)-1_{c\in(0,1)}\int_{c<|y|\leq1}y\Lambda_t(dy)+1_{c\in(1,\infty)}\int_{1<|y|\leq c}y\Lambda_t(dy)\right)\cdot\xi-\frac{1}{2c^2}(A(t)\xi\cdot\xi)\\
	&\quad \quad+\int_{\bR^d}(e^{iy\cdot\xi/c}-1-ic^{-1}y\cdot\xi1_{|y|\leq1})\Lambda_t(dy)+\frac{i}{c}\left(1_{c\in(0,1)}\int_{c<|y|\leq1}y\Lambda_t(dy)-1_{c\in(1,\infty)}\int_{1<|y|\leq c}y\Lambda_t(dy)\right)\cdot\xi\\
	&\quad=ia(t)\cdot\frac{\xi}{c}-\frac{1}{2}\left(A(t)\frac{\xi}{c}\right)\cdot\frac{\xi}{c}+\int_{\bR^d}\left(e^{iy\cdot\frac{\xi}{c}}-1-iy\cdot\frac{\xi}{c}1_{|y|\leq1}\right)\Lambda_t(dy)=\Psi_Z(t,\xi/c).
\end{aligned}
\end{equation*}
The proposition is proved.
\end{proof}

\begin{lem}
					\label{polynomial}
Let $\nu$ be a L\'evy measure on $\bR^d$, and $c_1$ be a positive constant. 
Assume that $c_1 \geq1$ and
\begin{equation}
					\label{eqn 20220718 21}
    \inf_{|\xi|=1}\int_{|y|\leq N_\nu }|y\cdot\xi|^2\nu(dy) >0
\end{equation}
for a  constant $N_\nu \in (0,\infty)$. 
Then there exists a positive constant $N$  such that for all $\xi\in B_{c_1}$,
$$
\int_{\bR^d}(1-\cos(y\cdot\xi))\nu_{c_1,N_\nu}( dy)
\geq N |\xi|^2,
$$
where $\nu_{c_1,N_\nu}(dy) :=  1_{|y| \leq 1} \nu( N_\nu c_1\, dy)$.
\end{lem}
\begin{proof}
Without loss of generality, we assume that $\xi\in B_{c_1}\setminus\{0\}$. Note that 
\begin{align}
							\label{eqn 20220718 01}
1-\cos(y\cdot\xi) \geq 0 \qquad \forall \xi~\text{and}~ \forall y
\end{align}
and
\begin{equation}
							\label{cosine}
1-\cos(\lambda)\geq\frac{1}{4}\lambda^2 \quad \forall \lambda \in [-1,1].
\end{equation}
Thus by \eqref{eqn 20220718 01}, \eqref{cosine}, and \eqref{eqn 20220718 21},
\begin{align*}
\int_{ \bR^d }(1-\cos(y\cdot\xi))\nu_{c_1,N_\nu}(dy)
&= \int_{|y|\leq 1 }(1-\cos(y\cdot\xi))\nu(N_\nu c_1\, dy) \\
&= \int_{|y|\leq N_\nu c_1 }\left(1-\cos\left(  \frac{ y\cdot\xi}{N_\nu c_1} \right) \right) \nu(dy) \\
& \geq \int_{  |y|\leq N_\nu }\left(1-\cos\left(  \frac{ y\cdot\xi}{N_\nu c_1} \right) \right)\nu(  dy) \\
&\geq N\int_{|y|\leq N_\nu}|y\cdot\xi|^2\nu(dy) \\
&\geq  N |\xi|^2.
\end{align*}
The lemma is proved.
\end{proof}

\begin{lem}
						\label{lem 20220715 01}
Let $s$ be a scaling function with $R_0=\infty$ and $\nu$ be a L\'evy measure on $\bR^d$. 
Assume that 
\begin{equation}
							\label{eqn 20220718 20}
    \inf_{r \in (0, \infty),|\xi|=c} s(r)\int_{|y|\leq N_\nu}(1-\cos(y \cdot \xi))\nu(r\, dy) >0,
\end{equation}
where $c$ is a positive constant. 
Then there exists a positive constant $N$  such that 
for all $\xi\in B_{c_s}\setminus B_{c_s^{-1}}$ and $r \in (0,\infty)$,
$$
s(r)\int_{\bR^d}(1-\cos(y\cdot\xi))\nu_{c,c_s,N_{\nu}}( r\, dy)
\geq N |\xi|^2,
$$
where $\nu_{c,c_s,N_{\nu}}(dy) =  1_{|y| \leq  c \cdot c_s \cdot N_\nu}(y)\nu(dy)$.
\end{lem}
\begin{proof}
Applying  \eqref{eqn 20220718 01} and \eqref{eqn 20220718 20}, we have
\begin{align*}
s(r c |\xi|^{-1})  \int_{\bR^d}(1-\cos(y\cdot\xi))\nu_{c,c_s,N_{\nu}} (r\, dy)
&=s(r c |\xi|^{-1})  \int_{ |y| \leq   c \cdot c_s \cdot N_\nu  }(1-\cos(y\cdot\xi))\nu(r\, dy) \\
 &= s(r c |\xi|^{-1})\int_{  c|\xi|^{-1} |y| \leq  c \cdot c_s \cdot N_\nu  }\left(1-\cos\left(y\cdot \frac{c \xi}{|\xi|} \right)\right)\nu (r c|\xi|^{-1}\, dy) \\
 & \geq s(r c|\xi|^{-1}) \int_{|y| \leq   N_\nu }\left(1-\cos\left(y\cdot \frac{c \xi}{ |\xi|} \right) \right)\nu (r c|\xi|^{-1}\, dy) \\
 &\geq N .
\end{align*}
Moreover, by the definition of the scaling function (Definition \ref{20.05.28.13.20}), there exists a positive constant $N$ such that
\begin{align*}
N^{-1} \leq \frac{s(r)}{s(r c |\xi|^{-1})}   \leq N. 
\end{align*}
Therefore, the lemma is proved. 
\end{proof}

\begin{lem}
						\label{smoothness}
Let $\mu\in\frL_d$ and $R>0$.
\begin{enumerate}[(i)]
\item
The function
$$
h_R(\xi):=\int_{|y|\leq R}(e^{iy\cdot\xi}-1-iy\cdot\xi 1_{|y|\leq 1})\mu(dy)
$$
is infinitely differentiable for all $\xi \in \bR^d$. 
Furthermore, for any multi-index $\alpha$,
$$
|D^{\alpha}h_R(\xi)|\leq N(1_{|\alpha|=1}(1+|\xi|)+1_{|\alpha|\geq2}),
$$
where $N=N(L(\mu),R)$.
\item
The function
$$
g_R(\xi):=\int_{|y|>R}(e^{iy\cdot\xi}-1-iy\cdot\xi 1_{|y|\leq 1})\mu(dy)
$$
is uniformly continuous and bounded for all $\xi \in \bR^d$.
\end{enumerate}
\end{lem}
\begin{proof}
($i$) Denote
$$
f(\lambda):=e^{i\lambda}-1-i\lambda.
$$
Then, obviously, $f$ is infinitely differentiable.
Thus, by the fundamental theorem of calculus,
\begin{equation}
\label{ftc}
f(y\cdot(\xi_1+\xi_2))-f(y\cdot\xi_1)=\int_0^1f'(y\cdot(\xi_1+\theta\xi_2))(y\cdot\xi_2)d\theta
\end{equation}
for all $\xi_1,\xi_2,y \in \bR^d$. 
Moreover, by Taylor's theorem,
\begin{equation}
\label{bound of f}
 |D^n f(\lambda)|\leq\frac{\lambda^2}{2}1_{n=0}+\lambda1_{n=1}+1_{n\geq2}
\end{equation}
for all $\lambda \in \bR$ and nonnegative integer $n$.
It is obvious that
$$
h_R(\xi)=\int_{|y|\leq R}f(y\cdot\xi)\mu(dy)+ 1_{R>1} i\int_{1<|y|\leq R}(y\cdot\xi)\mu(dy)
$$
and the function $1_{R>1} i\int_{1<|y|\leq R}(y\cdot\xi)\mu(dy)$ is infinitely differentiable for all $\xi$.
Thus it only remains to prove that the mapping
$$
\int_{|y|\leq R}f(y\cdot\xi)\mu(dy)
$$
is infinitely differentiable for all $\xi$. By \eqref{ftc} and \eqref{bound of f},
$$
\left|\frac{f(y\cdot(\xi+\varepsilon e_j))-f(y\cdot\xi)}{\varepsilon}\right|\leq|y|\int_{0}^1|Df(y\cdot(\xi+\theta \varepsilon e_j))|d\theta\leq(|\xi|+|\varepsilon|)|y|^2
$$
for all $\varepsilon > 0$ and $\xi,y \in \bR^d$.
Since $\mu$ is a L\'evy measure on $\bR^d$,  we  have
$$
\int_{|y|\leq R}|y|^2\mu(dy)<\infty.
$$
Thus the Lebesgue dominated convergence theorem yields
$$
|D_{\xi^j}h_R(\xi)|\leq |\xi|\int_{|y|\leq R}|y|^2\mu(dy)+\int_{1<|y|\leq R}|y|\mu(dy)1_{R>1}\leq N(L(\mu),R)(1+|\xi|).
$$
Repeating the above argument for all $n\geq2$, the assertion ($i$) is proved.

($ii$) Note that for all $\xi_2 \in \bR^d$
\begin{equation*}
\sup_{\xi_1\in\bR^d}|g_R(\xi_1+\xi_2)-g_R(\xi_1)|\leq\int_{|y|> R}\left(|e^{iy\cdot\xi_2}-1|+|\xi_2||y| 1_{|y|\leq 1} \right)\mu(dy).
\end{equation*}
Therefore, by the Lebesgue dominated convergence theorem, we have
$$
\sup_{\xi_1\in\bR^d}|g_R(\xi_1+\xi_2)-g_R(\xi_1)| \to 0 \quad \text{as} \quad \xi_2 \to 0.
$$
The lemma is proved.
\end{proof}

\begin{lem}
\label{smoothness-2}
Let $s$ be a scaling function with $R_0=\infty$ and $\nu$ be a L\'evy measure on $\bR^d$. 
Assume that 
\begin{align}
\sup_{r \in (0,\infty)}\left[ s(r)\int_{\bR^d}(1\wedge|r^{-1}y|^2)\nu(dy) \right]
= \sup_{r \in (0,\infty)} \left[ s(r)\int_{\bR^d}(1\wedge|y|^2)\nu(r\,dy) \right] < \infty.
\end{align}
Then for all $r,R \in (0,\infty)$,
\begin{enumerate}[(i)]
\item
The function
$$
h_{r,R}(\xi):= s(r)\int_{|y|\leq R}(e^{iy\cdot\xi}-1-iy\cdot\xi 1_{|y|\leq 1})\nu(r\, dy)
$$
is infinitely differentiable for all $\xi \in \bR^d$. 
Furthermore, for any multi-index $\alpha$,
$$
|D^{\alpha}h_{r,R}(\xi)|\leq N(1_{|\alpha|=1}(1+|\xi|)+1_{|\alpha|\geq2}),
$$
where 
$$
L(r,\nu):= s(r)\int_{\bR^d}(1\wedge|y|^2)\nu(r\,dy) 
$$
and $N=N(L(r,\nu),R)$, i.e. $N$ is independent of $r$.
\item
The function
$$
g_{r,R}(\xi):= s(r)\int_{|y|>R}(e^{iy\cdot\xi}-1-iy\cdot\xi 1_{|y|\leq 1})\nu(r\, dy)
$$
is uniformly continuous and bounded for all $\xi \in \bR^d$ uniformly for all $r \in (0,\infty)$, i.e.
$$
\sup_{r \in (0,\infty), \xi_1\in\bR^d}|g_{r,R}(\xi_1+\xi_2)-g_{r,R}(\xi_1)| \to 0 \quad \text{as} \quad \xi_2 \to 0.
$$
\end{enumerate}
\end{lem}
\begin{proof}
Keep in mind that $\left[ s(r)\int_{\bR^d}(1\wedge|y|^2)\nu(r\,dy) \right]$ is uniformly bounded for all $ r \in (0,\infty)$ 
by the assumption. 
Then the proof is almost identical to that of Lemma \ref{smoothness}.
\end{proof}

\vspace{3mm}
\textbf{Proof of Propositions \ref{22.07.07.16.05} and \ref{22.07.07.16.05-2}}\hspace{3mm} 
Let $\varphi \in \cS(\bR^d)$ so that $\cF[\varphi] \in C_c^\infty(B_{c_s} \setminus B_{c_s^{-1}} )$.
First, we prove Proposition \ref{22.07.07.16.05}.
For $c_1 \in [1,\infty)$, 
we put
$$
\nu_{N_\nu}(dy) := 1_{|y|\leq N_\nu} \nu(c_1\,dy),
$$
where $\nu\in\frL_d$ is taken from \eqref{20.08.11.14.28} and $c_1$ will be chosen later.
Due to the assumption in Proposition \ref{22.07.07.16.05} and Remark \ref{22.07.25.11.42}, for all $r\in(0,\infty)$ and $t\geq0$,
$$
\nu_{r,t}:=\Lambda_t^r-s(r)^{-1}\nu_{N_\nu}
$$
is a L\'evy measure on $\bR^d$.
By Proposition \ref{20.05.03.19.17}, there exist an additive process  $Y^r$ (resp. $Y^1$)  with the bounded triplet $(a^r_{\Lambda}(t),0,\nu_{r,t})_{t\geq0}$ (resp. $(a(t),0,\nu_{1,t})_{t\geq0}$). Let $\eta_t^r$ (resp. $\eta_t^1$) be the probability distribution of $Y_t^r$ (resp. $Y_t^1$). 
Observe that for all $r\in(0,\infty)$, $h\in[0,\infty)$, and $\xi\in\bR$,
\begin{align*}
    s(r)^{-1}\int_{\bR^d}(e^{iy\cdot\xi}-1-iy\cdot\xi1_{|y|\leq1})\nu_{N_\nu}(dy)&=:s(r)^{-1}\psi^{\nu_{N_\nu}}(\xi)\\
    &=\Psi_{Z^r}(h,\xi)-\Psi_{Y^{r}}(h,\xi)\\
    &=\Psi_Z(h,r^{-1}\xi)-\Psi_{Y^{r}}(h,\xi).
\end{align*}
Applying Fubini's theorem and \eqref{22.07.09.17.25}, for $r\in(0,1)$, we have
\begin{equation*}
\begin{aligned}
&\left|\cF^{-1}\left[e^{\int_{0}^{t}\Psi_{Z}(h,r^{-1}\cdot)dh}\cF[\varphi]\right](x)\right|\\
&=(2\pi)^{-d/2}\left|\int_{\bR^d}e^{ix\cdot\xi}\exp\left(\int_{0}^{t}\Psi_{Z^r}(h,\xi)dh\right)\cF[\varphi](\xi)d\xi\right|\\
&=(2\pi)^{-d/2}\left|\int_{\bR^d}e^{ix\cdot\xi}e^{s(r)^{-1}t\psi^{\nu_{N_\nu}}(\xi)}\cF[\varphi](\xi)\bE[e^{iY_{t}^r\cdot\xi}]d\xi\right|\\
&=(2\pi)^{-d/2}\left|\int_{\bR^d}e^{ix\cdot\xi}e^{s(r)^{-1}t\psi^{\nu_{N_\nu}}(\xi)}\cF[\varphi](\xi)\int_{\bR^d}e^{iy\cdot\xi}\eta_{t}^r(dy)d\xi\right|\\
&\leq\int_{\bR^d}|\cF^{-1}[e^{s(r)^{-1}t\psi^{\nu_{N_\nu}}}\cF[\varphi]](x+y)|\eta_{t}^r(dy).
\end{aligned}
\end{equation*}
By virtue of Minkowski's inequality, it suffices to prove that there exist positive constants $C_1$ and $C_2$ such that
\begin{equation}
			\label{20.07.16.13.42-2}
\begin{gathered}
\int_{\bR^d}\left|\cF^{-1}\left[e^{t\psi^{\nu_{N_\nu}}}\cF[\varphi]\right](x)\right|dx\leq C_1e^{-C_2t}.
\end{gathered}
\end{equation}
If we put $\nu_{c_1,N_\nu}(dy) :=  1_{|y| \leq 1} \nu( N_\nu  c_1\, dy)$, then $\nu_{N_\nu}( N_\nu \,dy)= \nu_{c_1,N_\nu}(  dy)$.
Thus by Proposition \ref{20.05.03.19.17}, we have 
$$
\psi^{\nu_{N_\nu}}(  \xi) =  \psi^{\nu_{c_1,N_\nu}}(N_\nu  \xi)-id_{\nu}\cdot\xi,
$$
where
$$
d_{\nu}:=1_{N_{\nu}>1}N_{\nu}\int_{1<|y|\leq N_{\nu}}y\,\nu(N_{\nu}c_1\,dy).
$$
Therefore instead of \eqref{20.07.16.13.42-2}, it is sufficient to show 
\begin{equation}
			\label{20.07.16.13.42}
\begin{gathered}
\int_{\bR^d}\left|\cF^{-1}\left[e^{t\psi^{\nu_{c_1,N_\nu}}(N_\nu \cdot )}\cF[\varphi] \right](x)\right|dx\leq C_1e^{-C_2t}.
\end{gathered}
\end{equation}
Moreover, $|N_\nu \xi|  \in [N_\nu  c_s^{-1} , N_\nu c_s]$ for all $\xi\in B_{c_s}\setminus B_{c_s^{-1}}$.
Take $c_1 \geq1$ so that $
c_1\in(N_{\nu}c_s,\infty)$. Then by Lemmas \ref{polynomial} and \ref{smoothness}, for any multi-index $\alpha$ with respect to the space variable, we have
\begin{equation*}
\begin{gathered}
|D^{\alpha}(e^{t\psi^{\nu_{c_1,N_\nu}}(N_{\nu}\xi)}\cF[\varphi]( \xi))|^2\leq A_1(1+t^{2|\alpha|})e^{-A_2t|\xi|^2}1_{c_s^{-1}\leq|\xi|\leq c_s},\\
\end{gathered}
\end{equation*}
where $A_1$ and $A_2$ are  positive constants depending only on $d$, $\alpha$, $c_s$, $L(\nu)$, $c_1$, and $N_\nu$.
Recalling $d_0=\lfloor d/4 \rfloor+1$ and applying H\"older's inequality and Plancherel's theorem, we have
\begin{equation*}
\begin{aligned}
&\left(\int_{\bR^d}|\cF^{-1}[e^{t\psi^{\nu_{c_1,N_\nu}}(N_{\nu}\cdot)}\cF[\varphi]](x)|dx\right)^2\\
	&\leq \left(\int_{\bR^d}\frac{1}{(1+|x|^2)^{2d_0}}dx\right)\int_{\bR^d}(1+|x|^2)^{2d_0}|\cF^{-1}[e^{t\psi^{\nu_{c_1,N_\nu}}(N_{\nu}\cdot)}\cF[\varphi]](x)|^2dx\\
	&=N(d)\int_{c_s^{-1}\leq|\xi|\leq c_s} |(1-\Delta)^{d_0}(e^{t\psi^{\nu_{c_1,N_\nu}}(N_{\nu}\cdot)}\cF[\varphi])(\xi)|^2d\xi\\
	&\leq N(d,A_1)(1+t^{2d_0})\int_{c_s^{-1}\leq|\xi|\leq c_s} e^{-A_2t|\xi|^2}d\xi.
\end{aligned}
\end{equation*}
Since
\begin{equation*}
\begin{aligned}
    t^{2d_0}e^{-A_2t|\xi|^2}1_{c_s^{-1}\leq|\xi|^2\leq c_s}&\leq N(d)|\xi|^{-4d_0}e^{-A_2t|\xi|^2/2}1_{c_s^{-1}\leq|\xi|^2\leq c_s}\\
    &\leq N(d,c_s)e^{-A_2t|\xi|^2/2}1_{c_s^{-1}\leq|\xi|^2\leq c_s},
\end{aligned}
\end{equation*}
we finally obtain \eqref{20.07.16.13.42}.

For Proposition \ref{22.07.07.16.05-2}, using Proposition \ref{20.05.03.19.17}, we have
\begin{equation*}
    \cF^{-1}[\psi^{\mu}(r^{-1}\cdot)\cF[\varphi]](x)=a^r_{\mu}\cdot\nabla\varphi(x)+\int_{\bR^d}(\varphi(x+y)-\varphi(x)-y\cdot\nabla\varphi(x)1_{|y|\leq1})\mu^r(dy),
\end{equation*}
where
$$
a_{\mu}^r:=-\frac{1}{r}\int_{r<|y|\leq1}y\mu(dy).
$$
Since $\mu$ is symmetric, it is obvious that $a_{\mu}^r=0$.
Moreover following the argument in \eqref{20.04.10.16.45}, we have
$$
s(r)\int_{\bR^d}|\cF^{-1}[\psi^{\mu}(r^{-1}\cdot)\cF[\varphi]](x)|dx\leq 4s(r)L(\mu^r)\sum_{|\alpha|\leq2}\|D^{\alpha}\varphi\|_{L_1(\bR^d)}=4L(r;s,\mu)\sum_{|\alpha|\leq2}\|D^{\alpha}\varphi\|_{L_1(\bR^d)}.
$$
By \eqref{22.07.09.21.28}, Assumption \ref{22.06.26.18.47} holds. The Proposition is proved.

\vspace{3mm}
\textbf{Proof of Proposition \ref{22.07.07.16.05-3}} \hspace{3mm} 
The additional statement is obvious since
$$
s(r)\int_{\bR^d}(1\wedge|y|^2)\nu_{sym}(r\,dy)
=s(r)\int_{\bR^d}(1\wedge|y|^2)\mu (r\,dy).
$$
Moreover, the main proof is almost identical to that of Proposition \ref{22.07.07.16.05}.
We only remark that 
and Lemmas \ref{lem 20220715 01} and \ref{smoothness-2} are used instead of Lemmas \ref{polynomial} and \ref{smoothness}.

\subsection{Proof of Proposition \ref{properties}}
\label{20.06.01.15.11}
First, we provide the proof of Proposition \ref{properties}.
Recall the space $\tilde{C}^{\infty}(\bR^d)$ of functions whose $L_p$-norms of all derivatives are finite and 
we prove that the operator $(\lambda-\psi^{\mu}(D))^{\alpha/2}$ is bijective on $\tilde{C}^{\infty}(\bR^d)$ for all $\alpha\in\bR$ and $\lambda>0$.
\begin{lem}
\label{20.08.13.16.17}
Let $\mu$ be a symmetric L\'evy measure on $\bR^d$. For $\alpha\in\bR$ and $\lambda>0$, the operator
$$
(\lambda-\psi^{\mu}(D))^{\alpha/2}:\tilde{C}^{\infty}(\bR^d) \to \tilde{C}^{\infty}(\bR^d)
$$
is bijective. In particular, 
$$
\tilde{C}^{\infty}(\bR^d) \subset H_p^{\mu;\gamma}(\bR^d) 
$$
for all $p \in [1,\infty)$ and $\gamma \in \bR$. 
\end{lem}
\begin{proof}
For each $\alpha \in \bR$, there exists an integer $n$ such that
$$
\alpha \in [2n, 2n+2).
$$
Due to  basis properties of the Fourier transform and inverse transform, it is obvious that
\begin{align*}
(\lambda-\psi^{\mu}(D))^{\alpha/2}
&=(\lambda-\psi^{\mu}(D))^{(\alpha-2n)/2}(\lambda-\psi^{\mu}(D))^{n}\\
&=(\lambda-\psi^{\mu}(D))^{(\alpha-2n)/2}(\lambda-\psi^{\mu}(D)) \cdots (\lambda-\psi^{\mu}(D)).
\end{align*}
Thus it suffices to prove the lemma for $\alpha\in(0,2]$. We consider the simplest case $\alpha=2$ first.

\vspace{2mm}
\textbf{(Step 1).} Assume $\alpha=2$ and let $f \in \tilde{C}^{\infty}(\bR^d)$.
By Lemma \ref{20.03.30.13.32} and elementary properties of the Fourier transform, for $x\in\bR^d$
\begin{align}
							\label{2020082620}
\psi^{\mu}(D)f(x)=\int_{\bR^d}(f(x+y)-f(x)-y\cdot\nabla f(x)1_{|y|\leq1})\mu(dy).
\end{align}
The Lebesgue dominated convergence theorem easily implies that for any multi-index $\alpha$, 
$$
D^{\alpha}\psi^{\mu}(D)f(x)=\psi^{\mu}(D)D^{\alpha}f(x).
$$
Thus, applying Taylor's theorem and the Lebesgue dominated convergence theorem again, we have
\begin{align}
							\label{2020082610}
\|D^\alpha \psi^{\mu}(D)f\|_{L_p(\bR^d)}\leq 4L(\mu)\sum_{|\beta|\leq |\alpha|+2}\|D^{\beta} f\|_{L_p(\bR^d)},\quad \forall p\in[1,\infty].
\end{align}
In other words, we conclude that
\begin{align}
								\label{2020082521}
\left(\lambda - \psi^{\mu}(D)\right)f \in\tilde{C}^{\infty}(\bR^d).
\end{align}
By \cite[Theorem 9.8]{sato1999levy}, there exists a L\'evy process $Z$ with triplet $(0,0,\mu)$ such that
\begin{equation}
							\label{2020082530}
\begin{gathered}
\bE[e^{i\xi\cdot Z_t}]=e^{t\psi^{\tilde \mu}(\xi)} \quad \forall t \in [0,\infty),~ \xi \in \bR^d,
\end{gathered}
\end{equation}
By Fubini's theorem and Lemma \ref{20.03.30.13.32},
\begin{align}
							\label{2020082520}
(\lambda-\psi^{\mu}(D))^{-1}f(x)=\int_{0}^{\infty}e^{-\lambda t}\bE[f(x+Z_t)]dt.
\end{align}
Indeed,
\begin{align*}
\int_{0}^{\infty}e^{-\lambda t}\bE[f(x+Z_t)]dt
&=\cF^{-1} \left[\int_{0}^{\infty}e^{-\lambda t}\bE\left[ \cF[f](\xi) e^{i \xi \cdot Z_t} \right] dt \right](x) \\
&=\cF^{-1} \left[\int_{0}^{\infty}e^{-\lambda t} e^{ t \psi^{\mu}(\xi)} dt  \cF[f](\xi) \right] (x)\\
&=\cF^{-1} \left[\left(\lambda -\psi^{\mu}(\xi)\right) \cF[f](\xi) \right](x).
\end{align*}
The right-hand side of \eqref{2020082520} implies that 
\begin{align}
						\label{2020082522}
(\lambda-\psi^{\mu}(D))^{-1}f(x) \in  \tilde{C}^{\infty}(\bR^d).
\end{align}
Finally combining \eqref{2020082521} and \eqref{2020082522}, we complete the proof of the lemma for the case $\alpha=2$.

(\textbf{Step 2}). Assume $\alpha\in(0,2)$, let $f \in \tilde{C}^{\infty}(\bR^d)$, and denote $\phi(\lambda):=\lambda^{\alpha/2}$.
Then $\phi$  is a Bernstein function  and has the  representation (cf. \cite{schbern})
$$
\phi(\lambda)=N(\alpha)\int_0^{\infty}(1-e^{\lambda t})t^{-1-\frac{\alpha}{2}}dt.
$$
Due to this representation and Fubini's theorem, 
\begin{equation*}
\begin{aligned}
(\lambda-\psi^{\mu}(D))^{\alpha/2}f(x)&=
    \cF^{-1}[(\lambda-\psi^{\mu})^{\alpha/2}\cF[f]](x)=\cF^{-1}[\phi(\lambda-\psi^{\mu})\cF[f]](x)\\
    &=N(\alpha)\cF^{-1}\left[\int_{0}^{\infty}(1-e^{-\lambda t}e^{t\psi^{\mu}})\cF[f]t^{-1-\frac{\alpha}{2}}dt\right](x)\\
    &=N(\alpha)\int_0^{\infty}(f(x)-e^{-\lambda t}\bE[f(x+Z_t)])t^{-1-\frac{\alpha}{2}}dt,
\end{aligned}
\end{equation*}
where $Z$ is a L\'evy process satisfying \eqref{2020082530}. 
Note that
\begin{align*}
\frac{d}{dr} e^{-\lambda r}\bE[f(x+Z_r)]
=-\lambda e^{-\lambda r}\bE[f(x+Z_r)]+  e^{-\lambda r}\psi^{\mu}(D)\bE[f(x+Z_r)]
\end{align*}
Thus by the fundamental theorem of calculus and Fubini's theorem,
\begin{align}
								\label{2020082621}
(\lambda-\psi^{\mu}(D))^{\alpha/2}f(x)=N(\alpha)\int_{0}^{\infty}t^{-1-\frac{\alpha}{2}}\int_0^te^{-\lambda r}\bE[(\lambda-\psi^{\mu}(D)) f(x+Z_r)]drdt.
\end{align}
Since $(\lambda-\psi^{\mu}(D))f\in\tilde{C}^{\infty}(\bR^d)$ by Step 1, we conclude that 
$$
(\lambda-\psi^{\mu}(D))^{\alpha/2}f\in\tilde{C}^{\infty}(\bR^d).
$$

Finally, using the above fact and Step 1 again, we have
$$
(\lambda-\psi^{\mu}(D))^{-\alpha/2}f=(\lambda-\psi^{\mu}(D))^{\frac{2-\alpha}{2}}(\lambda-\psi^{\mu}(D))^{-1}f\in\tilde{C}^{\infty}(\bR^d).
$$
The lemma is proved.
\end{proof}

\textbf{Proof of Proposition \ref{properties}}
($i$) Recall \eqref{sch conver} and thus if $\phi_n\to\phi$ in $\cS(\bR^d)$, then for any multi-index $\alpha$, $\|D^{\alpha}\phi_n-D^{\alpha}\phi\|_{L_p(\bR^d)}$ converges to 0.
Following \eqref{2020082610}, we have
$$
\|\phi_n-\phi\|_{H_p^{\mu;\gamma}(\bR^d)}\leq N\sum_{|\alpha|\leq |\gamma|+2}\|D^{\alpha}\phi_n-D^{\alpha}\phi\|_{L_p(\bR^d)}
$$
and obtain the result.

\vspace{3mm}

($ii$) It suffices to prove that $(1-\psi^{\mu}(D))^{\gamma/2}f$ is a tempered distribution. For $\phi\in\cS(\bR^d)$,
$$
|(1-\psi^{\mu}(D))^{\gamma/2}f,\phi)|\leq \|f\|_{L_p(\bR^d)}\|\phi\|_{H_p^{\mu;\gamma}(\bR^d)}.
$$
By ($i$), $(1-\psi^{\mu}(D))^{\gamma/2}f$ is a tempered distribution.

\vspace{3mm}

($iii$) It suffices to prove the completeness. Let $\{f_n\}_{n=1}^{\infty}\subseteq H_p^{\mu;\gamma}(\bR^d)$ be a Cauchy sequence. 
Then by the completeness of the space $L_p(\bR^d)$, there exists a $g\in L_p(\bR^d)$ such that
$$
(1-\psi^{\mu}(D))^{\gamma/2}f_n\to g,\quad \text{in } L_p(\bR^d).
$$
Put $f:=(1-\psi^{\mu}(D))^{-\gamma/2}g$. Then by ($ii$),  $f\in H_p^{\mu;\gamma}(\bR^d)$. 
Therefore, $H_p^{\mu;\gamma}(\bR^d)$ is a Banach space.

\vspace{3mm}

($iv$) Recall that $C_c^{\infty}(\bR^d)$ is dense in $L_p(\bR^d)$ and
$$
C_c^{\infty}(\bR^d)\subseteq \tilde{C}^{\infty}(\bR^d) \subset H_p^{\mu;\gamma}(\bR^d).
$$
First, we show that $\tilde{C}^{\infty}(\bR^d)$ is dense in $H_p^{\mu;\gamma}(\bR^d)$.
For $f\in H_p^{\mu;\gamma}(\bR^d)$, there exists a sequence $g_n \in C_c^\infty(\bR^d)$ such that
$$
\|(1-\psi^{\mu}(D))^{\gamma/2}f-g_n\|_{L_p(\bR^d)}\to 0 ~\text{as}~ n \to \infty.
$$
By Lemma \ref{20.08.13.16.17}, $h_n:=(1-\psi^{\mu}(D))^{-\gamma/2}g_n\in\tilde{C}^{\infty}(\bR^d)$. 
Thus, it is sufficient to show that $C_c^\infty(\bR^d)$ is dense in $\tilde{C}^{\infty}(\bR^d)$ under the norm in $H_p^{\mu;\gamma}(\bR^d)$. 
Let $f \in \tilde{C}^{\infty}(\bR^d)$ and choose a nonnegative $\eta \in C_c^\infty(\bR^d)$ such that $\eta(0)=1$. 
Denote $\eta_n(x) = \eta (x/n)$. Then, due to Leibniz's product rule, it is easy to check that for any multi-index $\alpha$
\begin{align}
									\label{2020082622}
\|D^\alpha (\eta_n f) -D^\alpha f\|_{L_p(\bR^d)} \to 0 ~\text{as}~ n \to \infty.
\end{align}
Recalling nonlocal representations \eqref{2020082620} and \eqref{2020082621}, one can easily check that \eqref{2020082622} implies
$$
\eta_n f \to f \quad \in H_p^{\mu;\gamma}(\bR^d) \qquad \text{as}~n \to \infty.
$$
The remained statements are well-known. For instance, the proof of ($v$)-($ix$) can be found in \cite{farkas2001function} ; 
for ($v$), see Theorem 1.5.10 ;
for ($vi$), see Corollary 2.2.3 ;
for ($vii$), see Theorem 2.2.10 ; 
for ($viii$), see Theorem 2.3.1 ;
for ($ix$), see Theorem 2.4.6.
The proposition is proved.

\subsection{Proof of Proposition \ref{20.04.29.20.24}} In this subsection, we complete the proof of properties of the scaled Besov spaces. 
To apply a well-known theory in functional analysis, we recall an abstract definition related to quasi-Banach space-valued operators.
\begin{defn}
Let $A$ and $B$ be quasi-Banach spaces. We say that $B$ is retract of $A$ if there exist linear transformation $I:B\to A$ and $P:A\to B$ such that $PI$ is an identity operator in B.
\end{defn}
\begin{lem}
\label{20.06.28.16.49}
Let $p\in[1,\infty]$, $q\in(0,\infty]$ and $\gamma\in\bR$. Then the scaled Besov space $B_{p,q}^{s,\varphi;\gamma}(\bR^d)$ is a retract of $l_q^{s;\gamma}(L_p(\bR^d))$.
\end{lem}
\begin{proof}
Consider two maps ; 
\begin{equation*}
f\in B_{p,q}^{s,\varphi;\gamma}(\bR^d) \mapsto I(f):=(f\ast\varphi_0,f\ast\varphi_1,\cdots)
\end{equation*}
and
\begin{equation*}
\boldsymbol{f}=(f_0,f_1,\cdots)\in l_q^{s;\gamma}(L_p(\bR^d)) \mapsto P(\boldsymbol{f}):=\sum_{j=0}^{\infty}\sum_{l=-1}^1f_{j}\ast\varphi_{j+l},
\end{equation*}
where $\varphi_j=f_j=0$ if $j<0$. 
It is obvious that $PI$ is an identity operator in $B_{p,q}^{s,\varphi;\gamma}(\bR^d)$,
$$
\|f\|_{B_{p,q}^{s,\varphi;\gamma}(\bR^d)}=\|I(f)\|_{l_q^{s;\gamma}(L_p(\bR^d))},
$$
and $I$ is a linear transformation from $B_{p,q}^{s,\varphi;\gamma}(\bR^d)$ to $l_q^{s;\gamma}(L_p(\bR^d))$. 
By \eqref{almost orthogonality}, for $r\in \bN\cup\{0\}$,
\begin{equation*}
 \|P(\boldsymbol{f})\ast \varphi_r\|_{L_p(\bR^d)}\leq N \sum_{j=r-2}^{r+2}\|f_j\|_{L_p(\bR^d)},   
\end{equation*}
where $N$ depends only on $\varphi$ and $\varphi_0$. 
Moreover, by \eqref{20.04.28.13.46},
\begin{equation}
\label{20.08.19.10.31}
    s(c_s^{-r})^{-\frac{\gamma}{2}}\|P(\boldsymbol{f})\ast\varphi_r\|_{L_p(\bR^d)}\leq N\sum_{j=r-2}^{r+2}s(c_s^{-j})^{-\frac{\gamma}{2}}\|f_j\|_{L_p(\bR^d)},
\end{equation}
where $N=N(\|\varphi\|_{L_1(\bR^d)},\|\varphi_0\|_{L_1(\bR^d)},\gamma,s_U(c_s^2),s_U(c_s),s_L(c_s^2),s_L(c_s))$. By \eqref{20.08.19.10.31},
\begin{equation*}
    \begin{aligned}
    \|P(\boldsymbol{f})\|_{B_{p,q}^{s,\varphi;\gamma}(\bR^d)} \leq N\|\boldsymbol{f}\|_{l_q^{s;\gamma}(L_p(\bR^d))},\quad q\in(0,\infty],
    \end{aligned}
\end{equation*}
where $N$ is independent of $\boldsymbol{f}$. The above inequality implies that $P$ is a linear transformation from $l_q^{s;\gamma}(L_p(\bR^d))$ to $B_{p,q}^{s,\varphi;\gamma}(\bR^d)$. The proposition is proved.
\end{proof}

The following result yields that all scaling functions have a polynomial growth.
\begin{prop}
						\label{poly}
Let $\boldsymbol{s}=(s,s_L,s_U)$ be a scaling triple and
$$
\theta_0:=\log_{c_s}(s_L(c_s)),\quad\theta_1:=\log_{c_s}(s_U(c_s)).
$$
Then $0<\theta_0\leq\theta_1$ and there exists a positive constant $C_0$ such that  for all $0<r\leq 1$,
\begin{align}
							\label{2020082402}
C_0^{-1}r^{\theta_1}\leq s(r)\leq C_0r^{\theta_0},
\end{align}
where $C_0$ depends only on $s(1),s_L(1),s_U(1), s_L(c_s^{-1})$, and $s_U(c_s)$.
\end{prop}
\begin{proof} 
Recalling the definition of $c_s$ \eqref{2020082401}, we have
$$
1<s_L(c_s)\leq \frac{s(1)}{s(c_s^{-1})}\leq s_U(c_s).
$$
Thus, the first assertion $0<\theta_0 \leq \theta_1$ is easily proved.
To prove \eqref{2020082402}, observe that for each $j\in\bN$, 
\begin{equation*}
    \frac{1}{s(c_s^{-j})}=\frac{1}{s(1)}\prod_{k=0}^{j-1}\frac{s(c_s^{-j+k+1})}{s(c_s^{-j+k})}.
\end{equation*}
Applying \eqref{20.04.28.13.46} to the above products, we have
\begin{align}
\label{scaling1}
s(1)\left(s_U(c_s)\right)^{-j}\leq  s(c_s^{-j})\leq s(1)\left(s_L(c_s)\right)^{-j}.
\end{align}
Note that
$$
(0,1]=\bigcup_{j=1}^{\infty}(c_s^{-j},c_s^{-j+1}].
$$
Thus for $r\in(0,1]$, there exists a unique $j\in\bN$ such that $r\in(c_s^{-j},c_s^{-j+1}]$. By \eqref{20.04.28.13.46} and  \eqref{scaling1},
\begin{equation*}
\begin{aligned}
    s(r)&\leq s_U(rc_s^j)s(c_s^{-j})\leq s_U(c_s)s(1)\left(s_L(c_s)\right)^{-j}= s_U(c_s)s(1)\left (c_s^{-j}\right)^{\log_{c_s}(s_L(c_s))}\\
    &\leq s_U(c_s)s(1)r^{\log_{c_s}(s_L(c_s))}=s_U(c_s)s(1)r^{\theta_0}
\end{aligned}
\end{equation*}
and
\begin{equation*}
\begin{aligned}
    s(r)&\geq s_L(rc_s^j)s(c_s^{-j})\geq s_L(1)s(1)\left(s_U(c_s)\right)^{-j}= s_L(1)s(1)s_U(c_s)^{-1}\left (c_s^{-j+1}\right)^{\log_{c_s}(s_U(c_s))}\\
    &\geq s_L(1)s(1)s_U(c_s)^{-1}r^{\log_{c_s}(s_U(c_s))}=:s_L(1)s(1)s_U(c_s)^{-1}r^{\theta_1}.
\end{aligned}
\end{equation*}
The proposition is proved.
\end{proof}

\textbf{Proof of Proposition \ref{20.04.29.20.24}}
$(i)$ By Proposition \ref{poly}, for $j\geq0$,
\begin{equation}
\label{20.08.19.10.57}
    C_0^{-1}2^{jm_s\theta_0}=C_0^{-1}c_s^{j\theta_0}\leq s(c_s^{-j})^{-1}\leq C_0 c_s^{j\theta_1}=C_0 2^{jm_s\theta_1},
\end{equation}
where $C_0=C_0(s(1),s_L(1),s_U(1),s_L(c_s^{-1}),s_U(c_s))$. For $\zeta\in\Phi_2(\bR^d)$, we have
\begin{equation}
\label{20.08.19.10.55}
\cF[f]\cF[\varphi_j]=
\begin{cases}
\cF[f](\cF[\zeta_0]+\cF[\zeta_1]+\cdots+\cF[\zeta_{m_s}])\cF[\varphi_0], & \text{if } j=0,\\
\cF[f](\cF[\zeta_{m_s(j-1)}]+\cdots+\cF[\zeta_{m_s(j+1)}])\cF[\varphi_j], & \text{if } j\geq1.
\end{cases}
\end{equation}
Similarly, for $l\in\bN\cup\{0\}$, we have
\begin{equation}
\label{20.08.19.10.56}
\begin{gathered}
\cF[f]\cF[\zeta_j]=
\begin{cases}
\cF[f]\{\cF[\varphi_{l-1}]+\cF[\varphi_{l}]+\cF[\varphi_{l+1}]\}\cF[\zeta_j],&\text{if } j=m_sl,\\
\cF[f]\{\cF[\varphi_{l}]+\cF[\varphi_{l+1}]\}\cF[\zeta_j],&\text{if } m_sl+1\leq j\leq m_s(l+1)-1.
\end{cases}
\end{gathered}
\end{equation}
Recalling Definition \ref{20.08.20.17.26} and using \eqref{20.08.19.10.57}, \eqref{20.08.19.10.55} and \eqref{20.08.19.10.56}, we obtain the result.

\vspace{3mm}

$(ii)$ It suffices to prove the completeness. Let $\{f_n\}_{n=1}^{\infty}\subseteq B_{p,q}^{s,\varphi;\gamma}(\bR^d)$ be a Cauchy sequence.
Then, there exists a $\boldsymbol{f}:=(f^1,f^2,\cdots)\in l_q^{s;\gamma}(L_p(\bR^d))$ such that
$$
I(f_n)=(f_n\ast\varphi_0,f_n\ast\varphi_1,\cdots)\to\boldsymbol{f},\quad \text{in } l_q^{s;\gamma}(L_p(\bR^d)).
$$
Since $f_n\ast\varphi_j\to f^j$ in $L_p(\bR^d)$ as $n\to\infty$, we observe that for $\phi\in\cS(\bR^d)$,
\begin{equation}
\label{20.08.17.11.38}
\begin{aligned}
(f^j,\phi)&=\lim_{n\to\infty}(f_n\ast\varphi_j,\phi)=\lim_{n\to\infty}(f_n\ast\varphi_j\ast(\varphi_{j-1}+\varphi_j+\varphi_{j+1}),\phi)\\
&=(f^j\ast(\varphi_{j-1}+\varphi_j+\varphi_{j+1}),\phi),
\end{aligned}
\end{equation}
where $\varphi_j=0$ if $j<0$. Let
$$
f:=P(\boldsymbol{f})=\sum_{j=0}^{\infty}\sum_{l=-1}^1f^j\ast\varphi_{j+l}=\sum_{j=0}^{\infty}f^j.
$$
By Lemma \ref{20.06.28.16.49},
$$
\|f\|_{B_{p,q}^{s,\varphi;\gamma}(\bR^d)}\leq N\|\boldsymbol{f}\|_{l_q^{s;\gamma}(L_p(\bR^d))}<\infty,
$$
and this also implies $f_n$ converges to $f$ in $B_{p,q}^{s,\varphi;\gamma}(\bR^d)$.

\vspace{3mm}

($iii$)  
It is well-known that the Schwartz class is dense in the classical Besov spaces (e.g. \cite[Theorem 2.3.3]{triebel1983theory}), i.e.
$$
\cS(\bR^d)\subseteq B_{p,q}^{\gamma'}(\bR^d)
$$
for all $\gamma'\in\bR$.
Thus by $(i)$, 
$$
\cS(\bR^d)\subseteq B_{p,q}^{s,\varphi;\gamma}(\bR^d).
$$

Next, we claim that 
$$
\tilde C^\infty(\bR^d)\subseteq B_{p,q}^{\gamma}(\bR^d),\quad \forall \gamma\in\bR.
$$
First, assume $\gamma\in(0,1)$. We use an equivalent norm of the classical Besov spaces. 
By \cite[Remark 2 in p.113]{triebel1983theory},
\begin{align}
							\label{2020082701}
\|f\|_{B_{p,q}^{\gamma}(\bR^d)}\approx \|f\|_{L_p(\bR^d)}+\left(\int_{\bR^d}|h|^{-d-\gamma q}\|\Delta_h^2f\|_{L_p(\bR^d)}^qdh\right)^{1/q},
\end{align}
where $\Delta_h^2f(x):=f(x+2h)-2f(x+h)+f(x)$. For $f\in\tilde{C}^{\infty}(\bR^d)$,
$$
\|\Delta_h^2f\|_{L_p(\bR^d)} \leq |h|^2\sum_{|\alpha|=2}\|D^{\alpha}f\|_{L_p(\bR^d)}1_{|h|\leq 1}+3\|f\|_{L_p(\bR^d)}1_{|h|>1}\leq N(|h|^21_{|h|\leq 1}+1_{|h|>1}),
$$
where $N$ depends only on $\|f\|_{L_p(\bR^d)}$ and $\sum_{|\alpha|=2}\|D^{\alpha}f\|_{L_p(\bR^d)}$.
Thus, the right-hand side of \eqref{2020082701} is finite and it implies that
\begin{align}
							\label{2020082710}
\tilde{C}^{\infty}(\bR^d)\subseteq B_{p,q}^{\gamma}(\bR^d),\quad \forall \gamma\in(0,1).
\end{align}
For $n\in\bN$ and $f\in\tilde{C}^{\infty}(\bR^d)$, we have
$$
(1-\Delta)^{n/2}f,(1-\Delta)^{-n/2}f\in \tilde{C}^{\infty}(\bR^d).
$$
Thus, for each $\gamma \in \bR$, finding $n\in\bZ$ such that $\gamma = n+\gamma'$ with $\gamma' \in (0,1)$ and recalling \eqref{2020082710}, we have
\begin{align*}
\tilde{C}^{\infty}(\bR^d)\subseteq B_{p,q}^{\gamma}(\bR^d).
\end{align*}
Due to (i), we finally obtain
$$
\tilde{C}^{\infty}(\bR^d)\subseteq B_{p,q}^{m_s\theta_1\gamma}(\bR^d)\subseteq B_{p,q}^{s,\varphi;\gamma}(\bR^d),\quad \forall\gamma\in\bR.
$$
To prove the denseness of $C_c^\infty(\bR^d)$, choose a Littlewood-Paley function $\varphi\in  \Phi_{c_s}(\bR^d)$ and define
$$
\eta_k(x):=\sum_{j=0}^k\varphi_j(x)\qquad \forall k \geq 2.
$$
Then, obviously $\eta_k\in\cS(\bR^d)$.
Observe that for all $\xi \in \bR^d$,
\begin{align}
								\label{2020082720}
\cF[\varphi_j](\xi)=\cF[\varphi_j](\xi) \cF[\eta_k](\xi) \qquad \forall j =\{0,1,\ldots, k-1\}
\end{align}
and
\begin{align}
								\label{2020082721}
\cF[\varphi_j](\xi)\cF[\eta_k](\xi) =0  \qquad \forall j =\{k+2,k+3,\ldots\}.
\end{align}
Thus, for each $f\in B_{p,q}^{s,\varphi;\gamma}(\bR^d)$, applying \eqref{2020082720} we have
\begin{equation*}
    \|f-\eta_k\ast f\|_{B_{p,q}^{s,\varphi;\gamma}(\bR^d)}^q\leq N\sum_{j=k}^{\infty}s(c_s^{-j})^{-\frac{q\gamma}{2}}\|f\ast\varphi_j\|_{L_p(\bR^d)}^q,
\end{equation*}
and $\|f-\eta_k\ast f\|_{B_{p,q}^{s,\varphi;\gamma}(\bR^d)}$ goes to zero as $k \to \infty$,
where $N$ is independent of $k$. 
Moreover, it is obvious that $\eta_k\ast f \in B_{p,q}^{\gamma'}(\bR^d)$ for all  $\gamma' \in \bR$ and
it is well-known that $C_c^\infty(\bR^d)$ is dense in the classical Besov spaces $B_{p,q}^{\gamma'}(\bR^d)$ for all $\gamma' \in \bR$.
Therefore, there exists a sequence $f_n \in C_c^\infty(\bR^d)$ such that
$$
f_n \to \eta_k\ast f  \quad \text{in} \quad  B_{p,q}^{m_s\theta_1\gamma}(\bR^d).
$$
Finally applying (i), we conclude that $C_c^\infty(\bR^d)$ is dense in $B_{p,q}^{s,\varphi;\gamma}(\bR^d)$.

\vspace{3mm}

$(iv)$ It is an easy consequence of the almost orthogonal properties of Littlewood-Paley functions (cf. Remark \ref{20.04.29.20.00}).

\vspace{3mm}

$(v)$ For $j\geq1$, $-\infty<\gamma_1\leq\gamma_2<\infty$
$$
s(c_s^{-j})^{-\frac{\gamma_1}{2}}\leq s(c_s^{-j})^{-\frac{\gamma_2}{2}}.
$$
This implies the first inequality in \eqref{20.08.20.10.49}. For $q_1\in(0,\infty)$,
\begin{equation}
\label{20.08.20.13.20}
\begin{aligned}
    \sup_{j\in\bN}s(c_s^{-j})^{-\frac{q_1\gamma}{2}}\|f\ast\varphi_j\|_{L_p(\bR^d)}^{q_1}\leq \sum_{j=1}^{\infty}s(c_s^{-j})^{-\frac{q_1\gamma}{2}}\|f\ast\varphi_j\|_{L_p(\bR^d)}^{q_1}.
\end{aligned}
\end{equation}
For $0<q_1\leq q_2<\infty$,
\begin{equation}
\label{20.08.20.13.21}
    \left(\sum_{j=1}^{\infty}s(c_s^{-j})^{-\frac{q_2\gamma }{2}}\|f\ast\varphi_j\|_{L_p(\bR^d)}^{q_2}\right)^{\frac{q_1}{q_2}}\leq \sum_{j=1}^{\infty}s(c_s^{-j})^{-\frac{q_1\gamma}{2}}\|f\ast\varphi_j\|_{L_p(\bR^d)}^{q_1}.
\end{equation}
By \eqref{20.08.20.13.20} and \eqref{20.08.20.13.21}, we obtain the second inequality in \eqref{20.08.20.10.49}.

\vspace{3mm}

$(vi)$ By $(i)$ and properties of classical Besov space \cite[Theorem 2.6.1]{triebel1978interpolation}, 
$$
B_{p',q'}^{-m_s\theta_0\gamma}(\bR^d)\subseteq\left(B_{p,q}^{s,\varphi;\gamma}(\bR^d)\right)^*\subseteq B_{p',q'}^{-m_s\theta_1\gamma}(\bR^d)\subseteq \cS'(\bR^d),
$$
where $\left(B_{p,q}^{s,\varphi;\gamma}(\bR^d)\right)^*$ is a topological dual space of $B_{p,q}^{s,\varphi;\gamma}(\bR^d)$.

\vspace{3mm}

($vii$) 
Note that
$$
\left(s(c_s^{-j})^{-\frac{q_0\gamma_0}{2}}\right)^{\frac{q(1-\theta)}{q_0}}\left(s(c_s^{-j})^{-\frac{q_1\gamma_1}{2}}\right)^{\frac{q\theta}{q_1}}=s(c_s^{-j})^{-\frac{q\gamma}{2}}.
$$
Then by Theorem \cite[Theorem 5.5.3]{bergh2012interpolation}, we obtain the result.

\vspace{3mm}

($viii$) By ($vi$) and \cite[Theorem 6.4.2]{bergh2012interpolation}, 
$$
[B_{p,q_0}^{s;\gamma_0}(\bR^d),B_{p,q_1}^{s;\gamma_1}(\bR^d)]_{\theta}
$$
is a retract of $l_q^{s;
\gamma}(L_p(\bR^d))$ with the mappings $I$ and $P$ defined in Lemma \ref{20.06.28.16.49}. The proposition is proved.


\end{document}